\documentclass[10pt, reqno]{amsart}

\usepackage{amsmath, amssymb, amsthm}
\usepackage{bbm}
\usepackage{mathtools}
\usepackage{enumitem}
\usepackage[dvipsnames]{xcolor}
\usepackage[linktocpage]{hyperref}
\usepackage[hmarginratio={1:1}]{geometry}
\usepackage[overload]{textcase}
\usepackage[scr=rsfs]{mathalfa}
\usepackage{setspace}
\hypersetup{
	colorlinks = true,
	linkcolor=MidnightBlue,
	citecolor=BrickRed,
}

\numberwithin{equation}{section}

\newcommand{\NN}{\mathbb{N}}

\newcommand{\HH}{\mathbb{H}}
\newcommand{\RR}{\mathbb{R}}

\newcommand{\Om}{\Omega}

\newcommand{\abs}[1]{\left\vert #1 \right\vert}
\newcommand{\set}[1]{\left\{ #1 \right\}}
\newcommand{\norm}[1]{\left\Vert #1 \right\Vert}

\newcommand{\inner}[1]{\left\langle#1\right\rangle}
\newcommand{\pinner}[1]{\left(#1\right)}

\newcommand{\D}{\mathcal{D}}

\renewcommand{\tilde}{\widetilde}

\renewcommand{\complement}{\mathtt{C}}
\DeclareMathOperator{\supp}{supp}

\renewcommand{\d}{\mathrm{d}}
\renewcommand{\epsilon}{\varepsilon}

\theoremstyle{definition}
\newtheorem{Def}{Definition}[section]

\theoremstyle{remark}
\newtheorem{rem}{Remark}[section]

\theoremstyle{plain}
\newtheorem{thm}{Theorem}
\newtheorem{prop}{Proposition}[section]
\newtheorem{lem}[prop]{Lemma}
\newtheorem{cor}[prop]{Corollary}

\title[Critical Weighted $p$-Laplace Equations]{A Struwe-type Decomposition Result for Weighted Critical $p$-Laplace Equations}
\author{Edward Chernysh}
\date{\today}
\thanks{This work constitutes part of the author's PhD thesis at McGill University, which was completed with the support of the Canada Graduate Scholarships -- Doctoral (CGS D) administered by the Natural Sciences and Engineering Research Council of Canada (NSERC)}

\begin{document}
	
	\begin{abstract}
		We establish Struwe-type decompositions of Palais-Smale sequences for a class of critical $p$-Laplace equations of the Caffarelli-Kohn-Nirenberg type in a bounded domain $\Omega\subset\RR^n$, $n\ge2$, containing the origin. In doing so, we highlight important differences introduced by the weights and require new rescaling laws to account for this new framework.
	\end{abstract}
	
	\maketitle
	
	\section{Introduction}
	Let $\Om \subset \RR^n$ be a bounded domain containing the origin and fix a dimension $n \geq 2$. For any $1 < p < n$, we let $q$ be the critical Caffarelli-Kohn-Nirenberg (henceforth abbreviated CKN) exponent given by the rule
	\begin{equation}
		q := \frac{np}{n-p(1+a-b)},
	\end{equation}
	whilst subject to the constraints
	\begin{equation}\label{eq:conditions}
		a < \frac{n-p}{p} \quad \text{and} \quad a \leq b < a+1.
	\end{equation}
	In particular, condition \eqref{eq:conditions} implies that $\max(ap, qb) < n$ and $p < q \le p^\ast$ where
	\[
	p^\ast := \frac{np}{n-p}
	\]
	is the critical Sobolev constant. Therefore, the weight functions ${x} \mapsto \abs{x}^{-ap}$ and $x \mapsto\abs{x}^{-bq}$ both belong to $L^1_{\text{loc}}(\RR^n)$. 
	Let us also define a homogeneity exponent:
	\begin{equation}\label{eq:exponent}
		\gamma := \frac{n-bq}{q} = \frac{n}{q} - b = \frac{n-p(1+a-b)}{p} - b = \frac{n-p(1+a)}{p}.
	\end{equation}
	In this article, we investigate the loss of compactness for Palais-Smale sequences (oft abbreviated $(PS)$-sequences) associated to the critical problem
	\begin{equation}\label{eq:probOm}\tag{A}
		\begin{cases}
			-\operatorname{div}\left(\abs{x}^{-ap}\abs{\nabla u}^{p-2}\nabla u\right) = \abs{x}^{-bq}\abs{u}^{q-2}u &\text{in }\Om\\
			u\in \D_a^{1,p}\left(\Om, 0\right)
		\end{cases}
	\end{equation}
	where $\D_a^{1,p}(\Omega, 0)$ denotes a homogeneous Sobolev space, with singularity at the origin, obtained via completion of $C_c^\infty(\Omega)$ with respect to a weighted norm. In fact, given an open subset $U$ of $\RR^n$ and singularity $x_0\in \RR^n$, we define
	\[
	\D_a^{1,p}(U, x_0) = \D^{1,p}(U, \abs{x-x_0}^{-ap})
	\]
	(simply written $\D^{1,p}(U)$ if $a=0$) as the completion of $C_c^\infty(U)$ under the norm
	\begin{equation}\label{xnorm}
		\norm{\cdot}_{\D_a^{1,p}(U, x_0)} := \left( \int_U \abs{\nabla (\cdot)}^p\abs{x-x_0}^{-ap}\d{x}\right)^{1/p}.
	\end{equation}
	Problem \eqref{eq:probOm} is the Euler-Lagrange problem for the Caffarelli-Kohn-Nirenberg (CKN)-inequality (see \S\ref{section:historical context and classification}) in \(\Omega\), and may be described via the energy functional $\phi : \D_a^{1,p}(\Om, 0) \to \RR$, given by
	\[
	\phi(u) := \int_\Omega \abs{\nabla u}^p\abs{x}^{-ap}\d{x} - \int_\Omega\abs{u}^q\abs{x}^{-bq}\d{x}.
	\]
	Indeed,  $u\in \D_a^{1,p}(\Om, 0)$ solves problem \eqref{eq:probOm} provided
	\[	\inner{\phi^\prime(u), h} = 0 	\qquad\forall h\in \D_a^{1,p}(\Om, 0). \]
	Here, with $p^\prime:=p/(p-1)$ and $\D^{-1,p^\prime}_a(U, x_0)$ being the topological dual of $\D_a^{1,p}(U, x_0)$, the map $\phi^\prime : \D_a^{1,p}(\Om, 0)\to\D_a^{-1,p^\prime}(\Om, 0)$ denotes the (Fr\'echet) derivative of $\phi$.
	
	Formally, in this notation, a $(PS)$-sequence for \eqref{eq:probOm} is a sequence $(u_\alpha)$ in $\D_a^{1,p}(\Om, 0)$ of bounded energy (meaning $\phi(u_\alpha)$ is uniformly bounded in $\alpha$) such that $\phi^\prime(u_\alpha)\to 0$ strongly in $\D_a^{-1, p^\prime}(\Om, 0)$ as $\alpha\to\infty$. The purpose of our paper is to establish a Struwe-type decomposition result for the weighted problem \eqref{eq:probOm}. 

	Loosely speaking, these decompositions provide an asymptotic expansion of a given $(PS)$-sequence, within the energy space, in terms of perturbed and rescaled solutions to some ``limiting'' partial differential equations. An important distinction in the general weighted case (i.e. $(a,b) \ne (0,0)$) is the occurrence of this next weighted limiting problem:
	\begin{equation}\label{eq:probLim}\tag{B}
		\begin{cases}
			-\operatorname{div}\left(\abs{x}^{-ap}\abs{\nabla u}^{p-2}\nabla u\right) = \abs{x}^{-bq}\abs{u}^{q-2}u &\text{in }\RR^n\\
			u\in \D_a^{1,p}\left(\RR^n, 0\right)
		\end{cases}
	\end{equation}
	whose associated energy functional $\phi_{0, \infty} : \D_a^{1,p}(\RR^n, 0) \to \RR$ is given by
	\[
	\phi_{0, \infty}(u) := \int_{\RR^n} \abs{\nabla u}^p\abs{x}^{-ap}\d{x} - \int_{\RR^n}\abs{u}^q\abs{x}^{-bq}\d{x}.
	\]

	We distinguish two cases:  \(a\ne b\) and $a =b$. The case \(a = b \ne 0\) is the most challenging and features a new blow-up configuration, which occurs only in this case. However, we shall begin by presenting the case $a \ne b$, which also deviates from the unweighted model case. Formally, when \(a\ne b\), the following theorem holds:
	
	\begin{thm}\label{thm:main a<b}	
		Let $\Om \subset \RR^n$ be a bounded domain containing the origin and let $(u_\alpha)$ be a Palais-Smale sequence for \eqref{eq:probOm}. Assume $a \neq b$ and let $\gamma > 0$ be the homogeneity exponent from \eqref{eq:exponent}. Then, there exists a subsequence $(u_\beta)$ of $(u_\alpha)$ along with
		\begin{enumerate}[label=(\arabic*)]
			\item a solution $v_0$ of \eqref{eq:probOm};
			\item finitely many non-trivial functions $w_1,\dots, w_k$ for the weighted limiting problem \eqref{eq:probLim} and, for each $j=1,\dots, k$, a sequence $(\lambda_\beta^{(j)})$ in $(0,\infty$) such that
			
			\[
			\lambda_\beta^{(j)} \to 0, \quad \text{as } \beta \to \infty.
			\]
		\end{enumerate}
		Furthermore, as $\beta \to \infty$, we obtain
		\begin{enumerate}[label=(\roman*)]
			\item Decomposition in the Energy Space:
			\begin{equation}\label{eq:decomposition in the energy space a<b}
				u_\beta - v_0 - \sum_{j=1}^k \left( \lambda_\beta^{(j)} \right)^{-\gamma} w_j\left( \frac{\cdot}{\lambda_\beta^{(j)} }\right)  \to 0\quad\text{ in } \D^{1,p}_a(\RR^n, 0)
			\end{equation}
			\item Norm Decomposition:
			\[
			\norm{u_\beta}_{\D^{1,p}_a(\Om, 0)}^p \to \norm{v_0}_{\D^{1,p}_a(\Om, 0)}^p + \sum_{j=1}^k \norm{w_j}_{\D_a^{1,p}(\RR^n, 0)}^p
			\]
			\item Energy Decomposition:
			\[
			\phi(u_\beta) \to \phi(v_0) + \sum_{j=1}^k \phi_{0,\infty}(w_j)
			\]
		\end{enumerate}
	\end{thm}
	
	It should be noted that whenever our $(PS)$-sequence \((u_\alpha)\) is non-negative, the functions \(v_0\), \(w_j\) are themselves non-negative.
	
	The first result of this type was established by Struwe \cite{struwe1984global} in the case $p=2$ without weights (i.e. when $a=b=0$). In later years, assuming $a=b=0$, Mercuri-Willem \cite{mercuri-willem} showed that the decomposition obtained by Struwe \cite{struwe1984global} holds for non-negative $(PS)$-sequences and all $p \in (1,n)$.
	
	We remark that in the conclusions of Theorem \ref{thm:main a<b}, all ``bubbles'' (i.e. the rescaled $w_j$'s in \eqref{eq:decomposition in the energy space a<b}) concentrate at the origin. This stands in sharp contrast to the model case $a=b=0$, considered by Struwe \cite{struwe1984global} and Mercuri-Willem \cite{mercuri-willem}, wherein the bubbles may concentrate at any point of $\overline{\Om}$. Consequently, when $a<b$, one does not need to impose regularity conditions on $\partial \Om$.

	Deviating strongly from the classical model case, in the limit case $a=b \ne 0$, a $(PS)$-sequence for \eqref{eq:probOm} may decompose (at the energy level) in terms of 3 potential limiting problems: the weighted problem \eqref{eq:probLim} and problems \eqref{eq:probLimWeightless}, \eqref{eq:probLimWeightlessHalf} which we state after our next main result.

	Before stating the theorem in the case $a=b$, we define a key transform $\tau^-$ in terms of a fixed function $\eta \in C_c^\infty(\RR^n; [0,1])$ supported in \(B(0,1)\) and chosen such that \(\eta \equiv  1\) on \(B(0,1/2)\). Then, given $y\in \RR^n$ and $\lambda > 0$, we define a transformation $\tau^-_{y, \alpha}$ according to the rule
	\[
	\tau^-_{y,\lambda}h(z) := 
	\begin{cases}
		\displaystyle \lambda^{-\gamma}h\left(\frac{z-y}{\lambda}\right) &\text{if } a=b=0\\
		\displaystyle\left(\frac{\lambda}{\abs{y}}\right)^{-b} \lambda^{-\gamma}h\left(\frac{z-y}{\lambda}\right)\eta\left(\frac{2(z-y)}{\abs{y}}\right) &\text{if } a=b \ne 0\,.
	\end{cases}
	\]
	We point out the presence of the factor \(\left(\frac{\lambda}{\abs{y}}\right)^{-b}\), a striking difference from the classical unweighted case $a=b=0$. This difference is discussed in more details after the statement of the theorem.
	
	\begin{thm}\label{thm:main a=b}
		Suppose $\Om\subset \RR^n$ is a bounded $C^1$-domain containing the origin and let $(u_\alpha)$ be a Palais-Smale sequence for problem \eqref{eq:probOm}. Assume $a = b$ (so that \(q=p^\ast\)) and let $\gamma > 0$ be the homogeneity exponent from \eqref{eq:exponent}. Then, there exists a subsequence $(u_\beta)$ of $(u_\alpha)$ along with
		\begin{enumerate}
			\item A solution $v_0$ of \eqref{eq:probOm}.
			\item Finitely many non-trivial functions $w_1, \dots, w_k$ for the weighted limiting problem \eqref{eq:probLim} and, for each $j=1, \dots, k$, a sequence $\left(\mu_\beta^{(j)}\right)$ in $(0, \infty)$ such that
			\(
			\mu_\beta^{(j)} \to 0
			\)
			as $\beta\to \infty$.
			\item Finitely many non-trivial solutions $v_1, \dots, v_l$ for an unweighted limiting problem \eqref{eq:probLimWeightless} in $\RR^n$ (see below) and, for each $j=1, \dots, l$, sequences $\left(\lambda_\beta^{(j)}\right)_\beta$ in $(0, \infty)$ and $\left(y_\beta^{(j)}\right)_\beta$ in $\Om$ satisfying, as $\beta\to \infty$,
			\[
			\lambda_\beta^{(j)} \to 0,\qquad \frac{\abs{y_\beta^{(j)}}}{\lambda_\beta^{(j)}} \to \infty \qquad \frac{\operatorname{dist}(y_\alpha^{(j)}, \partial \Om)}{\abs{\lambda^{(j)}_\alpha}}\to \infty
			\]
			\item Finitely many non-trivial solutions $v^\HH_1, \dots, v^\HH_{l^\HH}$ for an unweighted limiting problem \eqref{eq:probLimWeightlessHalf} in the half-space $\HH$ (see below) and, for each $j=1, \dots, l^\HH$, sequences $(\lambda_\beta^{\HH,(j)})_\beta $ in $ (0, \infty)$ and $(y_\beta^{\HH,(j)})_\beta $ in $ \Om$ satisfying, as $\beta\to\infty$,
			\[
			\lambda_\beta^{\HH,(j)} \to 0,\qquad \frac{\abs{y_\beta^{\HH,(j)}}}{\lambda_\beta^{\HH,(j)}}\to\infty,
			\qquad
			\frac{\operatorname{dist}(y_\beta^{\HH,(j)}, \partial \Om)}{\abs{\lambda^{\HH,(j)}_\beta}} \text{  is bounded.}
			\]
		\end{enumerate}
		These are such that the following decompositions are true as $\beta \to \infty$
		\begin{enumerate}[label=(\roman*)]
			\item Decomposition in the Energy Space:
			\[
			u_\beta 
			- v_0 
			- \sum_{j=1}^k \left(\mu_\beta^{(j)}\right)^{-\gamma}w_j\left(\frac{z}{\mu_\beta^{(j)}}\right)
			- \sum_{j=1}^l\left[\tau^-_{\lambda_\beta^{(j)}, y_\beta^{(j)}}v_j \right]
			- \sum_{j=1}^{l^\HH}\left[\tau^-_{\lambda_\beta^{\HH,(j)}, y_\beta^{\HH,(j)}}V^\HH_j \right]
			\to 0
			\]
			in $\D_a^{1,p}(\RR^n,0)$. Here, $V_j^\HH$ are translated rotations of $v_j^\HH$.
			\item Norm Decomposition:
			\[
			\norm{u_\beta}^p_{\D^{1,p}(\Om,0)} 
			\to
			\norm{v_0}_{\D_a^{1,p}(\Om,0)}
			+ \sum_{j=1}^k\norm{w_j}_{\D_a^{1,p}(\RR^n,0)}
			+ \sum_{j=1}^l\norm{v_j}_{\D^{1,p}(\RR^n)}
			+ \sum_{j=1}^{l^\HH}\norm{v_j^\HH}_{\D^{1,p}(\HH)}
			\]
			\item Energy Decomposition:
			\[
			\phi(u_\beta) 
			\to 
			\phi(v_0) 
			+ \sum_{j=1}^k \phi_{0,\infty}(w_j)
			+ \sum_{j=1}^l \phi_\infty(v_j)
			+ \sum_{j=1}^{l^\HH} \phi_\infty(v_j^\HH)
			\]
		\end{enumerate}
	\end{thm}
	\begin{rem}
		It should be noted that whenever our $(PS)$-sequence \((u_\alpha)\) is non-negative the functions \(v_0\), \(w_j\), \(v_j\), \(v_j^\HH\) are themselves non-negative. Therefore, in this case, thanks to a non-existence result of Mercuri-Willem \cite{mercuri-willem}, we may rule out the possibility of solutions to the half-space problem \eqref{eq:probLimWeightlessHalf}.
	\end{rem}
	
	As stated before the theorem, we observe three different bubbling configurations involving three limiting problems. First, we have the weighted problem \eqref{eq:probLim} that we previously encountered within the $a< b$ case. The second, is an entire unweighted analogue given by

	\begin{equation}\label{eq:probLimWeightless}\tag{C}
		\begin{cases}
			-\operatorname{div}\left(\abs{\nabla u}^{p-2}\nabla u\right) = \abs{u}^{p^\ast-2}u &\text{in }\RR^n\\
			u\in \D^{1,p}\left(\RR^n\right),
		\end{cases}
	\end{equation}
	with associated energy functional
	\[
	\phi_\infty:\D^{1,p}(\RR^n)\supseteq\D^{1,p}(\HH) \to \RR,\qquad
	\phi_\infty(u) := \int_{\RR^n} \abs{\nabla u}^p\d{x} - \int_{\RR^n}\abs{u}^{p^\ast}\d{x}.
	\]
	Furthermore, as one cannot consistently rule out ``bubbling" near the boundary of $\Om$, we are forced  to also consider the unweighted limiting problem set in the half-space $\HH := \set{x\in \RR^n : x_n > 0}$,
	\begin{equation}\label{eq:probLimWeightlessHalf}\tag{D}
		\begin{cases}
			-\operatorname{div}\left(\abs{\nabla u}^{p-2}\nabla u\right) = \abs{u}^{p^\ast-2}u &\text{in }\HH\\
			u\in \D^{1,p}\left(\HH\right).
		\end{cases}
	\end{equation}

	It is notable that Theorem \ref{thm:main a=b}, in contrast to Theorem \ref{thm:main a<b} and the model case studied by Struwe \cite{struwe1984global} and Mercuri-Willem \cite{mercuri-willem}, may produce bubbles corresponding to three problems, specifically \eqref{eq:probLim}, \eqref{eq:probLimWeightless}, \eqref{eq:probLimWeightlessHalf}. Nevertheless, it is Theorem \ref{thm:main a=b} that generalizes the model case $a=b=0$. 
	
	What is more, in Theorem \ref{thm:main a=b}, a striking distinction in our result is that we find bubbles derived from solutions belonging to different energy spaces (i.e. weighted \emph{and} unweighted spaces). Therefore, we require a new transform $\tau_{y, \lambda}^-$ that ``transfers information'' from the unweighted space $\D^{1,p}(\RR^n)$ to its weighted counterpart $\D_a^{1,p}(\RR^n,0)$. In said transform, we introduce a factor which does not appear in the model case: a power of $\lambda/\abs{y}$ accounting for the effect of the weighted measures $\abs{x}^{-ap}\d{x}$ and $\abs{x}^{-bq}\d{x}$, which are controlled via a cutoff function. Of course, it is important to observe that the $\lambda/\abs{y}$ ratio vanishes when $a=b=0$ and the cutoff is unnecessary in this case, at which point we recover the classical transform used by Struwe \cite{struwe1984global} and Mercuri-Willem \cite{mercuri-willem}. 
	
	Briefly addressing the proofs of Theorems \ref{thm:main a<b}-\ref{thm:main a=b}, it is instructive to observe that the transformation rules employed within our strategy deviate significantly from those used by Struwe \cite{struwe1984global} and Mercuri-Willem \cite{mercuri-willem}. Speaking informally, in the case $a < b$, we require a transform that accounts for the lack of translation invariance introduced by the aforementioned weighted measures. However, such a transform may seem counter-intuitive since it requires knowing, a priori, that the bubbles concentrate at the origin.  We also note that the proof of Theorem \ref{thm:main a<b} coincides with a sub-argument in the $a=b$ case and, on that account, constitutes a key facet of the global picture. In addition to this, since when $a=b$ we obtain bubbles derived from solutions belonging to distinct energy spaces, we shall also requires a transform $\tau_{y, \lambda}$ from the weighted space $\D_a^{1,p}(\RR^n,0)$ to $\D^{1,p}(\RR^n)$, which acts as a ``pseudo-inverse" to $\tau_{y, \lambda}^-$ appearing in Theorem \ref{thm:main a=b}; both $\tau_{y, \lambda}$ and $\tau_{y, \lambda}^-$ will be discussed in greater depth within \S2. We note that $\tau_{y, \lambda}$ will also be central in the proof of Theorem \ref{thm:main a<b}, where it is used to rule out bubbling away from the origin.

	\subsection{Historical Context and Classifications}\label{section:historical context and classification} Struwe-type decompositions find their roots in the pioneering 1984 work of Michael Struwe \cite{struwe1984global}, where the author establishes the contents of Theorem \ref{thm:main a=b} in the case $p=2$ without weights ($a=b=0$). Therein, a non-existence result of Pohozaev \cite{pohozaevNonexistence} was invoked to rule out the possibility of solutions to \eqref{eq:probLimWeightlessHalf}. Therefore, in this case, one exclusively encounters solutions to \eqref{eq:probLimWeightless}. Still requiring that $a=b=0$, the result of Struwe was improved by Mercuri-Willem \cite{mercuri-willem} to include all $1 < p < n$, provided the $(PS)$-sequence is non-negative. In this setting, thanks to a non-existence result of Mercuri-Willem \cite{mercuri-willem}, they once again only produce functions satisfying  problem \eqref{eq:probLimWeightless}. We also refer the reader to the work of Alves \cite{alves2002existence}.

	Results of this type have been used to shed light on existence and multiplicity of solutions, see for instance \cite{DevillanovaSolimini, ClappWeth, vetois2007, BarlettaCanditoMaranoPerera, clappRios}. In addition, we note that decompositions of this type have been established in other contexts. Examples include Hebey-Robert \cite{HebeyRobert} for Paneitz type operators, Saintier \cite{Saintier} for \(p\)-Laplace operators on compact manifolds, El-Hamidi-V\'etois \cite{Vetois3} for anisotropic operators, Almaraz \cite{Almaraz} for non-linear boundary conditions and Mazumdar \cite{mazumdar} for polyharmonic operators on compact manifolds.
	
	It should also be noted that, in 1985, Lions \cite{Lions1, Lions2} published works relating to concentration-compactness phenomena in the calculus of variations. In particular, in a spirit similar to Struwe's result, he established a decomposition with convergence taking place at the level of measures.
	
	We also point out that the stability of the CKN-inequality, which closely relates to quantitative versions of Struwe decompositions, has been recently studied in Ciraolo-Figalli-Maggi \cite{ciraolo2018quantitative}, Figalli-Glaudo \cite{figalli2020sharp}, Deng-Sun-Wei \cite{deng2021sharp}, and Wei-Wu \cite{weiWu}.

	\subsubsection{Classifications} Our understanding of the loss of compactness phenomenon presented in Theorems \ref{thm:main a<b}, \ref{thm:main a=b} may be improved by characterizing solutions to problems \eqref{eq:probLim}, \eqref{eq:probLimWeightless} and \eqref{eq:probLimWeightlessHalf}. As such, we provide some historical context and recent results, particularly with regards to classifications of solutions to problem \eqref{eq:probLim} and its unweighted analogue \eqref{eq:probLimWeightless}.
	
	Let us first remark that problems \eqref{eq:probLim} and \eqref{eq:probLimWeightless} are intimately related to the  Caffarelli-Kohn-Nirenberg (CKN) inequality on $\RR^n$. Proven by Caffarelli-Kohn-Nirenberg in \cite{CKN}, the CKN-inequality states that there is a best (i.e. minimal) constant $C>0$ such that, for any $x_0\in\RR^n$, there holds
	\begin{equation}\label{eq:CKN}
		\left( \int_{\RR^n} \abs{u(x)}^q\abs{x-x_0}^{-bq}\d{x} \right)^{1/q} \leq C\left( \int_{\RR^n} \abs{\nabla u(x)}^p\abs{x-x_0}^{-ap}\d{x} \right)^{1/p}
	\end{equation}
	for all $u\in C_c^\infty(\RR^n)$. When $a=b=0$, this reduces to the Gagliardo-Nirenberg-Sobolev inequality and the best constant is known by works of Aubin \cite{aubin} and Talenti \cite{talenti1976best}, though computations of this constant date as far back as Rodemich's seminar \cite{rodemich}. 
	
	By standard techniques in the calculus of variations, a suitable constant multiple of functions achieving equality in \eqref{eq:CKN} solve problem \eqref{eq:probLim}. In fact, early classification were for minimizers of the ratio
	\begin{equation}\label{eq:CKNratio}
		\dfrac{\left( \int_{\RR^n} \abs{\nabla u(x)}^p\abs{x}^{-ap}\d{x} \right)^{1/p}}{\left( \int_{\RR^n} \abs{u(x)}^q\abs{x}^{-bq}\d{x} \right)^{1/q}},\qquad u\in \D_a^{1,p}(\RR^n, 0).
	\end{equation}
	In the unweighted case $a=b=0$, Talenti \cite{talenti1976best} exhibited a family of (radially symmetric) functions minimizing \eqref{eq:CKNratio}. For $p=2$ and $n\ge 3$, the existence of radially symmetric functions minimizing \eqref{eq:CKNratio} was determined by Lieb \cite{lieb1983} when $a=0$. In addition, an explicit form for said best constant was estblished therein. Furthermore, for $p=2$, Chou and Chu \cite{chou-chu} determined the best constant in \eqref{eq:CKN} when $a\ge 0$ and categorizes the functions minimizing \eqref{eq:CKNratio} when $a>0$. Considering instead $a<0$, Horiuchi \cite{Horiuchi1997} showed that, when $p=2, n\ge 3$ and $a=b < 0$, the best constant in \eqref{eq:CKN} is precisely the Sobolev constant and equality cannot be achieved by any function in $\D_a^{1,p}(\RR^n, x_0)$. In this same paper, Horiuchi also determined the best constant in \eqref{eq:CKN} for $1<p<{n}/({1+a-b})$, provided $a\ge 0$ and $b < n/q$.
	
	In recent years, positive solutions to problem \eqref{eq:probLimWeightless} have been fully classified and have the form
	\begin{equation}\label{eq:limitingSolutionsUnweighted}
		x \mapsto \left[\frac{\left(\mu n \left(\frac{n-p}{p-1}\right)^{p-1}\right)^{1/p}}{\mu + \abs{x-x_0}^{\frac{p}{p-1}}}\right]^{\frac{p}{q-p}},\qquad \mu>0, x_0\in\RR^n.
	\end{equation}
	This classification was obtained in the case $p=2$ by Caffarelli, Gidas and Spruck \cite{CaffarelliGidasSpruck}; they showed, via the method of moving planes, that all positive solutions to \eqref{eq:probLimWeightless} are radial and thus determined an explicit form for said solutions. Furthermore, Damascelli, Merch\'an, Montoro and Sciunzi \cite{damascellimerchanmontorosciunzi} utilized a previous classification due to Guedda and Veron \cite{guedda} to extend this result to the range $\frac{2n}{n+2}\le p < 2$. Similarly, V\'etois \cite{vetois2016} improved the range to $1<p<2$ by exploiting a decay result of Damascelli-Ramaswamy \cite{damascelliRamaswamy}. Finally, the case $2<p<n$ was handled by Sciunzi \cite{SCIUNZI201612} using an a-priori estimate of V\'etois \cite{vetois2016}. We mention that, for ranges of $p$ depending on $n$, this classification is also known for solutions that are \emph{not} assumed to have finite energy (i.e. $u$ need not belong to $\D^{1,p}(\RR^n)$ a priori), see Catino, Monticelli and Roncoroni \cite{catinomonticellironcorini}, Ou \cite{Ou2022} and V\'etois \cite{vetois2023} for more detail. 
	
	Similar classification results are known for the weighted analog of \eqref{eq:probLimWeightless}, i.e. problem \eqref{eq:probLim}. Indeed, in some cases, positive solutions to problem \eqref{eq:probLim} are known to be of the form
	\begin{equation}\label{eq:limitingSolutionsWeighted}
		x \mapsto \left[\frac{\left(\dfrac{\mu n (n-p(1+a))^p}{(n-p(1+a-b))(p-1)^{p-1}}\right)^{1/p}}{\mu + c\abs{x}^{\frac{(n-p(1+a))}{(n-p(1+a-b))}\frac{p(1+a-b)}{p-1}}}\right]^{\frac{p}{q-p}},\qquad \mu>0
	\end{equation}
	We note that, if $a=b=0$, the weighted measures $\abs{x}^{-ap}\d{x}$ and $\abs{x}^{-bq}\d{x}$ are translation invariant; in particular the functions in \eqref{eq:limitingSolutionsWeighted} cover all possible solutions of \eqref{eq:probLim}, up to translation. Put otherwise, in the case $a=b=0$ there is no distinction between problems \eqref{eq:probLim} and \eqref{eq:probLimWeightless} so all positive solutions are given by equation \eqref{eq:limitingSolutionsUnweighted}; we omit this case from the reminder of our discussion. When $a=0$, generalizing a result of Guedda-Veron \cite{guedda} to allow for $b\ne 0$, Ghoussoub and Yuan \cite{ghoussoub-yuan} showed that every positive radial solutions of \eqref{eq:probLim} is of the form  \eqref{eq:limitingSolutionsWeighted}. Considering only $p=2$, Dolbeault and Esteban \cite{DolbeaultEsteban} show that all positive (not a-priori assumed radial) solutions in the non-limiting case $a<b$ are as in \eqref{eq:probLim}, provided either 
	\begin{itemize}
		\item $a\ge 0$
		\item $a<0$ and $b \le \gamma\left(\frac{n}{p\sqrt{\gamma^2+n-1}} - 1\right)$.
	\end{itemize}
	Moreover, for dimensions $n\ge 3$, Ciraolo and Corso \cite{ciraolo-corso}, utilizing a-priori asymptotic estimates of V\'etois-Shakerian \cite{shakerian},  demonstrated that this classification extends to the whole range $1<p<n$ in the limit case $a=b \ge 0$ as well as a wider range of possible values for $a,b$ provided $p<n/2$. Building upon the result of Ciraolo-Corso \cite{ciraolo-corso} and again utilizing estimates from V\'etois-Shakerian \cite{shakerian}, Le and Le \cite{LeLe} showed that, when $a=0$, all positive solutions of \eqref{eq:probLim} are radial and therefore classified by Ghoussoub-Yuan \cite{ghoussoub-yuan}. 
	
	We remark that all classification results (for positive solutions) discussed thus far yield radial solutions. However, certain minimizers of \eqref{eq:CKN} are known to break symmetry for some values of $a,b,p$; see \cite{catrinawang, byeonwang, fellischneider, smetswillem, DolbeaultEstebanLossTarantello, ChibaHoriuchi} for details.
	
	\subsection{Paper structure} We devote the remainder of this paper to the establishment of Theorems \ref{thm:main a<b}-\ref{thm:main a=b}. Firstly, in \S2, we introduce some convenient notation and recall important properties of the homogenous weighted Sobolev space  $\D_a^{1,p}$ and Palais-Smale sequences that will be liberally invoked during the proof of our main results. Secondly, in \S3, we discuss the new transformation rules in greater detail and, in doing so, we provide explicit bounds for their operator norms. Following this, in \S4, we shift our attention to the limiting behaviour of rescaled domains and make precise in what sense domains may converge. Then, in \S5, we consider how the new transforms interact with energy functionals associated to problems \eqref{eq:probOm}--\eqref{eq:probLimWeightlessHalf}. Leveraging our work from the previous sections, we provide results that will be central to the proofs of our main theorems in \S6. Specifically, we show how one can extract a solution to the problem in $\Om$ from a given $(PS)$-sequence, which constitutes the first step of each main proof. In addition, still in \S6, we demonstrate that, when \(a=b\), the new transforms gives rise to new ``bubble'' terms derived from solutions to three distinct limiting problems, and how one may recover a $(PS)$-sequence for the original problem. Finally, in \S7-\S8 using these aforementioned lemmas, we provide the proofs of Theorems \ref{thm:main a<b}-\ref{thm:main a=b} by way of an iterative step-by-step procedure. 
	
	\subsection*{Acknowledgments} I would like to thank my supervisors, Professor J\'{e}r\^{o}me V\'{e}tois and Pengfei Guan, whose time and support were indispensable to the completion of this work. I would also like to thank Dana Berman for her valuable suggestions.
	
	\section{Homogeneous Sobolev Spaces and Palais-Smale Sequences} 
	
	In this section, we introduce notation involving weighted homogeneous Sobolev spaces and provide results describing properties of the functions living therein. Recall that, for any open set $U$ and $x_0 \in \RR^n$, we let $\D_a^{1,p}(U, x_0)$ be the completion of $C_c^\infty(U)$ under the norm \eqref{xnorm}.  In the same vein, we introduce the following two weighted norms	\[
	\norm{u}_{L^q_b(U,x_0)}^q :=\int_U \abs{u}^q\abs{x-x_0}^{-bq}\d{x}, 
	\quad
	\norm{v}_{L^p_a(U,x_0)}^p := \int_U \abs{v}^p\abs{x-x_0}^{-ap}\d{x}.
	\]
	The above expressions induce norms on the weighted Lebesgue spaces $L^q_b(U, x_0)$ and $L^p_a(U,x_0)$, consisting (respectively) of those (Lebsegue measurable) $u, v : U \to \RR$ such that
	\[
	\int_U \abs{u}^q\abs{x-x_0}^{-bq}\d{x} < \infty, \quad \text{and}\quad   \int_U \abs{v}^p\abs{x-x_0}^{-ap}\d{x} < \infty.
	\]
	In this notation, an equivalent formulation of the CKN-inequality \eqref{eq:CKN} is thus given by
	\begin{equation}\label{eq:CKNv2}
		C_{a,b}\norm{w}_{L^q_b(\RR^n, x_0)}^{p} \leq \norm{\nabla w}_{L^p_a(\RR^n, x_0)}^p, \quad \forall w \in \D_a^{1,p}(\RR^n, x_0)
	\end{equation}
	where $C_{a,b} = C_{a,b}(n,p)$ denotes the sharp constant that will henceforth be called a CKN-constant.
	
	\begin{rem}
		As mentioned in \S1, when $a=b=0$, we shall omit the presence of weights and only write $\D^{1,p}(U)$. In this case, with $U = \RR^n$, available literature includes, for instance, Willem \cite[\S 7.2]{willemFunctional} and this homogeneous Sobolev space is precisely
		\[
		\D^{1,p}(\RR^n) = \left\{ u \in L^{p^\ast}(\RR^n) : \nabla u \in L^p(\RR^n) \right\}.
		\]
	\end{rem}
	
	For the purposes of this paper, we shall require results from Sobolev theory in the context of the weighted space $\D_a^{1,p}(U,x_0)$, which we recall below. The reader may consult Chernysh \cite[Chapter 2]{mscChernysh} for more details on these results. Descriptively, every function $u \in \D_a^{1,p}(U, x_0)$ belongs to $L^q_b(U,x_0)$, whereas $\nabla u$ exists in the weak sense on $U \setminus \{ x_0\}$ such that $\nabla u \in L^p_a(U \setminus \{x_0\}, x_0)$. Moreover, when $a \ge 0$, we can assert that $\nabla u$ exists in the weak sense on the \emph{whole} of $U$. We further note that weak gradients vanish on level sets. That is, for every \(c\in\RR\),
	\[
	\nabla u \equiv 0 \text{ almost everywhere on } \set{x\in \RR^n : u(x) = c}.
	\]
	Following this, we also recall that, under condition \eqref{eq:conditions}, $\D_a^{1,p}(U, x_0)$ is always a reflexive Banach space. In fact, every bounded sequence in $\D_a^{1,p}(U, x_0)$ admits a subsequence converging weakly and almost everywhere on the set $U$. We also remark that several key results for traditional Sobolev spaces have known analogues in $\D_a^{1,p}(U, x_0)$, e.g. the Rellich-Kondrachov embedding theorem. In particular, the natural embedding \(\D_a^{1,p}(U, x_0) \hookrightarrow L_{\text{loc}}^p\big(\RR^n, \abs{x-x_0}^{-ap}\big)\) is compact.
	
	Denoting by $\D^{-1,p^\prime}_a(U, x_0)$, with $p^\prime := p/(p-1)$, the topological dual of $\D_a^{1,p}(U, x_0)$, a Riesz-type representation theorem asserts that every functional $\varphi \in \D_a^{-1,p^\prime}(U,x_0)$ takes the form
	\begin{equation}\label{eq:Riesz}
		\varphi(u) = \int_U \pinner{\nabla u, \mathbf{g}} \abs{x-x_0}^{-ap}\d{x}, \qquad \forall u \in \D_a^{1,p}(U, x_0)
	\end{equation}
	for some $\RR^n$-valued function $\mathbf{g} \in L^{p^\prime}_a(\RR^n, x_0)$. Here, \(\pinner{\cdot, \ast}\) denotes the usual inner-product on \(\RR^n\). Finally, for simplicity, the topological dual of $\D^{1,p}(U)$ is simply denoted $\D^{-1,p^\prime}(U)$.
	
	\subsection{Pointwise Convergence of the Gradients}
	
	We have noted that bounded sequences in $\D^{1,p}_a(U,x_0)$ have subsequences that converge weakly and almost everywhere on $U$. In this subsection, we give conditions under which we can find a subsequence whose \emph{gradients} converge pointwise.  More specifically, we establish a weighted version of Theorem 3.3 in Mercuri-Willem \cite{mercuri-willem}.

	Let us first define  the Lipschitz continuous function  $T : \RR \to \RR$ by the rule
	\begin{equation}\label{eq:TDef}
		T(x) := \begin{cases}
			x & \text{if } \abs{x} \leq 1,\\
			\frac{x}{\abs{x}} & \text{if } \abs{x} > 1.
		\end{cases}
	\end{equation}
	By considering suitable mollifications of \(T\), it is straightforward to see that, given $u \in \D^{1,p}_a(U,x_0)$, $T(u) \in \D^{1,p}_a(U,x_0)$.
	
	\begin{prop}\label{prop:gradientConvergence}
		Let $(u_\alpha)$ be a bounded sequence in $\D^{1,p}_a(U,x_0)$ converging pointwise almost everywhere to $u \in \D^{1,p}_a(U,x_0)$ as $\alpha \to \infty$. Let $T$ be given by \eqref{eq:TDef} and assume that $(U_k)$ is an increasing sequence of open subsets of $U$ such that $\bigcup_{k \geq 1} U_k = U$. Assume further that
		\begin{equation}\label{eq:lemGradient1}
			\lim_{\alpha \to \infty} \int_{U_k} \pinner{\left[ \abs{\nabla u_\alpha}^{p-2}\nabla u_\alpha - \abs{\nabla u}^{p-2}\nabla u\right], \nabla \left[T(u_\alpha - u)\right]}\,\d{\mu} = 0
		\end{equation}
		for each $k \geq 1$, where $\d{\mu} = \abs{x-x_0}^{-ap}\d{x}$. Then, there exists a subsequence $(u_\beta)$ such that  $\nabla u_\beta \to \nabla u$ almost everywhere on $U$.
	\end{prop}
	\begin{proof}

		We adapt an argument from Szulkin-Willem \cite{Szulkin-Willem}, to show that, given \(k\in \NN\), up to a subsequence, \(\nabla u_\alpha \to\nabla u\) almost everywhere on \(U_k\).

		Fix \(k\in\NN\) and, for each $\alpha \in \NN$, consider the set
		\[
		E^{(k)}_\alpha := \left\{ x \in U_k  : \abs{u_\alpha(x) - u(x)} \leq 1  \right\}.
		\]
		Then, \eqref{eq:lemGradient1} can be written as
		\begin{equation}\label{eq:lemgradient2}
			\lim_{\alpha \to \infty} \int_{E^{(k)}_\alpha} \pinner{\left[\abs{\nabla u_\alpha}^{p-2}\nabla u_\alpha - \abs{\nabla u}^{p-2}\nabla u\right], [\nabla u_\alpha  - \nabla u]}\abs{x-x_0}^{-ap}\d{x} = 0.
		\end{equation}
		Since the integrand is non-negative and $\mathbf{1}_{E^{(k)}_\alpha}(x) \to 1$ for almost every $x \in U_k$, passing to a subsequence if necessary, we see that
		\begin{align*}
			\pinner{\left[\abs{\nabla u_\alpha}^{p-2}\nabla u_\alpha - \abs{\nabla u}^{p-2}\nabla u\right], [\nabla u_\alpha  - \nabla u]} \to 0
		\end{align*}
		pointwise almost everywhere on $U_k$. Appealing classical results (see, for instance, Szulkin-Willem \cite[Lemma 2.1]{Szulkin-Willem}), it follows that $\nabla u_\alpha \to \nabla u$ pointwise a.e. on $U_k$. Finally, a diagonal argument completes the proof.
	\end{proof}
	
	\begin{cor}\label{cor:gradientConvergence}
		Let $(u_\alpha)$ be a bounded sequence in $\D^{1,p}_a(U,x_0)$ converging pointwise almost everywhere to $u \in \D^{1,p}_a(U,x_0)$ as $\alpha \to \infty$. Under the assumptions of Proposition \ref{prop:gradientConvergence}, there exists a subsequence $(u_\beta)$ with the following properties:
		\begin{enumerate}[label=(\arabic*)]
			\item $\nabla u_\beta \to \nabla u$ pointwise almost everywhere on $U$;
			\item $\lim\limits_{\beta \to \infty} \left( \norm{u_\beta}^p_{\D^{1,p}_a(U,x_0)} - \norm{u_\beta - u}^p_{\D^{1,p}_a(U,x_0)} \right) = \norm{u}^p_{\D^{1,p}_a(U,x_0)}$;	
			\item and $$\abs{\nabla u_\beta}^{p-2}\nabla u_\beta - \abs{\nabla u_\beta - \nabla u}^{p-2}(\nabla u_\beta - \nabla u) \to \abs{\nabla u}^{p-2}\nabla u$$ strongly in $L^{\frac{p}{p-1}}(U, \abs{x-x_0}^{-ap})$
			\item and $$\abs{u_\beta}^{q-2} u_\beta - \abs{u_\beta - u}^{q-2}(u_\beta - u) \to \abs{u}^{q-2}u$$ strongly in $L^{\frac{q}{q-1}}(U, \abs{x-x_0}^{-bq})$.
		\end{enumerate}
	\end{cor}
	\begin{proof}
		The first item is the conclusion of Proposition \ref{prop:gradientConvergence} whence (2) is a direct consequence of the Br\'ezis-Lieb Lemma. Finally, (3) and (4) follows from straightforward adaptations of Alves \cite[Lemma 3]{alves2002existence} and Mercuri-Willem \cite[Lemma 3.2]{mercuri-willem}.
	\end{proof}
	
	\subsection{Palais-Smale Sequences and Energies of Critical Points}
	
	\begin{prop}\label{prop:psBounded}
		A Palais-Smale sequence $(u_\alpha)$ for \eqref{eq:probOm} is bounded.
	\end{prop}
	\begin{proof}
		We follow the argument used in Struwe \cite[Lemma 2.3]{struwe1984global}.
		Let $(u_\alpha)$ be a Palais-Smale sequence for \eqref{eq:probOm}. A simple calculation shows that, for some \(C>0\),
		\begin{align*}
			\norm{u_\alpha}_{\D^{1,p}_a(\Om, 0)}^p &= p\phi(u_\alpha) + \frac{p}{q} \int_\Om \frac{\abs{u_\alpha}^q}{\abs{x}^{bq}}\d{x}
			\leq C +\frac{p}{q} \int_\Om \frac{\abs{u_\alpha}^q}{\abs{x}^{bq}}\d{x}
		\end{align*}
		where we have used that $\phi(u_\alpha)$ is bounded in $\alpha$. Next, observe that
		\[
		\left(1-\frac{p}{q}\right)\int_\Om \frac{\abs{u_\alpha}^q}{\abs{x}^{bq}}\d{x}
		= p\phi(u_\alpha) - \inner{\phi^\prime(u_\alpha), u_\alpha}
		\le C + o(1)\norm{u_\alpha}_{\D^{1,p}_a(\Om, 0)} 
		\]	
		It then follows that, for a suitable \(\tilde{C}>0\),
		\[
		\norm{u_\alpha}_{\D^{1,p}_a(\Om, 0)}^p 
		\leq C +\frac{p}{q} \int_\Om \frac{\abs{u_\alpha}^q}{\abs{x}^{bq}}\d{x}\\	
		\leq \tilde{C} + o(1)\norm{u_\alpha}_{\D^{1,p}_a(\Om, 0)}
		\]
		whence $(u_\alpha)$ is bounded in $\D^{1,p}_a(\Om, 0)$. 
	\end{proof}
	
	Using this, we show that (in the limit) Palais-Smale sequences for $\phi$ are uniformly bounded from below by zero in energy. More precisely, we have the following:
	
	\begin{lem}\label{lem:energy}
		Let $(u_\alpha)$ be a Palais-Smale sequence for \eqref{eq:probOm}. Then, $(u_\alpha)$ has non-negative limiting energy. More precisely,
		\[
		\liminf_{\alpha \to \infty} \phi(u_\alpha) \ge 0.
		\]
	\end{lem}
	\begin{proof}
		Since $(u_\alpha)$ is bounded and a Palais-Smale sequence, it is readily seen that
		\[
		\int_{\Om} \abs{\nabla u_\alpha}^p\abs{x}^{-ap}\d{x} - \int_{\Om} \abs{u_\alpha}^q\abs{x}^{-bq}\d{x} = \inner{\phi^\prime(u_\alpha), u_\alpha}= o(1)
		\]
		as $\alpha \to \infty$. Consequently, as $\alpha \to\infty$,
		\[
		\phi(u_\alpha) = \frac{\norm{\nabla u_\alpha}_{L^p_a(\Om, 0)}^p}{p} - \frac{ \norm{u_\alpha}_{L^q_b(\Om, 0)}^q}{q}
		= \left( \frac{1}{p} - \frac{1}{q}\right) \norm{u_\alpha}_{L^q_b(\Om, 0)}^q + o(1)
		\]
		Passing to the limit inferior, we conclude the proof.
	\end{proof}
	
	\begin{prop}\label{prop:energy}
		Let $C_{a,b} > 0$ be the CKN-constant, in particular \(C_{a,b}>0\) is such that \eqref{eq:CKNv2} holds. Let \(U\subseteq \RR^n\) be an open set and suppose $u \in \D^{1,p}_a(U,0)$ is a non-trivial critical point of $\phi_{0,\infty}$ in \(U\), i.e. \(\phi^\prime_{0, \infty}(u, h) = 0\) for all \(h\in \D_a^{1,p}(U, 0)\). Then,
		\[
		0<\left(C_{a,b}\right)^{\frac{n}{p(1+a-b)}}\left(\frac{1+a-b}{n}\right) \le \phi_{0,\infty}(u).
		\]
		In particular, 
		\begin{enumerate}
			\item The constant \(\left(C_{a,b}\right)^{\frac{n}{p(1+a-b)}}\left(\frac{1+a-b}{n}\right)\) is a lower bound for the energy of non-trivial solutions to problems \eqref{eq:probOm} and \eqref{eq:probLim}.
			\item If \(a=b\), then \(\left(C_{0,0}\right)^{\frac{n}{p}}\left(\frac{1}{n}\right)\) is a lower bound for the energy of non-trivial solutions to problems \eqref{eq:probLimWeightless} and \eqref{eq:probLimWeightlessHalf}.
		\end{enumerate}
	\end{prop}
	\begin{proof}
		Since $u \in \D^{1,p}_a(U, 0)$ is a critical point of $\phi_{0,\infty}$ in \(U\), after extending \(u\) by \(0\) outside of \(U\), we have
		\begin{equation}\label{eq:norms_equality}
			0 = \inner{\phi^\prime_{x_0,\infty}(u),u} = \int_{\RR^n} \abs{\nabla u}^p\abs{x+x_0}^{-ap}\d{x} - \int_{\RR^n} \abs{u}^q\abs{x+x_0}^{-bq}\d{x}.
		\end{equation}
		Therefore,
		\[
		\begin{aligned}
			\phi_{0,\infty}(u) 
			= \frac{1}{p}\norm{\nabla u}_{L^p_a(\RR^n, -x_0)}^p - \frac{1}{q}\norm{u}^{q}_{L^q_b(\RR^n, -x_0)} 
			&= \left(\frac{1}{p} - \frac{1}{q}\right)\norm{u}^{q}_{L^q_b(\RR^n, -x_0)}\\
			&= \left(\frac{1+a-b}{n}\right)\norm{u}^{q}_{L^q_b(\RR^n, -x_0)}
		\end{aligned}
		\]
		On the other hand, by \eqref{eq:CKNv2} and using again \eqref{eq:norms_equality},
		\[
		C_{a,b} \norm{u}_{L^q_b(\RR^n, -x_0)}^p \leq \norm{\nabla u}_{L^p_a(\RR^n, -x_0)}^p = \norm{u}^{q}_{L^q_b(\RR^n, -x_0)},
		\]
		whence
		\[
		0 < \left(C_{a,b}\right)^{q/(q-p)} \le \norm{u}^{q}_{L^q_b(\RR^n, -x_0)}.
		\]
		We conclude that,
		\[
		\begin{aligned}
			\phi_{0,\infty}(u) 
			= \left(\frac{1+a-b}{n}\right)\norm{u}^{q}_{L^q_b(\RR^n, -x_0)} 
			&\ge \left(C_{a,b}\right)^{q/(q-p)}\left(\frac{1+a-b}{n}\right)\\
			&= \left(C_{a,b}\right)^{\frac{n}{p(1+a-b)}}\left(\frac{1+a-b}{n}\right).
		\end{aligned}
		\]
	\end{proof}
	
	\section{Homogeneity and New Rescaling Transformations}
	Let $u \in \D^{1,p}_a(\RR^n,0)$ be given and fix an arbitrary point $x_0 \in \RR^n$ along with some scaling factor $\lambda > 0$. We consider the following rescaling of $u$:
	\begin{equation}\label{eq:rescale}
		v(x) := \lambda^{\gamma} u(\lambda(x+x_0)),
	\end{equation}
	where $\gamma > 0$ is the homogeneity exponent given by equation \eqref{eq:exponent}. With this rescaling, an easy calculation yields that
	\[
	\norm{v}_{L^q_b(\RR^n, -x_0)} = \norm{u}_{L^q_b(\RR^n, 0)} \quad \text{and} \quad 	\norm{\nabla v}_{L^p_a(\RR^n, -x_0)} = \norm{\nabla u}_{L^p_a(\RR^n, 0)}
	\]
	whence it follows that $v \in \D^{1,p}_a(\RR^n,-x_0)$. 
	
	\subsection{Modifying the Rescaling: A New Transform}
	Let $u \in \D_a^{1,p}(\RR^n, 0)$ and consider an arbitrary point $y \in \RR^n\setminus\set{0}$ along with some $\lambda > 0$. We also fix a smooth function $\eta \in C_c^\infty(\RR^n)$ satisfying \begin{equation}\label{eq:cutoffDef}
		\begin{cases}
			0 \le \eta \le 1\\
			\eta = 1 &\text{in }B(0, 1/2)\\
			\eta = 0 &\text{outside } B(0,1).
		\end{cases}
	\end{equation}
	Define the transform
	\[
	\tau_{y, \lambda} : \D_a^{1,p}(\RR^n,0) \to \D^{1,p}(\RR^n)
	\]
	via the rule
	\begin{equation}\label{eq:rescale_mod}
		\tau_{y, \lambda}u(x) = \begin{cases}
			\lambda^\gamma u(\lambda x + y) &\text{if } a = b = 0\\
			\displaystyle
			\left(\frac{\lambda}{\abs{y}}\right)^b \lambda^\gamma u\left(\lambda x + y\right)\eta\left(\frac{2\lambda x}{\abs{y}}\right) &\text{otherwise}
		\end{cases}
	\end{equation}
	
	We now examine how this transformation behaves with respect to operator bounds. 
	\begin{lem}\label{lem:rescale_mod_bound}
		There exists a constant $C>0$ such that, given a point $y \in \RR^n\setminus\set{0}$ along with some $\lambda > 0$, there holds
		\[
		\begin{aligned}
			\norm{\tau_{y,\lambda}u}_{L^q(\RR^n)} &\le C\norm{u}_{L^q_b(\RR^n,0)}\\
			\norm{\tau_{y,\lambda}u}_{\D^{1,p}(\RR^n)} &\le C\left(\frac{\lambda}{\abs{y}}\right)^{b-a}\norm{u}_{\D_a^{1,p}(\RR^n,0)}
		\end{aligned}
		\]
		for all $u\in \D_a^{1,p}(\RR^n,0)$.
	\end{lem}
	\begin{proof} The case $a=b=0$ follows directly from homogeneity. Otherwise, we begin with a simple observation. Suppose $x\in \RR^n$ is such that
		\[
		\abs{\frac{2\lambda x}{\abs{y}}} \le 1 \iff \abs{\lambda x} \le \frac12\abs{y}.
		\]
		By the triangle and reverse-triangle inequalities, we then have
		\[
		\frac{1}{2} \le \frac{\abs{\lambda x + y}}{\abs{y}} \le \frac32.
		\]
		Therefore, there exists a constant $C = C(ap, bq) > 0$ independent of $\lambda, y$ such that
		\begin{equation}\label{eq:ratio bound}
			\left(\frac{\abs{\lambda x + y}}{\abs{y}}\right)^{ap} \le C,\quad
			\left(\frac{\abs{\lambda x + y}}{\abs{y}}\right)^{bq} \le C
		\end{equation}
		for all $x\in \RR^n$ such that ${2\lambda \abs{x}}\le{\abs{y}}$. In particular, these last equations hold whenever ${2\lambda x}/{\abs{y}}$ is in the support of $\eta$.
			
		Given $u\in \D_a^{1,p}(\RR^n,0)$, we use inequality \eqref{eq:ratio bound} in order to bound the $L^q$ and $\D^{1,p}$ norms of $v := \tau_{y,\lambda}u$. Since $0\le\eta\le 1$, observe that
		\begin{align*}
			\norm{v}_{L^q(\RR^n)}^q
			= \int_{\RR^n}\abs{v}^q
			&\le \left(\frac{\lambda}{\abs{y}}\right)^{bq} \lambda^{\gamma q} \int_{2\lambda\abs{x}\le \abs{y}}\abs{u\left(\lambda x + y\right)}^q\d{x}\\
			&=\lambda^n\int_{2\lambda\abs{x}\le \abs{y}}\abs{u\left(\lambda x + y\right)}^q\abs{y}^{-bq}\d{x}\\
			&\le C\lambda^n\int_{2\lambda\abs{x}\le \abs{y}}\abs{u\left(\lambda x + y\right)}^q\abs{\lambda x + y}^{-bq}\d{x}\\
			&= C\int_{2\abs{z-y} \le \abs{y}}\abs{u\left(z\right)}^q\abs{z}^{-bq}\d{z}\\
			&\le C\norm{u}_{L^q_b(\RR^n,0)}^q
		\end{align*}		
		which establishes the first inequality. In order to bound the $\D^{1,p}$-norm of $v$, we will consider two separate terms obtained via the product rule:
		\begin{align*}
			\norm{v}_{\D^{1,p}(\RR^n)}^p
			&=\int_{\RR^n} \abs{\nabla v}^p\d{x}\\
			&\le 2^{p-1}\left(\frac{\lambda}{\abs{y}}\right)^{p(b-a)}\lambda^{n}\int_{\RR^n}\abs{\eta\left(\frac{2\lambda x}{\abs{y}}\right)\nabla u (\lambda x + y)}^p\abs{y}^{-ap}\d{x}\\
			&\quad + 2^{p-1+p} \left(\frac{\lambda}{\abs{y}}\right)^p\lambda^{np/q}\int_{\RR^n}\abs{u (\lambda x + y)\nabla \eta\left(\frac{2\lambda x}{\abs{y}}\right)}^p\abs{y}^{-bp}\d{x}\\
			&=: I_1 + I_2. 
		\end{align*}
		Here, we have used that $p>1$ so that convexity implies $(\kappa + \iota)^p \le 2^{p-1}\left(\kappa^p + \iota^p\right)$ for all $\kappa,\iota \ge 0$. We bound $I_1$ following the procedure used to bound the $L^q$-norm of $v$. This yields
		\[
		I_1 \le C 2^{p-1}\left(\frac{\lambda}{\abs{y}}\right)^{p(b-a)}\norm{u}_{\D_a^{1,p}(\RR^n,0)}^p.
		\]
		For the second term, with $C_1>0$ such that $2^{1-2p}C_1$ is an upper bound of $\abs{\nabla \eta}^p$, we have
		\begin{align*}
			I_2
			&\le C_1 \left(\frac{\lambda}{\abs{y}}\right)^p\lambda^{np/q}\int_{2\lambda\abs{x}\le \abs{y}}\abs{u (\lambda x + y)}^p\abs{y}^{-bp}\d{x}\\
			&\le  C_1 \left(\frac{\lambda}{\abs{y}}\right)^p\left(\lambda^n\int_{2\lambda\abs{x}\le \abs{y}}\abs{u (\lambda x + y)}^q\abs{y}^{-bq}\d{x}\right)^{p/q}\left(\int_{2\lambda\abs{x}\le \abs{y}} 1 \d{x}\right)^{p(1+a-b)/n}
		\end{align*}
		by H\"older's inequality (with H\"older exponent $q/p$ and H\"older-conjugate exponent $n/[p(1+a-b)]$). Then, up to a modification of the constant $C>0$, we see that
		\begin{align*}
			I_2 &\le C \left(\frac{\lambda}{\abs{y}}\right)^{p(b-a)}\left(\lambda^n\int_{2\lambda\abs{x}\le \abs{y}}\abs{u (\lambda x + y)}^q\abs{\lambda x + y}^{-bq}\d{x}\right)^{p/q}\\
			&= C \left(\frac{\lambda}{\abs{y}}\right)^{p(b-a)}\left(\int_{2\abs{z-y}\le \abs{y}}\abs{u (z)}^q\abs{z}^{-bq}\d{x}\right)^{p/q}\\
			&\le  C \left(\frac{\lambda}{\abs{y}}\right)^{p(b-a)}\norm{u}_{L_b^q(\RR^n,0)}^p.
		\end{align*}
		Finally, by the CKN inequality \eqref{eq:CKN}, up to another modification of the constant $C>0$, we deduce that
		\[
		I_2 \le C\left(\frac{\lambda}{\abs{y}}\right)^{p(b-a)}\norm{u}_{\D_a^{1,p}(\RR^n,0)}^p.
		\]
		Combining this with our inequality for $I_1$, we deduce that (possibly adjusting $C>0$ again)
		\[
		\norm{v}_{\D^{1,p}(\RR^n)}^p \le C\left(\frac{\lambda}{\abs{y}}\right)^{p(b-a)}\norm{u}_{\D_a^{1,p}(\RR^n,0)}^p
		\]
		which completes our proof since $C$ is independent of $u,y,\lambda$.
	\end{proof}
	
	\begin{rem}\label{rem:rescale_mod_bound_set}
		Inspecting the proof of Lemma \ref{lem:rescale_mod_bound}, it is easy to see that we may instead bound the norms on an arbitrary measurable set $U\subseteq \RR^n$. Indeed, the very same argument shows that the same constant $C>0$ (which is independent of $y, \lambda, U$) satisfies
		\begin{equation}
			\begin{aligned}
				\norm{\tau_{y,\lambda}u}_{L^q(U)} &\le C\norm{u}_{L^q_b(\lambda U + y,0)}\\
				\norm{\nabla\left[\tau_{y,\lambda}u\right]}_{L^p(U)} &\le C\left(\frac{\lambda}{\abs{y}}\right)^{b-a}\left(\norm{\nabla u}_{L^p_a(\lambda U + y,0)} +  \norm{u}_{L^q_b(\lambda U + y,0)}\right)
			\end{aligned}
		\end{equation}
		for all $u\in \D_a^{1,p}(\RR^n,0)$.
	\end{rem}
	In line with what we have done thus far, we shall also require a suitable ``reverse'' transform which will serve as a pseudoinverse for $\tau_{y,\lambda}$. Formally, we define
	\[
	\tau_{y, \lambda}^- : \D^{1,p}(\RR^n) \to \D_a^{1,p}(\RR^n,0)
	\]
	via the rule
	\begin{equation}\label{eq:rescale_mod_inv}
		\tau_{y, \lambda}^-v(z) = \begin{cases}
			\displaystyle\lambda^{-\gamma} v\left(\frac{z-y}{\lambda}\right) &\text{if } a = b = 0\\
			\displaystyle
			\left(\frac{\lambda}{\abs{y}}\right)^{-b} \lambda^{-\gamma} v\left(\frac{z-y}{\lambda}\right)\eta\left(\frac{2(z-y)}{\abs{y}}\right)&\text{otherwise}.
		\end{cases}
	\end{equation}
	Obviously, by inspection, it is easy to check that $\tau_{y,\lambda}^-$ is the proper inverse of $\tau_{y,\lambda}$ when $a=b=0$. With this in mind, let us now establish a partial analog of Lemma \ref{lem:rescale_mod_bound}, thereby providing operator bounds related to this pseudoinversion mapping.
	\begin{lem}\label{lem:rescale_mod_inv_bound}
		There exists a constant $C>0$ such that, given a point $y \in \RR^n\setminus\set{0}$ along with some $\lambda > 0$, there holds
		\[
		\begin{aligned}
			\norm{\tau^-_{y,\lambda}v}_{L^q_b(\RR^n,0)} &\le C\norm{v}_{L^q(\RR^n)}\\
			\norm{\tau^-_{y,\lambda}v}_{\D_a^{1,p}(\RR^n,0)} &\le C\left(\frac{\abs{y}}{\lambda}\right)^{p(b-a)}\norm{v}_{\D^{1,p}(\RR^n)}
		\end{aligned}
		\]
		for all $v\in \D^{1,p}(\RR^n)$.
	\end{lem}
	\begin{proof} Once again, the case $a=b=0$ is trivial since it follows from homogeneity. Otherwise, as in the proof of Lemma \ref{lem:rescale_mod_bound} with equation \eqref{eq:ratio bound}, we see that there exists a constant $C>0$ independent of $y, \lambda$ such that
		\begin{equation}\label{eq:ratio bound inv}
			\left(\frac{\abs{\lambda x + y}}{\abs{y}}\right)^{-ap} \le C,\quad
			\left(\frac{\abs{\lambda x + y}}{\abs{y}}\right)^{-bq} \le C
		\end{equation}
		for all $x\in \RR^n$ satisfying $2\lambda \abs{x} \le \abs{y}$. In particular, \eqref{eq:ratio bound inv} is valid for all $x$ such that $2\lambda x/\abs{y}$ is in the support of $\eta$.
		
		For an arbitrary $v\in \D^{1,p}(\RR^n)$, set $u := \tau_{y, \lambda}v$. Using that $0\le\eta\le 1$ we see that
		\begin{align*}
			\norm{u}_{L^q_b(\RR^n,0)}^q
			&=\int_{\RR^n}\abs{u(z)}^q\abs{z}^{-bq}\d{z}\\
			&\le \left(\frac{\lambda}{\abs{y}}\right)^{-bq} \lambda^{-\gamma q}\int_{2\abs{z-y}\le \abs{y}} \abs{v\left(\frac{z-y}{\lambda}\right)}^q\abs{z}^{-bq}\d{z}\\
			&= \int_{2\abs{z-y}\le \abs{y}} \abs{v(x)}^q\left(\frac{\abs{\lambda x + y}}{\abs{y}}\right)^{-bq}\d{x}\\
			&\le C\norm{v}_{L^q(\RR^n)}^q.
		\end{align*}
		Next, in order to bound the $\D_a^{1,p}$ norm of $u$, we once again apply the product rule to obtain two terms:
		\begin{align*}
			\norm{u}_{\D_a^{1,p}(\RR^n,0)}^p
			&=\int_{\RR^n}\abs{\nabla u(z)}^p\abs{z}^{-ap}\d{z}\\
			&\le 2^{p-1}\left(\frac{\lambda}{\abs{y}}\right)^{-b p} \lambda^{-n + ap}\int_{\RR^n} \abs{\eta\left(\frac{2(z-y)}{\abs{y}}\right)\nabla v\left(\frac{z-y}{\lambda}\right)}^p\abs{z}^{-ap}\d{z}\\
			&\quad + 2^{p-1+p}\left(\frac{\lambda}{\abs{y}}\right)^{-b p} \lambda^{-\gamma p}\abs{y}^{-p}\int_{\RR^n} \abs{ v\left(\frac{z-y}{\lambda}\right)\nabla\eta\left(\frac{2(z-y)}{\abs{y}}\right)}^p\abs{z}^{-ap}\d{z}\\
			&=:I_1 + I_2
		\end{align*}
		where we have used that $(\kappa + \iota)^p \le 2^{p-1}\left(\kappa^p + \iota^p\right)$ for all $\kappa,\iota \ge 0$. Then, after a change of variables, observe that
		\begin{align*}
			I_1 &\le 2^{p-1}\left(\frac{\lambda}{\abs{y}}\right)^{-b p} \lambda^{ap}\int_{2\lambda \abs{x} \le \abs{y}} \abs{\nabla v(x)}^p\abs{\lambda x + y}^{-ap}\d{z}\\
			&\le C 2^{p-1}\left(\frac{\abs{y}}{\lambda}\right)^{p(b-a)} \norm{v}_{\D^{1,p}(\RR^n)}^p.
		\end{align*}
		Similarly, with $C_1>0$ such that $2^{1-2p}C_1$ is an upper bound of $\abs{\nabla\eta}^p$ we see that
		\begin{align*}
			I_2 &\le C_1 \left(\frac{\lambda}{\abs{y}}\right)^{-b p} \lambda^{ap + p}\abs{y}^{-p}\int_{2\lambda \abs{x} \le \abs{y}} \abs{ v\left(x\right)}^p\abs{\lambda x + y}^{-ap}\d{x}\\
			&\le C C_1 \left(\frac{\lambda}{\abs{y}}\right)^{p(1+a-b)} \int_{2\lambda \abs{x} \le \abs{y}} \abs{ v\left(x\right)}^p\d{x}.
		\end{align*}
		Then, by H\"older's inequality (with exponent $p^\ast/p$ and H\"older-conjugate exponent $n/p$),
		\begin{align*}
			I_2 &\le C C_1 2^p\left(\frac{\lambda}{\abs{y}}\right)^{p(1+a-b)}\norm{v}_{L^{p^\ast}(\RR^n)}^p\left(\int_{2\lambda\abs{x}\le \abs{y}}1\right)^{p/n}\\
			&\le C\left(\frac{\abs{y}}{\lambda}\right)^{p(b-a)} \norm{v}_{L^{p^\ast}(\RR^n)}^p
		\end{align*}
		up to an adjustment of the constant $C$ in the second inequality. We then conclude our proof via an application of the Gagliardo-Nirenberg-Sobolev inequality (i.e. the CKN inequality \eqref{eq:CKN} with $a=b=0$) and combining our bounds for $I_1$ and $I_2$.
	\end{proof}
	\begin{rem}\label{rem:rescale_mod_inv_bound_set}
		Let now $U \subseteq \RR^n$ be Lebesgue measurable. As was noted for the forward transform, we observe that an identical argument exhibits the existence of a constant $C>0$ (independent of $y, \lambda, U$) such that
		\begin{equation}
			\begin{aligned}
				\norm{\tau^-_{y,\lambda}v}_{L^q_b(U,0)} &\le C\norm{v}_{L^q([U-y]/\lambda)}\\
				\norm{\nabla\left[\tau^-_{y,\lambda}v\right]}_{L^p_a(U,0)} &\le C\left(\frac{\abs{y}}{\lambda}\right)^{p(b-a)}\left(\norm{v}_{L^q([U-y]/\lambda)} + \norm{\nabla v}_{L^p([U-y]/\lambda)}\right)
			\end{aligned}
		\end{equation}
		for all $v\in \D^{1,p}(\RR^n)$.
	\end{rem}
	
	\subsubsection{Transformations and the $L^q$-norm} As was shown, the transform $\tau_{y, \lambda}$ is particularly well behaved with respect to the $L^q$-norm. It turns out that when $\abs{y}/\lambda$ is large, the $L^q$-norm of the transform remains \emph{almost} unchanged on a set where $\eta$ acts as the identity. More precisely, we have the following result:
	
	\begin{lem}\label{lem:unit ball Lq norm}
		There exists a constant $C>0$ such that, for any $u\in \D^{1,p}_a(\RR^n,0)$ and $y\in \RR^n\setminus\set{0}$, $\lambda > 0$ such that $\abs{y} \ge 4\lambda$, the remainder $R$ in
		\[
		\norm{\tau_{y,\lambda} u}_{L^q(B(0,1))}^q = \norm{u}_{L^q_b(B(y,\lambda),0)}^q + R
		\]
		satisfies
		\[
		\abs{R} \le C\frac{\lambda}{\abs{y}}\norm{u}_{L^q_b(B(y,\lambda),0)}^q
		\]
	\end{lem}
	Before presenting the proof, we recall that Lemma \ref{lem:rescale_mod_bound} and \ref{lem:rescale_mod_inv_bound}  relied on $\eta$ in order to restrict our support; then, on this set, we utilized uniform bounds. Estimates of this form will be of repeated use and are thus presented independently as follows:
	\begin{lem}\label{lem:cutoff tools}
		There exists a constant $C>0$ such that
		\begin{enumerate}[label=(\alph*)]
			\item\label{point1} Given $y\in \RR^n\setminus\set{0}$ and $\lambda > 0$,
			\[
			\frac{1}{C} \le \frac{\abs{\lambda x + y}}{\abs{y}} \le C \quad\text{for all } x\in \RR^n \text{ such that } \abs{x} \le \frac{\abs{y}}{2\lambda}.
			\]
			\item\label{point2} Given $y\in \RR^n\setminus\set{0}$, $\lambda > 0$ and any bounded set $K$ such that $\rho := \sup_{x\in K}\abs{x}$ satisfies $4\lambda \rho \le \abs{y}$, there holds
			\[
			\frac{\abs{\abs{y}^{-ap} - \abs{\lambda x + y}^{-ap}}}{\abs{y}^{-ap}} \le C\rho\frac{\lambda}{\abs{y}}, \qquad \frac{\abs{\abs{y}^{-bq} - \abs{\lambda x + y}^{-bq}}}{\abs{y}^{-bq}} \le C\rho\frac{\lambda}{\abs{y}}
			\]
			for all $x\in K$.
		\end{enumerate} 
	\end{lem}
	\begin{proof}
		The first inequality follows from the reverse and usual triangle inequalities. Next, let $y, \lambda, K, \rho$ be as in \ref{point2}. Given an arbitrary $x\in K$ write $z := \lambda x + y$ and observe that
		\[
		\abs{z} \le \abs{y} + \lambda\abs{x} \le \left(1+\frac{\lambda\rho}{\abs{y}}\right)\abs{y}.
		\]
		Using also the reverse triangle inequality, we see that
		\[
		\left(1-\epsilon\right)\abs{y} \le \abs{z} \le \left(1+\epsilon\right)\abs{y}
		\]
		where $\epsilon := \lambda\rho/\abs{y} \le 1/4$. From this, we further deduce that
		\[
		\abs{\abs{y}^{-ap} - \abs{z}^{-ap}} \le \abs{1-\left(1 + \epsilon_\pm\right)^{-ap}}\abs{y}^{-ap}
		\]
		where $\epsilon_{\pm}$ is either $\epsilon$ or $-\epsilon$. It then follows from the Mean Value Theorem that
		\[
		\abs{\abs{y}^{-ap} - \abs{z}^{-ap}} \le \abs{ap}\left(1+\delta\right)^{-ap-1}\epsilon\abs{y}^{-ap}
		\]
		for some $\delta$ in $[-\epsilon,\epsilon]$. From this last equation and recalling that $\abs{\delta}\le\epsilon \le 1/4$, we see that there exists a constant $C>0$ independent of $y, \lambda$ such that
		\[
		\frac{\abs{\abs{y}^{-ap} - \abs{z}^{-ap}}}{\abs{y}^{-ap}} \le C\epsilon,\qquad \forall z\in\lambda K + y.
		\]
		Analogously, modifying $C>0$ if necessary, we may establish that
		\[
		\frac{\abs{\abs{y}^{-bq} - \abs{z}^{-bq}}}{\abs{y}^{-bq}} \le C\epsilon,\qquad \forall z\in\lambda K + y.
		\]
	\end{proof}
	
	With this, we are ready to prove Lemma \ref{lem:unit ball Lq norm}
	\begin{proof}[Proof of Lemma \ref{lem:unit ball Lq norm}]
		Let $u\in \D^{1,p}_a(\RR^n, 0)$, $y\in \RR^n\setminus\set{0}$, and $\lambda > 0$ such that $\abs{y} \ge 4\lambda$. Properties of $\eta$ imply that
		\[
		\int_{B(0,1)}\abs{\tau_{y,\lambda}u}^q\d{x}
		=\lambda^n\int_{B(0,1)}\abs{u\left(\lambda x + y\right)}^q\abs{y}^{-bq}\d{x}
		=\int_{B(y,\lambda)}\abs{u\left(z\right)}^q\abs{y}^{-bq}\d{z},
		\]
		where we have used a change of variables in this last equality. Therefore,
		\[
		\norm{\tau_{y,\lambda} u}_{L^q(B(0,1))}^q - \norm{u}_{L^q_b(B(y,\lambda),0)}^q
		=\int_{B(y,\lambda)}\abs{u\left(z\right)}^q\left(\abs{y}^{-bq}-\abs{z}^{-bq}\right)\d{z}.
		\]
		Then, using both parts of Lemma \ref{lem:cutoff tools}, we infer that there is a constant $C>0$ independent of $y,\lambda, u$ such that
		\begin{align*}
			\abs{\norm{\tau_{y,\lambda} u}_{L^q(B(0,1))}^q - \norm{u}_{L^q_b(B(y,\lambda),0)}^q}
			&\le C\frac{\lambda}{\abs{y}}\int_{B(y, \lambda)}\abs{u(z)}^q\abs{z}^{-bq}\d{z}
		\end{align*}
		which concludes our proof.
	\end{proof}
	
	\subsection{Locality} A useful property of the transforms we introduced is, when $a=b$, the ability to restrict domains. More specifically, we have the following technical result.
	\begin{lem}\label{lem:away from ball}
		Given $v\in \D^{1,p}(\RR^n)$, if $a=b$ then
		\[
		\norm{\tau^-_{y, \lambda}v}_{L_b^q(\RR^n\setminus B(y, \abs{y}/4))} \to 0,\quad\text{and}\quad
		\norm{\nabla \left[\tau^-_{y, \lambda}v\right]}_{L_a^p(\RR^n\setminus B(y, \abs{y}/4))} \to 0
		\]
		as $\abs{y}/\lambda \to\infty$.
	\end{lem}
	\begin{proof}
		This result follows almost immediately from Remark \ref{rem:rescale_mod_inv_bound_set}. Indeed, there is a constant $C>0$ independent of $\lambda, y, v$ such that
		\[
		\norm{\tau^-_{y,\lambda}v}_{L^q_b(\RR^n\setminus B(y, \abs{y}/4))} \le C\norm{v}_{L^q\left(\RR^n\setminus B\left(0, \frac{\abs{y}}{4\lambda}\right)\right)}.
		\]
		Now, since $a=b$ we note that the Sobolev exponent is 
		\[	
		p^\ast = \frac{np}{n-p} = q.
		\]
		Therefore, $v\in L^q(\RR^n)$ and, as $\abs{y}/\lambda \to\infty$, the dominated convergence theorem implies that
		\[
		\norm{\tau^-_{y,\lambda}v}_{L^q_b(\RR^n\setminus B(y, \abs{y}/4))} \le C\norm{v}_{L^q\left(\RR^n\setminus B\left(0, \frac{\abs{y}}{4\lambda}\right)\right)} \to 0.
		\]
		Similarly, since $a=b$, Remark \ref{rem:rescale_mod_inv_bound_set} allows us to write
		\[
		\norm{\nabla \left[\tau^-_{y, \lambda}v\right]}_{L_a^p(\RR^n\setminus B(y, \abs{y}/4))} \le C\left(\norm{v}_{L^q\left(\RR^n\setminus B\left(0, \frac{\abs{y}}{4\lambda}\right)\right)} + \norm{\nabla v}_{L^p\left(\RR^n\setminus B\left(0, \frac{\abs{y}}{4\lambda}\right)\right)}\right)
		\]
		for some $C>0$ independent of $y,\lambda, v$. Once again, the dominated convergence theorem implies that the right-hand side tends to $0$ as $\abs{y}/\lambda \to\infty$.
	\end{proof}
	
	\section{Convergence of Domains} As in the previous sections, we continue to assume that $\Omega\subset \RR^n$ denotes a bounded domain containing the origin. 
Throughout this section, we fix sequences $(y_\alpha)$ in $ \overline{\Om}$ and $\lambda_\alpha $ in $ (0, \infty)$ satisfying some combination of the following properties:
\begin{enumerate}[label = (P\arabic*)]
	\item\label{it:lambda=0} $\lambda_\alpha \to 0$;
	\item\label{it:ratio x_0} The ratio converges in the sense that, for some $x_0\in \RR^n$,
	\[
	\frac{y_\alpha}{\lambda_\alpha} \to x_0.
	\] 
	\item\label{it:ratio inf} The ratio diverges:
	\[
	\frac{\abs{y_\alpha}}{\lambda_\alpha} \to \infty,
	\]
	\item\label{it:concentrates at boundary} The sequence $(y_\alpha)$ approaches the boundary at rate no less than the rescaling factor; i.e. the following ratio is bounded:
	\[
	\frac{\operatorname{dist}(y_\alpha, \partial\Om)}{\lambda_\alpha}.
	\]
	\item\label{it:concentrates in domain} The sequence $(y_\alpha)$ approaches the boundary at a rate slower than the rescaling factor; i.e.
	\[
	\frac{\operatorname{dist}(y_\alpha, \partial\Om)}{\lambda_\alpha} \to \infty.
	\]
\end{enumerate}

The goal of this section is to study the limiting behaviour of specific sequences of sets. First, we describe the sense in which we will consider the convergence of sets.
\begin{Def}
	Let \((U_\alpha)\) be a sequence of subsets of \(\RR^n\). We say that \((U_\alpha)\) converges to \(U\subseteq \RR^n\) and write \(U_\alpha\to U\) provided
	\begin{enumerate}[label=(\alph*)]
		\item\label{dom:1} Any compact subset of $U$ is contained in $U_\alpha$ for all $\alpha$ large.
		\item\label{dom:2} Any compact subset of $\left(\overline{U}\right)^\complement$ is contained in $\left(\overline{U_\alpha}\right)^\complement\subseteq U_\alpha^\complement$ for all $\alpha$ large.
	\end{enumerate}
\end{Def}

Now, suppose condition \ref{it:ratio x_0} holds. Then we consider the following deformations of \(\Omega\):
\[
\Om_\alpha = \frac{\Omega}{\lambda_\alpha} - x_0.
\]
Clearly, these sets are approximations of \((\Om-y_\alpha)/\lambda_\alpha\). Furthermore, we have the following result:
\begin{prop}\label{prop:domain convergence x_0}
	Suppose the sequence \((\lambda_\alpha)\) converges. Then, we have \(\Om_\alpha \to\Om_\infty\) where
	\[
	\Om_\infty
	=
	\begin{cases}
		\RR^n &\text{if condition \ref{it:lambda=0} holds, i.e. }\lambda_\alpha \to 0\\
		\dfrac{\Om}{\lambda}-x_0 &\text{if } \lambda_\alpha \to\lambda > 0.
	\end{cases}
	\]
\end{prop}

\begin{proof}
	We first treat the case where $\lambda_\alpha \to 0$. Since $\Omega$ is an open set containing the origin, we have \(B(0, \delta)\subseteq \Omega\) for some \(\delta > 0\). It readily follows that
	\[
	B\left(-x_0,\frac{\delta}{\lambda_\alpha}\right) \subseteq \Om_\alpha
	\]
	for every index \(\alpha\). Then, for an arbitrary compact set \(K\), using \(\lambda_\alpha \to 0\) we see that
	\[
	K \subseteq B\left(-x_0,\frac{\delta}{\lambda_\alpha}\right) \subseteq \Om_\alpha
	\]
	for all \(\alpha\) large. Hence, $\Om_\alpha \to \RR^n$. 
	
	Let us now assume that we are in the case $\lambda_\alpha \to \lambda > 0$ and fix a compact set $K \subseteq \Om_\infty$. Then, $\lambda(K + x_0)$ is a compact subset of $\Om$. Since $\Om$ is open, the Lebesgue Number Lemma guarantees the existence of $\epsilon > 0$ such that $B(x,\epsilon) \subseteq \Om$ for all $x \in \lambda(K+x_0)$. Put otherwise, for all \(z\in K\) there holds
	\[
	B\left(\frac{\lambda}{\lambda_\alpha}(z+x_0) - x_0, \frac{\epsilon}{\lambda_\alpha}\right) \subseteq \Om_\alpha.
	\]
	On the other hand, since \(K+x_0\) is compact, it is bounded by some \(M>0\). Then, for any \(z\in K\), if \(\alpha\) is such that
	\[
	\abs{\lambda_\alpha - \lambda} < \frac{\epsilon}{M}
	\]
	then
	\[
	\abs{z - \left(\frac{\lambda}{\lambda_\alpha}(z+x_0) - x_0\right)} = \frac{1}{\lambda_\alpha}\abs{\lambda-\lambda_\alpha}\abs{z+x_0} 
	< 
	\frac{\epsilon}{\lambda_\alpha}.
	\]
	Since \(\lambda_\alpha\to\lambda>0\), it follows that for all \(\alpha\) sufficiently large we have
	\[
	K \subseteq B\left(\frac{\lambda}{\lambda_\alpha}(z+x_0) - x_0, \frac{\epsilon}{\lambda_\alpha}\right) \subseteq \Om_\alpha.
	\]	
	Finally, given a compact set $K \subseteq \left( \overline{\Om_\infty} \right)^\complement$, the same argument shows that \(K\subseteq \overline{\Omega_\alpha}^\complement\) for all \(\alpha\) large. We once again conclude that \(\Om_\alpha \to \Om_\infty\).
\end{proof}

Next, suppose condition \ref{it:ratio inf} holds. Then we will be interested in the following deformations of \(\Omega\):
\[
\Om_\alpha = \frac{\Omega - y_\alpha}{\lambda_\alpha}.
\]
In this case, if condition \ref{it:lambda=0} holds then we have the following result:

\begin{prop}\label{prop:domain convergence}
	Suppose $\Omega$ has $C^1$ boundary. If \ref{it:lambda=0} holds, i.e. $\lambda_\alpha \to 0$, then, up to a subsequence, $\Om_\alpha := (\Om-y_\alpha)/\lambda_\alpha$ tends to a set $\Om_\infty\subseteq \RR^n$.
	\begin{enumerate}
		\item If the distance between $y_\alpha$ and the boundary tends to $0$ no slower than $(\lambda_\alpha)$, i.e. if condition \ref{it:concentrates at boundary} holds, then $\Om_\infty$ is a half-space.
		\item On the other hand, if condition \ref{it:concentrates in domain} holds then $\Om_\infty = \RR^n$.
	\end{enumerate}
\end{prop}
This result follows from a straightforward adaption of El-Hamidi-V\'etois \cite[Proposition 2.4]{Vetois3}. 

\begin{lem}\label{lem:convergence of functions on domains}
	Suppose $\Om_\alpha \to \Om_\infty$, with $\Om_\infty$ being either $\RR^n$ or a half space. Let $(v_\alpha)$ in $ \D^{1,p}(\RR^n)$ be a sequence converging to some $v\in \D^{1,p}(\RR^n)$ pointwise almost everywhere. If $v_\alpha \in \D^{1,p}(\Om_\alpha)$ for each $\alpha$ then $v\in \D^{1,p}(\Om_\infty)$.
\end{lem}
\begin{proof}
	If $\Om_\infty = \RR^n$ there is nothing to show. Otherwise, up to a translation, there exists a unit vector $\nu$ such that
	\[
	\Om_\infty = \set{x\in \RR^n : x\cdot \nu > 0}.
	\]
	Now, observe that $\phi_\epsilon(x) := v(x-\epsilon\nu)$ converges to $v$ in $\D^{1,p}(\RR^n)$ as $\epsilon\searrow 0$ by standard results. Therefore, if we can show that $\phi_\epsilon \in \D^{1,p}(\Om_\infty)$ for each $\epsilon > 0$ then we must have $v\in \D^{1,p}(\Om_\infty)$ which completes the proof.
	
	To this end, we first show that $v$ vanishes outside of $\overline{\Om_\infty}$ by showing that $v$ vanishes on every compact set $K$ such that $K\cap\overline{\Om_\infty} = \varnothing$. Indeed, if $K$ is compact and does not intersect $\overline{\Om_\infty}$ then, since $\Om_\alpha \to \Om_\infty$, we see that $K \cap \overline{\Om_\alpha} = \varnothing$ for all $\alpha$ large. Thus, $v_\alpha$ vanishes on $K$ for each $\alpha$ large. Since $v_\alpha \to v$ pointwise almost everywhere, it follows that $v$ vanishes on $K$. 
	
	Having shown that $v$ is supported on $\overline{\Om_\infty}$, it readily follows that, every $\phi_\epsilon$ is supported on
	\[
	\set{x\in \RR^n : x\cdot \nu \ge \epsilon}.
	\]
	Now, in order to prove that $\phi_\epsilon \in \D^{1,p}(\Om_\infty)$ we show that, for any $\delta > 0$, we may find a function in $C_c^\infty(\Om_\infty)$ whose $\D^{1,p}$-distance to $\phi_\epsilon$ is less than $\delta$. With this in mind, let $\delta > 0$ be given and set $\zeta\in C_c^\infty(\RR^n; [0,1])$ to be a smooth function of compact support with $\zeta\equiv 1$ on $B(0, 1)$. Then, for any $h_1>0$, define
	\[
	\zeta_{h_1}(x):= \zeta(h_1x)
	\]
	so that $\phi_\epsilon\zeta_{h_1} \in \D^{1,p}(\RR^n)$. By the product rule and using that $(\iota + \kappa)^p \le 2^{p-1}(\iota^p + \kappa^p)$ for all $\iota, \kappa \ge 0$,
	\begin{equation}\label{eq:cutoff}
		\norm{\phi_\epsilon - \zeta_{h_1}\phi_\epsilon}^p_{\D^{1,p}(\RR^n)}\\
		\le 2^{p-1} \int_{\abs{x}\ge 1/h_1} \abs{\nabla \phi_\epsilon}^p\d{x} + 2^{p-1} \int_{\abs{x}\ge 1/h_1}  \abs{\phi_\epsilon\nabla\zeta_{h_1}}^p\d{x}.
	\end{equation}
	The first term on the right-hand side tends to $0$ as $h_1\searrow 0$. As for the second term, by H\"older's inequality with exponent $q/p$ and conjugate exponent $n/[p(1+a-b)]$ we may write
	\begin{align*}
		\int_{\abs{x}\ge h_1}  \abs{\phi_\epsilon\nabla\zeta_{h_1}}^p\d{x}
		&\le \left(\int_{\abs{x} \ge 1/h_1} \abs{ \phi_\epsilon}^q\d{x}\right)^{p/q}\left(\int_{\abs{x}\ge 1/h_1} \abs{\nabla\zeta_{h_1}}^{\frac{n}{1+a-b}}\d{x}\right)^{p(1+a-b)/n}\\
		&=\left(\int_{\abs{x} \ge 1/h_1} \abs{ \phi_\epsilon}^q\d{x}\right)^{p/q}\left(h_1^{\frac{n}{1+a-b}-n}\int_{\abs{y}\ge 1} \abs{\nabla\zeta}^{\frac{n}{1+a-b}}\d{y}\right)^{p(1+a-b)/n}
	\end{align*}
	by a change of variables. Thus, we see that the right-hand side tends to $0$ as $h_1\searrow 0$. Assembling, we deduce that both terms on the right-hand side of equation \eqref{eq:cutoff} tend to $0$ as $h_1\searrow 0$. Thus, we may select $h_1>0$ sufficiently small so that
	\[
	\norm{\phi_\epsilon - \zeta_{h_1}\phi_\epsilon}_{\D^{1,p}(\RR^n)} < \frac{\delta}{2}.
	\]
	Now, for any $0<h \ll \epsilon$, observe that the mollification $\left(\zeta_{h_1}\phi_\epsilon\right)_h$ is in $C_c^\infty(\Om_\infty)$. Then, by standard mollification results,
	we may select $h>0$ sufficiently small so that
	\[
	\norm{\zeta_{h_1}\phi_\epsilon - (\zeta_{h_1}\phi_\epsilon)_h}_{\D^{1,p}(\RR^n)} < \frac{\delta}{2}.
	\]
	By the triangle inequality, $\norm{\phi_\epsilon-(\zeta_{h_1}\phi_\epsilon)_h}_{\D^{1,p}(\RR^n)} <\delta$ and we may conclude that $\phi_\epsilon \in \D^{1,p}(\Om_\infty)$, as desired.
\end{proof}

\section{Energy and the New Transforms}	We continue to assume that $\Omega\subset \RR^n$ denotes a bounded domain containing the origin and fix sequences $(y_\alpha)$ in $ \overline{\Om}$ and $\lambda_\alpha $ in $ (0, \infty)$. As in the previous section, the sequences \((y_\alpha), (\lambda_\alpha)\) will satisfy some combination of properties \ref{it:lambda=0}-\ref{it:concentrates in domain}.

In this section, we investigate how the new transforms \(\tau_{y, \lambda}, \tau_{y, \lambda}^-\) interact with the energy functionals \(\phi\) and \(\phi_\infty\). First, observe that if $a=b=0$, a change of variables shows that
\[
\inner{\phi_\infty^\prime(\tau_{y, \lambda}u), h} = \inner{\phi^\prime(u), \tau_{y,\lambda}^-h}.
\]
for any $h\in \D^{1,p}(\RR^n)$. More generally, when $a=b$, the difference between terms on either side of the equality will be negligible provided $\abs{y}/\lambda$ is large. This is made precise via the following result:

\begin{lem}\label{lem:transform energy} 
	If $a=b$, there exists a constant $C>0$ such that the remainder
	\[
	R:=\inner{\phi_\infty^\prime(\tau_{y, \lambda}u), h} - \inner{\phi^\prime(u), \tau_{y,\lambda}^-h}
	\]
	satisfies
	\[
	\abs{R} 
	\le C\rho\frac{\lambda}{\abs{y}}\left(\norm{u}_{\D_a^{1,p}(\Om, 0)}^{p-1}\norm{h}_{\D^{1,p}(\RR^n)} + \norm{u}_{L_b^q(\Om, 0)}^{q-1}\norm{h}_{L^q(\RR^n)}\right)
	\]
	for any $u\in \D_a^{1,p}(\Om, 0)$, $y \in \RR^n\setminus\set{0}$, $\lambda > 0$ and $h \in \D^{1,p}(\RR^n)$ with support $K$ satisfying
	\begin{equation}\label{eq:K cond}
		\rho := \sup_{x\in K}\abs{x} \le \frac{\abs{y}}{4\lambda},\qquad
		K \subseteq \frac{\Om-y}{\lambda}.
	\end{equation}
	In particular, suppose $\Om_\infty\subseteq \RR^n$ is such that \(\frac{\Om-y_\alpha}{\lambda_\alpha} =: \Om_\alpha \to \Om_\infty\). Let \((u_\alpha)\) be a bounded sequence in \(\D_a^{1,p}(\Om, 0)\) and assume property \ref{it:ratio inf} holds (i.e. \(\abs{y_\alpha}/\lambda_\alpha \to \infty\)). Then, given a fixed compact set \(K\subseteq \Om_\infty\),
	\[
	\inner{\phi_\infty^\prime(\tau_{y_\alpha, \lambda_\alpha}u_\alpha), h} - \inner{\phi^\prime(u_\alpha), \tau_{y_\alpha,\lambda_\alpha}^-h} \to 0
	\]
	uniformly over $h \in \set{\D^{1,p}(\RR^n): \supp(h) \subseteq K}$ as $\alpha\to\infty$.
\end{lem}

\begin{proof}
	Let $u\in \D_a^{1,p}(\Om, 0)$, $y \in \RR^n\setminus\set{0}$, $\lambda > 0$ and suppose $h \in \D^{1,p}(\RR^n)$ has support $K$ satisfying \eqref{eq:K cond}. First, notice that
	\[
	R = \inner{\phi_\infty^\prime(\tau_{y, \lambda}u), h} - \inner{\phi^\prime(u), \tau_{y,\lambda}^-h}
	\]
	is well defined since $\lambda K + y\subseteq \Om$ implies that $\tau_{y,\lambda}^-h \in \D_a^{1,p}(\Om, 0)$. Moreover, because $\rho\le \abs{y}/(4\lambda)$, any $x\in K = \operatorname{supp}(h)$ satisfies
	\begin{equation}\label{eq:eta=1}
		\abs{\frac{2\lambda x}{\abs{y}}} \le \frac{1}{2}\implies \eta\left(\frac{2\lambda x}{\abs{y}}\right) = 1
	\end{equation}
	Using this , we see that
	\begin{align*}
		\inner{\phi_\infty^\prime(\tau_{y, \lambda}u), h}
		&=\left(\frac{\lambda}{\abs{y}}\right)^{b(p-1)}\lambda^{(\gamma+1)(p-1)}\int_{\RR^n}\abs{\nabla u(\lambda x + y)}^{p-2} \pinner{\nabla u(\lambda x + y), \nabla h(x)}\d{x}\\
		&\quad -\left(\frac{\lambda}{\abs{y}}\right)^{b(q-1)}\lambda^{\gamma(q-1)}\int_{\RR^n}\abs{ u(\lambda x + y)}^{q-2} u(\lambda x + y)  h(x)\d{x}\\
		&=\left(\frac{\lambda}{\abs{y}}\right)^{b(p-1)}\lambda^{(\gamma+1)(p-1)-n}\int_{\RR^n}\abs{\nabla u(z)}^{p-2}\pinner{\nabla u(z), \nabla h\left(\frac{z-y}{\lambda}\right)}\d{z}\\
		&\quad -\left(\frac{\lambda}{\abs{y}}\right)^{b(q-1)}\lambda^{\gamma(q-1)-n}\int_{\RR^n}\abs{ u(z)}^{q-2} u(z)  h\left(\frac{z-y}{\lambda}\right)\d{z}.
	\end{align*}
	Using equation \eqref{eq:eta=1}, but with $x=(z-y)/\lambda$, yields that
	\begin{align*}
		\inner{\phi_\infty^\prime(\tau_{y, \lambda}u), h}
		&=\left(\frac{\lambda}{\abs{y}}\right)^{bp}\lambda^{(\gamma+1)p-n}\int_{\RR^n}\abs{\nabla u(z)}^{p-2}\pinner{\nabla u(z), \nabla \left[\tau^-_{y, \lambda}h\right](z)}\d{z}\\
		&\quad -\left(\frac{\lambda}{\abs{y}}\right)^{bq}\lambda^{\gamma q-n}\int_{\RR^n}\abs{ u(z)}^{q-2} u(z)  \left[\tau^-_{y, \lambda}h\right](z)\d{z}
	\end{align*}
	where $u$ is extended by $0$ outside of $\Om$. Simplifying and re-arranging, we obtain
	\begin{align*}
		\inner{\phi_\infty^\prime(\tau_{y, \lambda}u), h}
		&=\left(\frac{\lambda}{\abs{y}}\right)^{p(b-a)}\int_{\RR^n}\abs{\nabla u(z)}^{p-2}\pinner{\nabla u(z), \nabla \left[\tau^-_{y, \lambda}h\right](z)}\abs{y}^{-ap}\d{z}\\
		&\quad -\int_{\RR^n}\abs{ u(z)}^{q-2} u(z) \left[\tau^-_{y, \lambda}h\right](z) \abs{y}^{-bq}\d{z}
	\end{align*}
	With this, observe that since $a=b$ 
	\begin{align*}
		R
		:=&\inner{\phi_\infty^\prime(\tau_{y, \lambda}u), h} - \inner{\phi^\prime(u), \tau_{y,\lambda}^-h}\\
		&=\int_{\RR^n}\abs{\nabla u(z)}^{p-2}\pinner{\nabla u(z), \nabla \left[\tau^-_{y, \lambda}h\right](z)}\left(\abs{y}^{-ap} - \abs{z}^{-ap}\right)\d{z}\\
		&\quad -\int_{\RR^n}\abs{ u(z)}^{q-2} u(z) \left[\tau^-_{y, \lambda}h\right](z) \left(\abs{y}^{-bq}-\abs{z}^{-bq}\right)\d{z}.
	\end{align*}
	Then, using Lemma \ref{lem:cutoff tools}, we see that there is a constant $C>0$ independent of $y,\lambda, u, K$ such that
	\begin{align*}
		\abs{R} 
		&\le C\epsilon \int_{\RR^n}\abs{\nabla u(z)}^{p-1}\abs{\nabla \left[\tau^-_{y, \lambda}h\right](z)}\abs{z}^{-ap}\d{z}\\
		&\quad + C\epsilon \int_{\RR^n}\abs{ u(z)}^{q-1} \abs{\left[\tau^-_{y, \lambda}h\right](z)} \abs{z}^{-bq}\d{z}\\
		&\le C\epsilon\left(\norm{u}_{\D_a^{1,p}(\Om,0)}^{p-1}\norm{\tau_{y,\lambda}^-h}_{\D_a^{1,p}(\Om,0)} + \norm{u}_{L_b^q(\Om,0)}^{q-1}\norm{\tau_{y,\lambda}^-h}_{L_b^q(\Om,0)}\right)
	\end{align*}
	where $\epsilon := \lambda\rho/\abs{y}$ and the second inequalities follows from H\"older's inequality. We note that, for the last line, we have implicitly used that $u$ and $\tau_{y,\lambda}^-h$ are in $\D_a^{1,p}(\Om, 0)$ and extended by $0$ to the rest of $\RR^n$. Then, by Lemma \ref{lem:rescale_mod_inv_bound}, with a possible modification of the constant $C>0$ we deduce that
	\[
	\abs{R} 
	\le C\epsilon\left(\norm{u}_{\D_a^{1,p}(\Om,0)}^{p-1}\norm{h}_{\D^{1,p}(\RR^n)} + \norm{u}_{L_b^q(\Om,0)}^{q-1}\norm{h}_{L^q(\RR^n)}\right)
	\]
	where we have once again used that $a=b$. Finally,the second part of the Lemma follows from this last inequality and the CKN-inequality \eqref{eq:CKN}.
\end{proof}

In this last lemma, we studied how the composition $\phi_\infty^\prime\circ \tau_{y, \lambda}$ behaves by testing against a fixed function $h$. In a similar vein, with the aims of recovering a \((PS)\)-sequence after subtracting bubbles, let us now analyze the behaviour of $\phi_\infty^\prime$ when tested against transformed functions:

\begin{lem}\label{lem:v energy transform g}
	Suppose $\Om_\alpha = (\Om-y_\alpha)/\lambda_\alpha$ tends to $\Om_\infty$, with $\Om_\infty$ being either a half-space of all of $\RR^n$. If $a=b$, \ref{it:ratio inf} holds and $v\in \D^{1,p}(\Om_\infty)$ is a critical point of $\phi_\infty$ in $\Om_\infty$, i.e. 
	\(
	\inner{\phi^\prime_\infty(v), h} = 0\) for all \(h\in \D^{1,p}(\Om_\infty),
	\) 
	then
	\[
	\inner{\phi^\prime_\infty(v), \tau_{y_\alpha, \lambda_\alpha}g} = o\left(\norm{g}_{\D_a^{1,p}(\Om,0)}\right) \quad\text{as}\quad\alpha\to\infty.
	\]
	over $g\in \D_a^{1,p}(\Om,0)$.
\end{lem}
\begin{proof}
	By way of contradiction, suppose the lemma is false. Then there exists $\epsilon>0$ and a subsequence such that, for each $\alpha$, we may select $g_\alpha \in \D_a^{1,p}(\Om ,0)$ satisfying
	\[
	\abs{\inner{\phi_\infty^\prime(v), \tau_{y_\alpha, \lambda_\alpha}g_\alpha}} \ge \epsilon\norm{g_\alpha}_{\D_a^{1,p}(\Om,0)}
	\]
	Relabeling each $g_\alpha$ to be the rescaled $g_\alpha/\norm{g_\alpha}_{\D_a^{1,p}(\Om,0)}$, we obtain
	\begin{equation}\label{eq:3}
		\abs{\inner{\phi_\infty^\prime(v), \tau_{y_\alpha, \lambda_\alpha}g_\alpha}} \ge \epsilon
	\end{equation}
	for all $\alpha$ where $g_\alpha \in \D_a^{1,p}(\Om ,0)$ has unit norm.
	
	On the other hand, since $a=b$, it follows from Lemma \ref{lem:rescale_mod_bound} that $\tau_{y_\alpha, \lambda_\alpha}g_\alpha$ is a bounded sequence in $\D^{1,p}(\RR^n)$. Therefore, a subsequence of $(\tau_{y_\alpha, \lambda_\alpha}g_\alpha)$ converges weakly in $\D^{1,p}(\RR^n)$ and pointwise almost everywhere to some function $h\in \D^{1,p}(\RR^n)$. By Lemma \ref{lem:convergence of functions on domains}, in fact, $h\in \D^{1,p}(\Om_\infty)$. Then, by weak convergence and using that $v$ is a critical point of $\phi_\infty$ in $\Om_\infty$, we have
	\[
	\inner{\phi_\infty^\prime(v), \tau_{y_\alpha, \lambda_\alpha}g_\alpha} \to \inner{\phi_\infty^\prime(v), h} = 0
	\] 
	However, this contradict equation \eqref{eq:3}.
\end{proof}

\subsection{Locality} Due to the nature of the transform \(\tau_{y, \lambda}\), restricting domains may have negligible effect on norms (recall Lemma \ref{lem:away from ball}). In relation to the energy functional \(\phi_\infty\), a similar phenomenon can be observed:
\begin{lem}\label{lem:concentration limiting energy}
	Fix $v\in \D^{1,p}(\RR^n)$. If $a=b$ and property \ref{it:ratio inf} holds (i.e. \(\abs{y_\alpha}/\lambda_\alpha \to \infty\)) then 
	\[
	\begin{aligned}
		R_g(y_\alpha, \lambda_\alpha)
		:=\inner{\phi_\infty^\prime(v), \tau_{y_\alpha, \lambda_\alpha}g} 
		-\Bigg(& \int_{\abs{x} \le \frac{\abs{y_\alpha}}{4\lambda_\alpha}}\abs{\nabla v}^{p-2}\pinner{\nabla v,\nabla\left[\tau_{y_\alpha, \lambda_\alpha}g\right]}\left(\frac{\abs{\lambda_\alpha x + y_\alpha}}{\abs{y_\alpha}}\right)^{-ap}\d{x}\\
		&\qquad - \int_{\abs{x} \le \frac{\abs{y_\alpha}}{4\lambda_\alpha}}\abs{v}^{q-2} v\left[\tau_{y_\alpha, \lambda_\alpha}g\right]\left(\frac{\abs{\lambda_\alpha x + y_\alpha}}{\abs{y_\alpha}}\right)^{-bq}\d{x}\Bigg)
	\end{aligned}
	\]
	tends to \(0\) uniformly over $g\in \D_a^{1,p}(\Om,0)$ as $\alpha \to\infty$.
\end{lem}
\begin{proof}
	For any $g\in \D_a^{1,p}(\Om,0)$, observe that
	\begin{align*}
		R_g(y_\alpha, \lambda_\alpha)
		&=\int_{\abs{x} \le \frac{\abs{y_\alpha}}{4\lambda_\alpha}}\abs{\nabla v}^{p-2}\pinner{\nabla v,\nabla\left[\tau_{y_\alpha, \lambda_\alpha}g\right]}\left(\frac{\abs{y_\alpha}^{-ap} - \abs{\lambda_\alpha x + y_\alpha}^{-ap}}{\abs{y_\alpha}^{-ap}}\right)\d{x}\\
		&\quad+\int_{\abs{x} > \frac{\abs{y_\alpha}}{4\lambda_\alpha}}\abs{\nabla v}^{p-2}\pinner{\nabla v,\nabla\left[\tau_{y_\alpha, \lambda_\alpha}g\right]}\d{x}\\
		&\qquad - \int_{\abs{x} \le \frac{\abs{y_\alpha}}{4\lambda_\alpha}}\abs{v}^{q-2} v\left[\tau_{y_\alpha, \lambda_\alpha}g\right]\left(\frac{\abs{y_\alpha}^{-bq} - \abs{\lambda_\alpha x + y_\alpha}^{-bq}}{\abs{y_\alpha}^{-bq}}\right)\d{x}\\
		&\qquad\quad - \int_{\abs{x} > \frac{\abs{y_\alpha}}{4\lambda_\alpha}}\abs{v}^{q-2} v\left[\tau_{y_\alpha, \lambda_\alpha}g\right]\d{x}\\
		&=: I_{a}^{(1)} + I_a^{(2)} - I_b^{(1)} - I_b^{(2)}.
	\end{align*}
	For all \(\alpha\) sufficiently large (so that$\abs{y_\alpha}/\lambda_\alpha > 16$), we may split
	\begin{align*}
		I_a^{(1)}
		&=\int_{\abs{x} \le \sqrt{\frac{\abs{y_\alpha}}{\lambda_\alpha}}}\abs{\nabla v}^{p-2}\pinner{\nabla v,\nabla\left[\tau_{y_\alpha, \lambda_\alpha}g\right]}\left(\frac{\abs{y_\alpha}^{-ap} - \abs{\lambda_\alpha x + y_\alpha}^{-ap}}{\abs{y_\alpha}^{-ap}}\right)\d{x}\\
		&\quad + \int_{\sqrt{\frac{\abs{y_\alpha}}{\lambda_\alpha}}<\abs{x} \le \frac{\abs{y_\alpha}}{4\lambda_\alpha}}\abs{\nabla v}^{p-2}\pinner{\nabla v,\nabla\left[\tau_{y_\alpha, \lambda_\alpha}g\right]}\left(\frac{\abs{y_\alpha}^{-ap} - \abs{\lambda_\alpha x + y_\alpha}^{-ap}}{\abs{y_\alpha}^{-ap}}\right)\d{x}\\
		&=: I_a^{(1.1)} + I_a^{(1.2)}
	\end{align*}
	It follows from Lemma \ref{lem:cutoff tools}-\ref{point2} with $K = B\left(0, \sqrt{\abs{y_\alpha}/\lambda_\alpha}\right)$ that, for a constant $C>0$ independent of $\alpha$, there holds
	\begin{align*}
		\abs{I_a^{(1.1)}} 
		&\le C\sqrt{\frac{\abs{y_\alpha}}{\lambda_\alpha}}\frac{\lambda_\alpha}{\abs{y_\alpha}}\int_{\abs{x} \le \sqrt{\frac{\abs{y_\alpha}}{\lambda_\alpha}}}\abs{\nabla v}^{p-1}\abs{\nabla\left[\tau_{y_\alpha, \lambda_\alpha}g\right]}\d{x}\\
		&\le C\sqrt{\frac{\lambda_\alpha}{\abs{y_\alpha}}}\norm{\tau_{y_\alpha, \lambda_\alpha}g}_{\D^{1,p}(\RR^n)}\left(\int_{\abs{x} \le \sqrt{\frac{\abs{y_\alpha}}{\lambda_\alpha}}}\abs{\nabla v}^{p}\d{x}\right)^{(p-1)/p}.
	\end{align*}
	by H\"older's inequality. Applying Lemma \ref{lem:rescale_mod_bound} with $a=b$ and using that $v\in \D^{1,p}(\RR^n)$ is fixed we have, up to a modification of the constant $C$,
	\begin{equation*}
		\abs{I_a^{(1.1)}} \le C\sqrt{\frac{\lambda_\alpha}{\abs{y_\alpha}}}\norm{g}_{\D^{1,p}_a(\Om)} = o\left(\norm{g}_{\D^{1,p}_a(\Om,0)}\right)
	\end{equation*}
	as $\alpha \to \infty$. By the same procedure applied to $I_a^{(1.2)}$ we instead obtain
	\begin{align*}
		\abs{I_a^{(1.2)}} 
		&\le C\norm{g}_{\D^{1,p}_a(\Om)}\left(\int_{\sqrt{\frac{\abs{y_\alpha}}{\lambda_\alpha}}<\abs{x} \le \frac{\abs{y_\alpha}}{4\lambda_\alpha}}\abs{\nabla v}^{p}\d{x}\right)^{(p-1)/p}\\
		&= o\left(\norm{g}_{\D^{1,p}_a(\Om)}\right)
	\end{align*}
	by the dominated convergence theorem. Combining, we deduce that
	\begin{equation}\label{eq:I_a^1}
		I_a^{(1)} = o\left(\norm{g}_{\D^{1,p}_a(\Om)}\right)
	\end{equation}
	Similarly, using also the CKN inequality \eqref{eq:CKN}, we have
	\begin{equation}\label{eq:I_b^1}
		\abs{I_b^{(1)}}  = o\left(\norm{g}_{\D^{1,p}_a(\Om)}\right)
	\end{equation}
	Moreover, using H\"older's inequality, Lemma \ref{lem:rescale_mod_bound}, and the CKN inequality \eqref{eq:CKN}, we also have
	\[
	I_a^{(2)} + I_b^{(2)} = o\left(\norm{g}_{\D_a^{1,p}(\Om)}\right)\qquad \text{as } \alpha \to \infty.
	\]
	by the dominated convergence theorem. Combining the last three equations concludes our proof.
\end{proof}

\section{Iterative results and Solution Extraction}
Maintaining the notation of the previous section, $\Om\subset \RR^n$ denotes a bounded domain containing the origin and \(y_\alpha\) in \(\overline{\Om}\), \(\lambda_\alpha\) in \((0, \infty)\) are sequences that will be assumed to satisfy some combination of properties \ref{it:lambda=0}-\ref{it:concentrates in domain}.

\subsection{Extracting a Solution}

To establish Theorems \ref{thm:main a<b}-\ref{thm:main a=b}, the first step is to extract a (possibly trivial) solution to \eqref{eq:probOm} from a given Palais-Smale sequence. We recall that $(u_\alpha)$ in $ \D_a^{1,p}(\Om,0)$ is called a $(PS)$-sequence for problem \eqref{eq:probOm} provided it has bounded energy (meaning $\phi(u_\alpha)$ is bounded) and $\phi^\prime(u_\alpha)\to 0$ strongly in $\D_a^{-1, p^\prime}(\Om, 0)$. The following result borrows ideas from the unweighted proof presented by Mercuri-Willem \cite[Lemma 3.5]{mercuri-willem}.

\begin{lem}\label{lem:extraction}
	Let $(u_\alpha)$ in $\D^{1,p}_a(\Om, 0)$ be a Palais-Smale sequence for problem \eqref{eq:probOm} such that,
	\begin{enumerate}[label=(\roman*)]
		\item $u_\alpha \rightharpoonup u$ in $\D^{1,p}_a(\Om, 0)$;
		\item $u_\alpha \to u$ a.e. on $\Om$;
		\item $\phi(u_\alpha) \to c$
	\end{enumerate}
	as $\alpha \to \infty$. Passing to a subsequence if necessary, there holds
	\begin{enumerate}[label=(\alph*)]
		\item\label{it:esa} $\nabla u_\alpha \to \nabla u$ a.e. on $\Om$ and $\phi^\prime(u) = 0$;
		\item\label{it:esb} $\norm{u_\alpha}^p_{\D^{1,p}_a(\Om, 0)} - \norm{u_\alpha - u}^p_{\D^{1,p}_a(\Om, 0)} = \norm{u}^p_{\D^{1,p}_a(\Om, 0)} + o(1)$;
		\item\label{it:esc} $\phi(u_\alpha - u) \to c - \phi(u)$ as $\alpha \to \infty$.
		\item\label{it:esd} $\phi^\prime(u_\alpha-u)\to 0$ strongly in  $\D^{-1,p^\prime}_a(\Om, 0)$.
	\end{enumerate}
\end{lem}

\begin{proof}

	Let us define \(T\) by equation the equation \eqref{eq:TDef}.
	Using that \((u_\alpha)\) is a bounded sequence in \(\D_a^{1,p}(\Om, 0)\), we see that \((T(u_\alpha -u))\) is also a bounded sequence in \(\D_a^{1,p}(\Om, 0)\). In particular, up to a subsequence, \((T(u_\alpha -u))\) converges weakly in \(\D_a^{1,p}(\Om, 0)\) and pointwise almost everywhere to some function in \(\D_a^{1,p}(\Om, 0)\). Since \(u_\alpha \to u\) pointwise a.e., this function must be the trivial function. That is, \(T(u_\alpha -u) \rightharpoonup 0\) weakly in \(\D_a^{1,p}(\Om, 0)\) and pointwise almost everywhere.
	
	With this, let us write
	\begin{align*}
		\Gamma_\alpha
		&:=\int_{\Om} \pinner{\left[\abs{\nabla u_\alpha}^{p-2} \nabla u_\alpha - \abs{\nabla u}^{p-2} \nabla u\right], \nabla T(u_\alpha - u)}\abs{x}^{-ap}\d{x}\\
		&= \inner{\phi^\prime(u_\alpha), T(u_\alpha-u)} + \int_{\Om} \pinner{\abs{u_\alpha}^{q-2} u_\alpha, T(u_\alpha - u)}\abs{x}^{-bq}\d{x}\\
		&\qquad - \int_{\Om} \pinner{\abs{\nabla u}^{p-2} \nabla u, \nabla T(u_\alpha - u)}\abs{x}^{-ap}\d{x}.
	\end{align*} 
	Since \(\phi^\prime(u_\alpha)\to 0\) strongly in the dual of \(\D_a^{1,p}(\Om, 0)\) and \((T(u_\alpha -u))\) is bounded in \(\D_a^{1,p}(\Om, 0)\), the first term after the last equality tends to \(0\). Next, observe that, \(\norm{T(u_\alpha-u)}_{L_b^q(\Om, 0)} \to 0\) by the dominated convergence theorem and, by the CKN-inequality, \((u_\alpha)\) in \( L_b^q(\Om, 0)\) is a bounded sequence. Therefore, an application of H\"older's inequality shows that the second term also tends to \(0\). Finally, the third term tends to \(0\) since \((T(u_\alpha -u)) \rightharpoonup 0\) weakly in \(\D_a^{1,p}(\Om, 0)\).
	
	We conclude that \(\Gamma_\alpha \to 0\) as \(\alpha\to\infty\). Thus, we may apply Corollary \ref{cor:gradientConvergence} to see that, up to a subsequence,
	\begin{enumerate}
		\item $\nabla u_\alpha \to \nabla u$ almost everywhere on $\Om$,
		\item $\norm{u_\alpha}_{\D_a^{1,p}(\Om, 0)}^p - \norm{u_\alpha -u}_{\D_a^{1,p}(\Om, 0)}^p = \norm{u}_{\D_a^{1,p}(\Om ,0)}^p + o(1)$,
		\item\label{conc:es3} and
		\[
		\abs{\nabla u_\alpha}^{p-2}\nabla u_\alpha - \abs{\nabla u_\alpha - \nabla u}^{p-2}\left(\nabla u_\alpha - \nabla u\right) \to \abs{\nabla u}^{p-2}\nabla u
		\]
		strongly in $L^{p^\prime}(\Om, \abs{x}^{-ap})$, where $p^\prime = p/(p-1)$,
		\item\label{conc:es4} and
		\[
		\abs{u_\alpha}^{q-2} u_\alpha - \abs{ u_\alpha -  u}^{q-2}\left( u_\alpha -  u\right) \to \abs{ u}^{q-2} u
		\]
		strongly in $L^{q/(q-1)}(\Om, \abs{x}^{-bq})$.
	\end{enumerate}
	as $\alpha \to \infty$. In particular, we have established \ref{it:esa} and \ref{it:esb}.
	
	Conclusion \ref{it:esc} follows from the Br\'ezis-Lieb lemma. Indeed,
	\begin{align*}
		\phi(u_\alpha-u) 
		&= \frac{1}{p} \norm{u_\alpha - u}^p_{\D^{1,p}_a(\Om, 0)} - \frac{1}{q} \norm{u_\alpha - u}^q_{L^q(\Om, \abs{x}^{-bq})}\\
		&= \phi(u_\alpha) - \phi(u) + o(1)\\
		& \to c - \phi(u).
	\end{align*}
	
	Next, let $h \in \D^{1,p}_a(\Om, 0)$ be arbitrary. By H\"older's inequality followed by \eqref{conc:es3} and \eqref{conc:es4}, we have
	\[
	\inner{\phi^\prime(u_\alpha-u) - \phi^\prime(u_\alpha) + \phi^\prime(u), h} = o\left(\norm{h}_{ \D^{1,p}_a(\Om, 0)}\right)
	\]
	as $\alpha\to \infty$. On the other hand, by standard convergence results in Lebesgue spaces,
	\[
	\inner{\phi^\prime(u), h} =\lim_{\alpha\to\infty} \inner{\phi^\prime(u_\alpha), h} = 0
	\]
	where the second equality holds since $\phi^\prime(u_\alpha) \to 0$ strongly in $\D^{-1,p^\prime}_a(\Om, 0)$. Combining the last two equations, we see tha
	\[
	\inner{\phi^\prime(u_\alpha-u), h} = \inner{\phi^\prime(u_\alpha), h} + o\left(\norm{h}_{ \D^{1,p}_a(\Om, 0)}\right)
	\]
	as $\alpha\to \infty$. Since $\phi^\prime(u_\alpha) \to 0$ strongly in $\D^{1,p}_a(\Om, 0)$, we deduce \ref{it:esd}.
\end{proof}

\subsection{Iterative Results}

The proof of the main results, namely Theorems \ref{thm:main a<b}-\ref{thm:main a=b}, is iterative. Therefore, central to the proof of the main results of this paper are two iteration results which will enable us to proceed to the next iteration stage.  

The first handles weighted bubbles and, in the main theorem, will be used provided condition \ref{it:ratio x_0} holds. In this case, we introduce notation for a translated analog of problem \eqref{eq:probLim},
\begin{equation}\label{eq:probLimX0}
	\begin{cases}
		-\operatorname{div}\left(\abs{x+x_0}^{-ap}\abs{\nabla u}^{p-2}\nabla u\right) = \abs{x+x_0}^{-bq}\abs{u}^{q-2}u &\text{in }\RR^n\\
		u\in \D_a^{1,p}\left(\RR^n, -x_0\right).
	\end{cases}
\end{equation}
In relation, we define the associated energy functional $\phi_{x_0, \infty} : \D_a^{1,p}(\RR^n, -x_0) \to \RR$ given by
\[
\phi_{x_0, \infty}(u) := \int_{\RR^n} \abs{\nabla u}^p\abs{z+x_0}^{-ap}\d{z} - \int_{\RR^n}\abs{u}^q\abs{z+x_0}^{-bq}\d{z},
\]
We note that a function $v\in \D_a^{1,p}(\RR^n, -x_0)$ is a (weak) solution to \eqref{eq:probLimX0} provided $\phi_{x_0, \infty}^\prime(v) = 0$.

The second iteration result treats unweighted bubbling phenomena in the case \ref{it:ratio inf} holds. We note that the proofs of the iterative results are very similar. However, due to the presence of a non-homogeneous transform, the latter presents more technical challenges. 

\begin{prop}\label{prop:iteration msc}
	Let $(u_\alpha)$ in $ \D_a^{1,p}(\Om, 0)$ be a bounded sequence and fix some $x_0\in \RR^n$ along with \(v\in \D_a^{1,p}(\RR^n, -x_0)\). For each $\alpha\in\NN$, put
	\[
	v_\alpha(x) := \lambda_\alpha^\gamma u_\alpha\left(\lambda_\alpha(x+x_0)\right)
	\]
	and suppose that
	\begin{enumerate}[label = (\roman*)]
		\item $\lambda_\alpha \to 0$, i.e. \ref{it:lambda=0} holds,
		\item $\phi(u_\alpha)\to c$ and $\phi^\prime(u_\alpha) \to 0$ strongly in the dual of $\D_a^{1,p}(\Om,0)$,
		\item $v_\alpha \rightharpoonup v$ weakly in $\D_a^{1,p}(\RR^n, -x_0)$ and pointwise almost everywhere,
	\end{enumerate}
	as $\alpha\to\infty$. Then, passing to a subsequence if necessary, $\nabla v_\alpha \to \nabla v$ almost everywhere and $\phi_{x_0, \infty}^\prime(v) = 0$. Finally, the sequence in $\D_a^{1,p}(\RR^n, 0)$ given by
	\[
	w_\alpha(z) := u_\alpha(z) -  \lambda_\alpha^{-\gamma} v\left(\frac{z}{\lambda_\alpha}-x_0\right)
	\]
	satisfies
	\begin{enumerate}[label = (\alph*)]
		\item\label{it:prop iteration a msc} $\norm{w_\alpha}_{\D_a^{1,p}(\RR^n, 0)}^p = \norm{u_\alpha}_{\D_a^{1,p}(\RR^n, 0)}^p - \norm{v}_{\D_a^{1,p}(\RR^n, -x_0)}^p + o(1)$,
		\item\label{it:prop iteration b msc} $\phi_{0, \infty}(w_\alpha) \to c - \phi_{x_0, \infty}(v)$,
		\item\label{it:prop iteration c msc} $\phi^\prime_{0, \infty}(w_\alpha) \to 0$ strongly in the dual of $\D_a^{1,p}(\Om,0)$,
	\end{enumerate}
	as $\alpha \to\infty$.
\end{prop}

When $a<b$ in the problem parameters, this last proposition is the only iterative result required to establish the conclusions of Theorem \ref{thm:main a<b}. In the general setting, we shall further require the following:

\begin{prop}\label{prop:iteration}
	Suppose we are given a bounded sequence $(u_\alpha)$ in $\D_a^{1,p}(\Omega,0)$ along with a set $\Om_\infty$, which is either $\RR^n$ or a half-space, and $v\in \D^{1,p}(\RR^n)$ such that the following hold as \(\alpha\to\infty\):
	\begin{enumerate}[label = (\roman*)]
		\item\label{it:prop iteration i} $\phi(u_\alpha)\to c$ and $\phi^\prime(u_\alpha) \to 0$ strongly in the dual of $\D_a^{1,p}(\Om,0)$,
		\item\label{it:prop iteration ii} $\Om_\alpha = (\Om-y_\alpha)/\lambda_\alpha$ tends to $\Om_\infty$,
		\item\label{it:prop iteration iii} $v_\alpha \rightharpoonup v$ weakly in $\D^{1,p}(\RR^n)$ and pointwise almost everywhere,
	\end{enumerate}
	where $v_\alpha := \tau_{y_\alpha, \lambda_\alpha}u_\alpha$. 
	If $a=b$ (so \(q = p^\ast\)) and \ref{it:ratio inf} holds (i.e. $\abs{y_\alpha}/\lambda_\alpha \to\infty$) then, $v\in \D^{1,p}(\Om_\infty)$ and, up to a subsequence, $\nabla v_\alpha \to \nabla v$ almost everywhere. Moreover, $\phi^\prime_\infty(v) = 0$ in $\Om_\infty$; i.e.
	\[
	\inner{\phi_\infty^\prime(v), h} = 0\qquad\forall h\in \D^{1,p}(\Om_\infty).
	\]
	Finally, $w_\alpha := u_\alpha - \tau_{y_\alpha, \lambda_\alpha}^- v$ satisfies
	\begin{enumerate}[label = (\alph*)]
		\item\label{it:prop iteration a} $\norm{w_\alpha}_{\D_a^{1,p}(\RR^n,0)}^p = \norm{u_\alpha}_{\D_a^{1,p}(\Omega,0)}^p - \norm{v}_{\D^{1,p}(\RR^n)}^p + o(1)$,
		\item\label{it:prop iteration b} $\phi_{0, \infty}(w_\alpha) \to c - \phi_\infty(v)$,
		\item\label{it:prop iteration c} $\phi^\prime_{0, \infty}(w_\alpha) \to 0$ strongly in the dual of $\D_a^{1,p}(\Om,0)$.
	\end{enumerate}
	as $\alpha \to\infty$.
\end{prop}
Notice that the inference
\(
\phi^\prime_\infty(v) = 0
\)
in $\Om_\infty$ amounts to stating that $v$ solves a PDE in $\Om_\infty$. Specifically, if $\Om_\infty = \RR^n$, then $v$ solves \eqref{eq:probLimWeightless}, whilst $v$ satisfies   \eqref{eq:probLimWeightlessHalf}  if $\Om_\infty = \HH$.

\begin{proof}[Proof of Proposition \ref{prop:iteration msc}]
	We adapt the argument from Mercuri-Willem \cite{mercuri-willem}. Given $k \in \NN$, let $B_k$ denote the open ball $B(0,k) \subset \RR^n$. As a first step, we claim that $\phi^\prime_{x_0,\infty}(v_\alpha) \to 0$ strongly in $\D^{-1,p^\prime}_a(B_k,-x_0)$ as $\alpha \to \infty$. Indeed, fix $h \in C_c^\infty(B_k)$ and define
	\[
	h_\alpha(z) :=  \lambda_\alpha^{-\gamma}h\left(\frac{z}{\lambda_\alpha} - x_0\right)
	\]
	Observe that, since $h_\alpha$ is supported in $\lambda_\alpha B(x_0,k)$, for all \(\alpha\) large $h_\alpha \in C_c^\infty(\Om)$. Now, by a change of variables,
	\[
	\inner{\phi_{x_0,\infty}^\prime(v_\alpha),h} = \inner{\phi^\prime(u_\alpha),h_\alpha}
	\]
	for all \(\alpha\) large (so that the right-hand side is defined). In particular, using homogeneity, we see that
	\begin{align*}
		\abs{\inner{\phi_{x_0,\infty}^\prime(v_\alpha),h}} = \abs{\inner{\phi^\prime(u_\alpha),h_\alpha}}
		&\leq \norm{\phi^\prime(u_\alpha)}_{\D^{-1,p^\prime}_a(\Om, 0)} \norm{h_\alpha}_{\D^{1,p}_a(\Om, 0)}\\
		&=\norm{\phi^\prime(u_\alpha)}_{\D^{-1,p^\prime}_a(\Om, 0)} \norm{h}_{\D^{1,p}(B_k, -x_0)}.
	\end{align*}
	so $\phi^\prime_{x_0,\infty}(v_\alpha) \to 0$ strongly in $\D^{-1,p^\prime}_a(B_k,-x_0)$. Using this, we will extract a subsequence such that $\nabla v_\alpha \to \nabla v$ pointwise on $\RR^n$, almost everywhere. Let $\rho \in C_c^\infty(\RR^n)$ be a bump function such that
	\[
	\begin{cases}
		0 \leq \rho \leq 1 & \text{on } \RR^n,\\
		\rho \equiv 1 & \text{in } B_k,\\
		\rho \equiv 0 & \text{outside } B_{k+1}.
	\end{cases}
	\]
	Consider the vector-valued map
	\[
	f_\alpha := \abs{\nabla v_\alpha}^{p-2}\nabla v_\alpha - \abs{\nabla v}^{p-2}\nabla v,
	\]
	which satisfies $f_\alpha \cdot (\nabla v_\alpha - \nabla v) \geq 0$ almost everywhere (see Szulkin-Willem \cite{Szulkin-Willem}). Let $T$ be as in \eqref{eq:TDef}. For each fixed $k \in \NN$, an easy computation shows that
	\begin{align*}
		0&\le
		\int_{B_k} \pinner{f_\alpha, \nabla \left[T(v_\alpha - v)\right]} \abs{x+x_0}^{-ap}\d{x}\\
		&\leq \int_{\RR^n} \pinner{f_\alpha, \rho \nabla \left[T(v_\alpha - v)\right]} \abs{x+x_0}^{-ap}\d{x}\\
		&= \inner{\phi^\prime_{x_0,\infty}(v_\alpha), \rho T(v_\alpha-v)} + \underbrace{\int_{\RR^n} \abs{v_\alpha}^{q-2}v_\alpha \rho T(v_\alpha -v)\abs{x+x_0}^{-bq}\d{x}}_{=:I_\alpha^{(1)}}
		\\&\quad
		- \underbrace{\int_{\RR^n} \pinner{f_\alpha, T(v_\alpha-v)\nabla \rho}\abs{x+x_0}^{-ap}\d{x}}_{=: I_\alpha^{(2)}}
		- \underbrace{\int_{\RR^n} \abs{\nabla v}^{p-2}\left(\nabla v, \nabla \left[\rho T(v_\alpha-v)\right]\right)\abs{x+x_0}^{-ap}\d{x}}_{=:I_\alpha^{(3)}}.
	\end{align*}
	First, since \((v_\alpha)\) is a bounded sequence in \(\D_a^{1,p}(\RR^n, -x_0)\), we see that \((T(v_\alpha-v))\) is also a bounded sequence in \(\D_a^{1,p}(\RR^n, -x_0)\). Becasue \(\rho\in C_c^\infty(B_{k+1})\), it follows that \((\rho T(v_\alpha-v))\) is a bounded sequence in \(\D_a^{1,p}(B_{k+1}, -x_0)\). Since $\phi^\prime_{x_0,\infty}(v_\alpha) \to 0$ strongly in $\D^{-1,p^\prime}_a(B_{k+1},-x_0)$, we deduce that 
	\[
	\inner{\phi^\prime_{x_0,\infty}(v_\alpha), \rho T(v_\alpha-v)} \to 0
	\]
	as \(\alpha\to\infty\).
	
	Next, by the CKN inequality, \((v_\alpha)\) is also a bounded sequence in \(L_b^q(\RR^n, -x_0)\). Therefore, by H\"older's inequality followed by the dominated convergence theorem, we see that  \(I_\alpha^{(2)} \to 0\) as \(\alpha \to \infty\).
	
	Next, using again that \((v_\alpha)\) is a bounded sequence in \(\D_a^{1,p}(\RR^n, -x_0)\), we see that the sequence \((f_\alpha)\) is bounded in \(L^{p^\prime}(\RR^n, \abs{x+x_0}^{-ap})\). Therefore, using also H\"older's inequality and the fact that $\operatorname{supp}(\rho) \subseteq B_{k+1}$, it is easily seen that
	\[
	\abs{I_\alpha^{(2)}} \leq C \left( \int_{B_{k+1}} \abs{T(v_\alpha-v)}^p\abs{x+x_0}^{-ap}\d{x} \right)^{1/p}
	\]
	for some constant \(C>0\). By the Dominated Convergence Theorem, we conclude that the right-hand side tends to \(0\) so \(I_\alpha^{(2)} \to 0\) as \(\alpha\to\infty\). 
	
	Finally, observe that we can decompose
	\begin{align*}
		I_\alpha^{(3)}
		&= \int_{\RR^n} \abs{\nabla v}^{p-2}\pinner{\nabla v, \nabla \rho} T(v_\alpha-v)\abs{x+x_0}^{-ap}\d{x}\\
		&\quad + \int_{\RR^n} \rho \abs{\nabla v}^{p-2}\pinner{\nabla v, \nabla\left[v_\alpha -v \right]} \abs{x+x_0}^{-ap}\d{x}\\
		&\qquad - \int_{\left\{\abs{v_\alpha -v } > 1\right\}} \rho \abs{\nabla v}^{p-2}\pinner{\nabla v, \nabla\left[v_\alpha -v \right]}\abs{x+x_0}^{-ap}\d{x}.
	\end{align*}
	On the right-hand side, the first term converges to by H\"older's inequality and the Dominated Convergence Theorem. The second integral converges to zero since \(v_\alpha\rightharpoonup v\) weakly in \(\D_a^{1,p}(\RR^n, -x_0)\). Finally, the third integral also converges to zero by H\"older's inequality and the dominated convergence theorem. Indeed, in absolute value, the third integral is bounded above by
	\[
	\norm{v_\alpha - v}_{\D_a^{1,p}(\RR^n, -x_0)}\left(\int_{\left\{\abs{v_\alpha -v } > 1\right\}} \abs{\nabla v}^p\abs{x+x_0}^{-ap}\d{x}\right)^{(p-1)/p} \to 0.
	\]
	Therefore, we have also found that \(I_\alpha^{(3)} \to 0\) as \(\alpha\to\infty\).
	
	Combining our work, we conclude that
	\[
	\lim_{\alpha \to \infty} \int_{B_k} \pinner{\left[\abs{\nabla v_\alpha}^{p-2}\nabla v_\alpha - \abs{\nabla v}^{p-2}\nabla v\right], \nabla \left[T(v_\alpha - v)\right]} \abs{x+x_0}^{-ap}\d{x} = 0
	\]
	for each $k \in \NN$. Citing Proposition \ref{prop:gradientConvergence}, we may assume after passing to a subsequence that $\nabla v_\alpha \to \nabla v$ pointwise almost everywhere on $\RR^n$. Therefore, using also that \(v_\alpha\to v\) almost everywhere, it follows from standard convergence results that
	\[
	\inner{\phi^\prime_{x_0,\infty}(v),h} = \lim_{\alpha \to \infty} \inner{\phi^\prime_{x_0,\infty}(v_\alpha),h}
	\]
	for any \(h\in C_c^\infty(\RR^n)\). Since \(h\) must be supported in \(B_k\) for some \(k\) and  $\phi^\prime_{x_0,\infty}(v_\alpha) \to 0$ strongly in $\D^{-1,p^\prime}_a(B_k)$, we conclude that
	\[
	\inner{\phi^\prime_{x_0,\infty}(v),h} = 0
	\]
	Since \(h\in C_c^\infty(\RR^n)\) was arbitrary, we conclude that $\phi_{x_0,\infty}^\prime(v) = 0$.

	By Corollary \ref{cor:gradientConvergence}, up to a subsequence, we have
	\begin{enumerate}[label=(\arabic*)]
		\item\label{eq:lem2Brezis} $\norm{v_\alpha - v}^p_{\D^{1,p}_a(\RR^n,-x_0)} = \norm{v_\alpha}^p_{\D^{1,p}_a(\RR^n,-x_0)} -\norm{v}^p_{\D^{1,p}_a(\RR^n,-x_0)} + o(1)$;	
		\item\label{eq:lem2Alves} and $$\abs{\nabla v_\alpha}^{p-2}\nabla v - \abs{\nabla v_\alpha - \nabla v}^{p-2}(\nabla v_\alpha - \nabla v) \to \abs{\nabla v}^{p-2}\nabla v$$
		in $L^{p^\prime}(\RR^n, \abs{x+x_0}^{-ap})$;
		\item\label{eq:lem2Alves bq} and $$\abs{v_\alpha}^{q-2}v_\alpha - \abs{v_\alpha - v}^{q-2}(v_\alpha - v) \to \abs{v}^{q-2}u$$ strongly in $L^{\frac{q}{q-1}}(\RR^n, \abs{x+x_0}^{-bq})$
	\end{enumerate}
	as \(\alpha\to\infty\).
	
	By a simple change of variables, it is clear that \ref{eq:lem2Brezis} implies \ref{it:prop iteration a msc} directly. Next, using the same change of variables, we see that
	\[
	\phi_{0,\infty}(w_\alpha) = \phi_{x_0,\infty}(v_\alpha - v)
	\]
	Then, the Br\'ezis-Lieb lemma ensures that
	\begin{align*}
		\phi_{0,\infty}(w_\alpha) 
		= \phi_{x_0,\infty}(v_\alpha) - \phi_{x_0,\infty}(v) + o(1)
		&= \phi(u_\alpha) -\phi_{x_0,\infty}(v) + o(1)\\
		&= c -\phi_{x_0,\infty}(v) + o(1)
	\end{align*}
	where the second equality follows from another change of variables. Thus, we have established  \ref{it:prop iteration b msc}.
	
	Finally, in order to prove claim \ref{it:prop iteration c msc} we fix an arbitrary $g\in \D^{1,p}_a(\Om, 0)$. After rescaling, we see that
	\[
	\inner{\phi^\prime_{0,\infty}(w_\alpha), g} = \inner{\phi^\prime_{x_0,\infty}(v_\alpha - v), g_\alpha}
	\]
	where $g_\alpha(x) = \lambda_\alpha^{\gamma}g(\lambda_\alpha(x+x_0))$. Then, applying H\"older's inequality together with \ref{eq:lem2Alves} and a change of variables,
	\begin{align*}
		\inner{\phi^\prime_{0,\infty}(w_\alpha), g} &= \inner{\phi^\prime_{x_0,\infty}(v_\alpha-v), g_\alpha}\\
		&= \inner{\phi^\prime_{x_0,\infty}(v_\alpha), g_\alpha} - \inner{\phi^\prime_{x_0,\infty}(v), g_\alpha} +  o\left(\norm{g}_{\D^{1,p}_a(\Om, 0)}\right)\\
		&= \inner{\phi^\prime(u_\alpha),g} + o\left(\norm{g}_{\D^{1,p}_a(\Om, 0)}\right)
	\end{align*}
	as \(\alpha \to\infty\), over$g\in \D^{1,p}_a(\Om, 0)$.  We infer that $\phi^\prime_{0,\infty}(w_\alpha) \to 0$ strongly in $\D^{-1,p^\prime}_a(\Om, 0)$ and the proof is complete.
\end{proof}

We now turn towards results involved in the proof of Proposition \ref{prop:iteration}, beginning with the following lemma.

\begin{lem}\label{lem:w_alpha norms}
	Suppose $(u_\alpha)$ in $\D_a^{1,p}(\Omega,0)$ is a given bounded sequence and there is $v\in \D^{1,p}(\RR^n)$ such that
	\begin{enumerate}[label = (\roman*)]
		\item $\tau_{y_\alpha, \lambda_\alpha}u_\alpha \rightharpoonup v$ weakly in $\D^{1,p}(\RR^n)$ and pointwise almost everywhere.
		\item $\nabla\left[\tau_{y_\alpha, \lambda_\alpha}u_\alpha\right] \to \nabla v$ almost everywhere in $\RR^n$.
	\end{enumerate}
	If $a=b$ (so \(q = p^\ast\)) and \ref{it:ratio inf} holds (i.e. $\abs{y_\alpha}/\lambda_\alpha \to\infty$)  then $w_\alpha := u_\alpha - \tau_{y_\alpha, \lambda_\alpha}^- v$ satisfies
	\begin{enumerate}[label = (\alph*)]
		\item $\norm{w_\alpha}_{\D_a^{1,p}(\RR^n,0)}^p = \norm{u_\alpha}_{\D_a^{1,p}(\Omega,0)}^p - \norm{v}_{\D^{1,p}(\RR^n)}^p + o(1)$ and
		\item $\norm{w_\alpha}_{L_b^q(\RR^n,0)}^q = \norm{u_\alpha}_{L_b^q(\Omega,0)}^q - \norm{v}_{L^q(\RR^n)}^q + o(1)$
	\end{enumerate}
	as $\alpha \to\infty$
\end{lem}
\begin{proof}
	For convenience, let us write $\tau_\alpha^- := \tau_{y_\alpha, \lambda_\alpha}^-$ and extend each $(u_\alpha)$ by $0$ to all of $\RR^n$. In order to establish (a), we split
	\begin{equation}\label{eq:w_split}
		\norm{w_\alpha}_{\D_a^{1,p}(\RR^n)}^p
		= \int_{B_\alpha}\abs{\nabla w_\alpha(z)}^p\abs{z}^{-ap}\d{z} +
		\int_{\RR^n\setminus B_\alpha}\abs{\nabla w_\alpha(z)}^p\abs{z}^{-ap}\d{z}
	\end{equation}
	where $B_\alpha := B(y_\alpha, \abs{y_\alpha}/4)$. The idea is that in the ball $B_\alpha$, the $\eta$ term is neutralized and, due to Lemma \ref{lem:away from ball}, we expect $\tau_\alpha^- v$ to be negligible outside of $B_\alpha$.
	
	In order to handle the first term on the right-hand side of equation \eqref{eq:w_split}, we add and subtract terms to obtain
	\begin{equation}\label{eq:pre1}
		\begin{aligned}
			\int_{B_\alpha}\abs{\nabla w_\alpha(z)}^p\abs{z}^{-ap}\d{z}
			&= \int_{B_\alpha}\left(\abs{\nabla w_\alpha(z)}^p + \abs{\nabla \left[\tau_\alpha^-v\right](z)}^p - \abs{\nabla u_\alpha(z)}^p \right)\abs{z}^{-ap}\d{z}\\
			&\quad + \int_{B_\alpha}\abs{\nabla u_\alpha(z)}^p\abs{z}^{-ap}\d{z} - \int_{B_\alpha}\abs{\nabla \left[\tau_\alpha^-v\right](z)}^p\abs{z}^{-ap}\d{z}.
		\end{aligned}
	\end{equation}
	Now, observe that observe that since $a=b$, properties of $\eta$ and a change of variables yields
	\begin{align*}
		\int_{B_\alpha}\abs{\nabla \left[\tau_\alpha^-v\right](z)}^p\abs{z}^{-ap}\d{z}
		&=\int_{\abs{x} \le \frac{\abs{y_\alpha}}{4\lambda_\alpha}}\abs{\nabla v(x)}^p\left(\frac{\abs{\lambda_\alpha x + y_\alpha}}{\abs{y_\alpha}}\right)^{-ap}\d{x}\\
		&=\int_{\RR^n}\abs{\nabla v(x)}^p\d{x} + o(1)
	\end{align*}
	as $\alpha\to\infty$ by the dominated convergence theorem (use Lemma \ref{lem:cutoff tools}-\ref{point1} and that $\abs{y_\alpha}/\lambda_\alpha \to\infty$). Returning to equation \eqref{eq:pre1}, we see that, as $\alpha\to\infty$,
	\begin{equation}\label{eq:1}
		\begin{aligned}
			\int_{B_\alpha}\abs{\nabla w_\alpha(z)}^p\abs{z}^{-ap}\d{z}
			&= \int_{B_\alpha}\left(\abs{\nabla w_\alpha(z)}^p + \abs{\nabla \left[\tau_\alpha^-v\right](z)}^p - \abs{\nabla u_\alpha(z)}^p \right)\abs{z}^{-ap}\d{z}\\
			&\quad + \int_{B_\alpha}\abs{\nabla u_\alpha(z)}^p\abs{z}^{-ap}\d{z} - \int_{\RR^n}\abs{\nabla v(x)}^p\d{x} + o(1).
		\end{aligned}
	\end{equation}
	Moreover, it is easy to show that there exists a constant $C>0$ such that
	\begin{equation}\label{eq:ineq}
		\abs{\abs{x}^p + \abs{y}^p - \abs{x+y}^p} \le C\left(\abs{x}^{p-1}\abs{y} + \abs{x}\abs{y}^{p-1}\right)\qquad\forall x,y\in \RR^n.
	\end{equation}
	Therefore, we may bound
	\begin{align*}
		&\int_{B_\alpha}\abs{\abs{\nabla w_\alpha(z)}^p + \abs{\nabla \left[\tau_\alpha^-v\right](z)}^p - \abs{\nabla u_\alpha(z)}^p}\abs{z}^{-ap}\d{z}\\
		&\le C\int_{B_\alpha}\left(\abs{\nabla \left[\tau_\alpha^-v\right](z)}^{p-1}\abs{\nabla w_\alpha(z)} + \abs{\nabla \left[\tau_\alpha^-v\right](z)}\abs{\nabla w_\alpha(z)}^{p-1}\right)\abs{z}^{-ap}\d{z}\\
		&\le C\int_{\abs{x} \le \frac{\abs{y_\alpha}}{4\lambda_\alpha}}\left(\abs{\nabla v}^{p-1}\abs{\nabla [v_\alpha-v]} + \abs{\nabla v}\abs{\nabla [v_\alpha-v]}^{p-1}\right)\d{x}
	\end{align*}
	where, with a possible modification of the constant $C$, the last inequality follows from a change of variables. Leveraging pointwise convergence and boundedness in appropriate Lebesgue spaces, it follows from standard results that this last term is $o(1)$ as $\alpha\to\infty$. Therefore, returning to \eqref{eq:1}, we have
	\begin{equation}\label{eq:1f}
		\int_{B_\alpha}\abs{\nabla w_\alpha(z)}^p\abs{z}^{-ap}\d{z}
		= \int_{B_\alpha}\abs{\nabla u_\alpha(z)}^p\abs{z}^{-ap}\d{z} - \int_{\RR^n}\abs{\nabla v (x)}^p\d{x} + o(1).
	\end{equation}
	Next, to handle the second term on the right-hand side of equation \eqref{eq:w_split}, we once again add and subtract terms to obtain, as $\alpha\to \infty$,
	\begin{equation}\label{eq:2}
		\begin{aligned}
			&\int_{\RR^n\setminus B_\alpha}\abs{\nabla w_\alpha(z)}^p\abs{z}^{-ap}\d{z}\\
			&= \int_{\RR^n\setminus B_\alpha}\left(\abs{\nabla w_\alpha(z)}^p - \abs{\nabla u_\alpha(z)}^p - \abs{\nabla \left[\tau_\alpha^-v\right](z)}^p \right)\abs{z}^{-ap}\d{z}\\
			&\quad + \int_{\RR^n\setminus B_\alpha}\abs{\nabla u_\alpha(z)}^p\abs{z}^{-ap}\d{z} + \int_{\RR^n\setminus B_\alpha}\abs{\nabla \left[\tau_\alpha^-v\right](z)}^p\abs{z}^{-ap}\d{z}.
		\end{aligned}
	\end{equation}
	By Lemma \ref{lem:away from ball}, the last term on the right-hand side tends to $0$ as $\alpha\to\infty$. Furthermore, using once again inequality \eqref{eq:ineq}, we bound
	\begin{align*}
		&\int_{\RR^n\setminus B_\alpha}\abs{\abs{\nabla w_\alpha(z)}^p  - \abs{\nabla u_\alpha(z)}^p - \abs{\nabla \left[\tau_\alpha^-v\right](z)}^p }\abs{z}^{-ap}\d{z}\\
		&\le C\int_{\RR^n\setminus B_\alpha}\left(\abs{\nabla u_\alpha(z)}^{p-1}\abs{\nabla \left[\tau_\alpha^-v\right](z)} + \abs{\nabla u_\alpha(z)}\abs{\nabla \left[\tau_\alpha^-v\right](z)}^{p-1}\right)\abs{z}^{-ap}\d{z}\\
		&\le C\left(\norm{u_\alpha}_{\D_a^{1,p}(\Om,0)}^{p-1}\norm{\nabla \left[\tau_\alpha^-v\right]}_{L_a^p(\RR^n\setminus B_\alpha)} + \norm{u_\alpha}_{\D_a^{1,p}(\Om,0)}\norm{\nabla \left[\tau_\alpha^-v\right]}_{L_a^p(\RR^n\setminus B_\alpha)}^{p-1}\right)
	\end{align*}
	by H\"older's inequality. Since $(u_\alpha)$ is a bounded sequence in $\D_a^{1,p}(\Om,0)$, it follows from Lemma \ref{lem:away from ball} that the last equation tends to $0$ as $\alpha\to\infty$. Thus, returning to \eqref{eq:2}, we see that
	\begin{equation}\label{eq:2f}
		\int_{\RR^n\setminus B_\alpha}\abs{\nabla w_\alpha(z)}^p\abs{z}^{-ap}\d{z}
		= \int_{\RR^n\setminus B_\alpha}\abs{\nabla u_\alpha(z)}^p\abs{z}^{-ap}\d{z} +o(1)
	\end{equation}
	Combining equations \eqref{eq:1f} and \eqref{eq:2f}, we obtain
	\[
	\int_{\RR^n}\abs{\nabla w_\alpha(z)}^p\abs{z}^{-ap}\d{z}
	= \int_{\RR^n}\abs{\nabla u_\alpha(z)}^p\abs{z}^{-ap}\d{z} - \int_{\RR^n}\abs{\nabla v (x)}^p\d{x} + o(1)
	\]
	which is precisely (a). We may deduce (b) via a similar argument.
\end{proof}

\begin{lem}\label{lem:energies every term}
	Suppose $(u_\alpha)$ in $\D_a^{1,p}(\Omega, 0)$ is a given bounded sequence and there is $v\in \D^{1,p}(\RR^n)$ such that
	\begin{enumerate}[label = (\roman*)]
		\item $v_\alpha \rightharpoonup v$ weakly in $\D^{1,p}(\RR^n)$ and pointwise almost everywhere.
		\item $\nabla v_\alpha \to \nabla v$ almost everywhere in $\RR^n$,
		\item\label{it:lem_iii} with $p^\prime = p/(p-1)$ denoting the H\"older conjugate of $p$,
		\[
		\abs{\nabla v_\alpha}^{p-2}\nabla v_\alpha - \abs{\nabla v_\alpha - \nabla v}^{p-2}\left(\nabla v_\alpha - \nabla v\right) \to \abs{\nabla v}^{p-2}\nabla v
		\]
		in $L^{p^\prime}(\RR^n)$,
		\item\label{it:lem_iv} with $q^\prime = q/(q-1)$ denoting the H\"older conjugate of $q$,
		\[
		\abs{ v_\alpha}^{q-2} v_\alpha - \abs{ v_\alpha -  v}^{q-2}\left( v_\alpha -  v\right) \to \abs{ v}^{q-2} v
		\]
		in $L^{q^\prime}(\RR^n)$,
	\end{enumerate}
	where $v_\alpha := \tau_{y_\alpha, \lambda_\alpha}u_\alpha$. If $a=b$ (so \(q = p^\ast\)) and \ref{it:ratio inf} holds (i.e. $\abs{y_\alpha}/\lambda_\alpha \to\infty$) then, with $w_\alpha := u_\alpha - \tau^-_{y_\alpha, \lambda_\alpha}v$,
	\[
	\inner{\phi_{0,\infty}^\prime(w_\alpha), g} = \inner{\phi^\prime(u_\alpha), g} - \inner{\phi_\infty^\prime(v), \tau_{y_\alpha, \lambda_\alpha}g} + o\left(\norm{g}_{\D_a^{1,p}(\Om,0)}\right)
	\]
	as $\alpha \to\infty$ over $g\in \D_a^{1,p}(\Om, 0)$.
\end{lem}
\begin{proof}
	Let $g\in \D_a^{1,p}(\Om, 0)$ and define $B_\alpha := B(y_\alpha, \abs{y_\alpha}/4)$ as in Lemma \ref{lem:w_alpha norms}. We split as follows:
	\begin{align*}
		\inner{\phi^\prime_{0, \infty}(w_\alpha), g}
		&=\int_{B_\alpha}\abs{\nabla w_\alpha}^{p-2}\pinner{\nabla w_\alpha, \nabla g} \abs{z}^{-ap} \d{z} + \int_{\RR^n\setminus B_\alpha}\abs{\nabla w_\alpha}^{p-2}\pinner{\nabla w_\alpha, \nabla g}\abs{z}^{-ap}\d{z}\\
		&\qquad - \int_{B_\alpha}\abs{w_\alpha}^{q-2}w_\alpha g\abs{z}^{-bq}\d{z} - \int_{\RR^n\setminus B_\alpha}\abs{w_\alpha}^{q-2}w_\alpha g\abs{z}^{-bq}\d{z}\\
		&=: I^{(1)}_a + I^{(2)}_a - I^{(1)}_b - I^{(2)}_b.
	\end{align*}
	On $B_\alpha$, using properties of $\eta$ and performing a change of variables, we see that
	\[
	I^{(1)}_a
	=\int_{\abs{x}\le \frac{\abs{y_\alpha}}{4\lambda_\alpha}}\abs{\nabla [v_\alpha-v]}^{p-2}\pinner{\nabla [v_\alpha-v], \nabla \left[\tau_{y_\alpha, \lambda_\alpha}g\right]} \left(\frac{\abs{\lambda_\alpha x + y_\alpha}}{\abs{y_\alpha}}\right)^{-ap}\d{x}
	\]
	Using this expression for $I_a^{(1)}$, it follows from Lemma \ref{lem:cutoff tools}-\ref{point1} that there exists a constant $C>0$ such that
	\begin{align*}
		&\abs{I^{(1)}_a - \int_{\abs{x}\le \frac{\abs{y_\alpha}}{4\lambda_\alpha}}\pinner{\left[\abs{\nabla v_\alpha}^{p-2}\nabla v_\alpha - \abs{\nabla v}^{p-2}\nabla v \right], \nabla \left[\tau_{y_\alpha, \lambda_\alpha}g\right]} \left(\frac{\abs{\lambda_\alpha x + y_\alpha}}{\abs{y_\alpha}}\right)^{-ap}\d{x}}\\
		&\le C\int_{\abs{x}\le \frac{\abs{y_\alpha}}{4\lambda_\alpha}}\abs{\abs{\nabla [v_\alpha-v]}^{p-2}\nabla [v_\alpha-v] - \abs{\nabla v_\alpha}^{p-2}\nabla v_\alpha + \abs{\nabla v}^{p-2}\nabla v}\cdot \abs{\nabla \left[\tau_{y_\alpha, \lambda_\alpha}g\right]} \d{x}\\
		&= o\left(\norm{g}_{\D_a^{1,p}(\Om, 0)}\right)
	\end{align*}
	as $\alpha\to\infty$. Here, the last line is derived using H\"older's inequality followed by property \ref{it:lem_iii} and Lemma \ref{lem:rescale_mod_bound}. From this, we have
	\begin{align*}
		I^{(1)}_a
		&=\int_{\abs{x}\le \frac{\abs{y_\alpha}}{4\lambda_\alpha}}\abs{\nabla v_\alpha}^{p-2}\pinner{\nabla v_\alpha, \nabla \left[\tau_{y_\alpha, \lambda_\alpha}g\right]} \left(\frac{\abs{\lambda_\alpha x + y_\alpha}}{\abs{y_\alpha}}\right)^{-ap}\d{x} \\
		&\qquad - \int_{\abs{x}\le \frac{\abs{y_\alpha}}{4\lambda_\alpha}}\abs{\nabla v}^{p-2}\pinner{\nabla v , \nabla \left[\tau_{y_\alpha, \lambda_\alpha}g\right]} \left(\frac{\abs{\lambda_\alpha x + y_\alpha}}{\abs{y_\alpha}}\right)^{-ap}\d{x} + o\left(\norm{g}_{\D_a^{1,p}(\Om, 0)}\right)
	\end{align*}
	as $\alpha\to\infty$. After a change of variables on the first integral, we obtain
	\begin{align}
		I^{(1)}_a
		&=\int_{B_\alpha}\abs{\nabla u_\alpha}^{p-2}\pinner{\nabla u_\alpha, \nabla g} \abs{z}^{-ap}\d{z} \label{eq:lemI_a^1}\\
		&\qquad - \int_{\abs{x}\le \frac{\abs{y_\alpha}}{4\lambda_\alpha}}\abs{\nabla v}^{p-2}\pinner{\nabla v, \nabla \left[\tau_{y_\alpha, \lambda_\alpha}g\right]} \left(\frac{\abs{\lambda_\alpha x + y_\alpha}}{\abs{y_\alpha}}\right)^{-ap}\d{x} + o\left(\norm{g}_{\D_a^{1,p}(\Om, 0)}\right)\notag
	\end{align}
	as $\alpha\to\infty$. Similarly, using property \ref{it:lem_iv} instead of property \ref{it:lem_iii},
	we have
	\begin{align}
		I_b^{(1)}
		&= \int_{B_\alpha}\abs{u_\alpha}^{q-2}u_\alpha(z) g(z)\abs{z}^{-bq}\d{z} \label{eq:lemI_b^1}\\
		&\qquad -\int_{\abs{x}\le \frac{\abs{y_\alpha}}{4\lambda_\alpha}}\abs{v}^{q-2}v(x) \left[\tau_{y_\alpha, \lambda_\alpha}g\right](x)\left(\frac{\abs{\lambda_\alpha x + y_\alpha}}{\abs{y_\alpha}}\right)^{-bq}\d{x} + o\left(\norm{g}_{\D_a^{1,p}(\Om, 0)}\right)\notag
	\end{align}
	as $\alpha\to\infty$.
	
	Next, using H\"older's inequality, we may bound
	\begin{align}
		&\abs{I_a^{(2)} - \int_{\RR^n\setminus B_\alpha}\abs{\nabla u_\alpha}^{p-2}\pinner{\nabla u_\alpha, \nabla g}\abs{z}^{-ap}\d{z}}\notag\\
		&=\abs{\int_{\RR^n\setminus B_\alpha}\pinner{\left[\abs{\nabla w_\alpha}^{p-2}\nabla w_\alpha - \abs{\nabla u_\alpha}^{p-2}\nabla u_\alpha\right], \nabla g}\abs{z}^{-ap}\d{z}}\label{eq:lemtemp}\\
		&\le \norm{\abs{\nabla w_\alpha}^{p-2}\nabla w_\alpha - \abs{\nabla u_\alpha}^{p-2}\nabla u_\alpha}_{L^{p^\prime}_a(\RR^n\setminus B_\alpha)}\norm{g}_{\D_a^{1,p}(\Om, 0)}\notag
	\end{align}
	Let us now bound the right hand side by considering two distinct cases. First, if $1<p\le 2$ then it follows from a result of Mercuri-Willem \cite[Lemma 3.1]{mercuri-willem} that, for some constant $C>0$,
	\[
	\abs{\abs{\nabla w_\alpha}^{p-2}\nabla w_\alpha - \abs{\nabla u_\alpha}^{p-2}\nabla u_\alpha}
	\le C\abs{\nabla w_\alpha - \nabla u_\alpha}^{p-1} = C\abs{\nabla\left[\tau^-_{y_\alpha, \lambda_\alpha}v\right]}^{p-1}.
	\]
	Then,
	\[
	\norm{\abs{\nabla w_\alpha}^{p-2}\nabla w_\alpha - \abs{\nabla u_\alpha}^{p-2}\nabla u_\alpha}_{L^{p^\prime}_a(\RR^n\setminus B_\alpha)}
	\le C\norm{\nabla\left[\tau^-_{y_\alpha, \lambda_\alpha}v\right]}_{L^p_a(\RR^n\setminus B_\alpha)} \to 0
	\]
	as $\alpha \to \infty$ by 
	Lemma \ref{lem:away from ball}. In conclusion, when $1< p\le 2$, we have shown that
	\[
	\norm{\abs{\nabla w_\alpha}^{p-2}\nabla w_\alpha - \abs{\nabla u_\alpha}^{p-2}\nabla u_\alpha}_{L^{p^\prime}_a(\RR^n\setminus B_\alpha)} \to 0.
	\]
	On the other hand, if $p > 2$ then it is easy to see that there is a constant $C=C(p)>0$ such that
	\[
	\abs{\abs{x+y}^{p-2}(x+y) - \abs{x}^{p-2}x - \abs{y}^{p-2}y}\le C\left(\abs{x}^{p-2}\abs{y}  + \abs{x}\abs{y}^{p-2}\right),\quad\forall x,y\in \RR^n.
	\]
	Therefore,
	\begin{align*}
		&\abs{\abs{\nabla w_\alpha}^{p-2}\nabla w_\alpha - \abs{\nabla u_\alpha}^{p-2}\nabla u_\alpha - \abs{-\nabla\left[\tau^-_{y_\alpha, \lambda_\alpha}v\right]}^{p-2}\left(-\nabla\left[\tau^-_{y_\alpha, \lambda_\alpha}v\right]\right)} \\
		&\le C\left(\abs{u_\alpha}^{p-2}\abs{\nabla\left[\tau^-_{y_\alpha, \lambda_\alpha}v\right]} + \abs{u_\alpha}\abs{\nabla\left[\tau^-_{y_\alpha, \lambda_\alpha}v\right]}^{p-2}\right)
	\end{align*}
	Then, the triangle inequality followed by this last identity allows us bound as follows:
	\begin{align*}
		&\norm{\abs{\nabla w_\alpha}^{p-2}\nabla w_\alpha - \abs{\nabla u_\alpha}^{p-2}\nabla u_\alpha}_{L^{p^\prime}_a(\RR^n\setminus B_\alpha)}\\
		&\le \norm{\abs{\nabla w_\alpha}^{p-2}\nabla w_\alpha - \abs{\nabla u_\alpha}^{p-2}\nabla u_\alpha + \abs{\nabla\left[\tau^-_{y_\alpha, \lambda_\alpha}v\right]}^{p-2}\nabla\left[\tau^-_{y_\alpha, \lambda_\alpha}v\right]}_{L^{p^\prime}_a(\RR^n\setminus B_\alpha)} \\
		&\quad + \norm{\abs{\nabla\left[\tau^-_{y_\alpha, \lambda_\alpha}v\right]}^{p-1}}_{L^{p^\prime}_a(\RR^n\setminus B_\alpha)}\\
		&\le C\left(\norm{\abs{\nabla u_\alpha}^{p-2}\abs{\nabla\left[\tau^-_{y_\alpha, \lambda_\alpha}v\right]}}_{L^{p^\prime}_a(\RR^n\setminus B_\alpha)} + \norm{\abs{\nabla u_\alpha}\abs{\nabla\left[\tau^-_{y_\alpha, \lambda_\alpha}v\right]}^{p-2}}_{L^{p^\prime}_a(\RR^n\setminus B_\alpha)}\right) \\
		&\quad + \norm{\nabla\left[\tau^-_{y_\alpha, \lambda_\alpha}v\right]}_{L^{p}_a(\RR^n\setminus B_\alpha)}
	\end{align*}
	Then, using H\"older's inequality (with exponent $p-1 > 1$ and conjugate exponent $(p-1)/(p-2)$) and that $(u_\alpha)$ is a bounded sequence in $\D_a^{1,p}(\Om,0)$, it follows from Lemma \ref{lem:away from ball} that
	\[
	\norm{\abs{\nabla w_\alpha}^{p-2}\nabla w_\alpha - \abs{\nabla u_\alpha}^{p-2}\nabla u_\alpha}_{L^{p^\prime}_a(\RR^n\setminus B_\alpha)} \to 0.
	\]
	Having demonstrated this convergence both in the case $1<p\le 2$ and $p>2$, we return to equation \eqref{eq:lemtemp} and conclude that
	\[
	I_a^{(2)} = \int_{\RR^n\setminus B_\alpha}\abs{\nabla u_\alpha}^{p-2}\nabla \pinner{u_\alpha, \nabla g}\abs{z}^{-ap}\d{z} + o\left(\norm{g}_{\D_a^{1,p}(\Om,0)}\right)
	\]
	as $\alpha\to\infty$. Similarly, we obtain
	\[
	I_b^{(2)} = \int_{\RR^n\setminus B_\alpha}\abs{u_\alpha}^{q-2}u_\alpha g\abs{z}^{-bq}\d{z} + o\left(\norm{g}_{\D_a^{1,p}(\Om,0)}\right)
	\]
	as $\alpha\to\infty$.
	
	Combining these last two equations with \eqref{eq:lemI_a^1} and \eqref{eq:lemI_b^1}, as $\alpha\to\infty$
	\begin{align*}
		\inner{\phi^\prime_{0, \infty}(w_\alpha), g}
		&= I^{(1)}_a + I^{(2)}_a - I^{(1)}_b - I^{(2)}_b \\
		&= \int_{\RR^n}\abs{\nabla u_\alpha}^{p-2}\pinner{\nabla u_\alpha, \nabla g} \abs{z}^{-ap}\d{z} - \int_{\RR^n}\abs{u_\alpha}^{q-2}u_\alpha(z) g(z)\abs{z}^{-bq}\d{z}\\
		&\quad - \int_{\abs{x}\le \frac{\abs{y_\alpha}}{4\lambda_\alpha}}\abs{\nabla v}^{p-2}\pinner{\nabla v , \nabla \left[\tau_{y_\alpha, \lambda_\alpha}g\right]} \left(\frac{\abs{\lambda_\alpha x + y_\alpha}}{\abs{y_\alpha}}\right)^{-ap}\d{x}\\
		&\qquad +\int_{\abs{x}\le \frac{\abs{y_\alpha}}{4\lambda_\alpha}}\abs{v}^{q-2}v(x) \left[\tau_{y_\alpha, \lambda_\alpha}g\right](x)\left(\frac{\abs{\lambda_\alpha x + y_\alpha}}{\abs{y_\alpha}}\right)^{-bq}\d{x}\\
		&\qquad\quad + o\left(\norm{g}_{\D_a^{1,p}(\Om,0)}\right).
	\end{align*}
	Finally, by Lemma \ref{lem:concentration limiting energy},
	\[
	\inner{\phi^\prime_{0, \infty}(w_\alpha), g}
	=\inner{\phi^\prime(u_\alpha), g} - \inner{\phi^\prime_\infty(v), \tau_{y_\alpha, \lambda_\alpha}g} + o\left(\norm{g}_{\D_a^{1,p}(\Om,0)}\right).
	\]
\end{proof}

\begin{proof}[Proof of Proposition \ref{prop:iteration}]
	It immediately follows from Lemma \ref{lem:convergence of functions on domains} that $v\in \D^{1,p}(\Om_\infty)$. Since $\Om_\infty$ is either a half-space or $\RR^n$, we may find an increasing collection  $(B_k)_{k\in \NN}$ of open sets such that
	\begin{enumerate}
		\item $B_k$ is compactly contained in $\Om_\infty$ for each $k$;
		\item For each $k$, the closure of $B_k$ is a compact set contained in $B_{k+1}$;
		\item Every $x\in \Om_\infty$ is in $B_k$ for all $k$ large.
	\end{enumerate} 
	For fixed $k$, it follows from Lemma \ref{lem:transform energy} that
	\[
	\inner{\phi^\prime_\infty(v_\alpha), h} = \inner{\phi^\prime(u_\alpha), \tau^-_{y_\alpha,\lambda_\alpha} h} + o\left(\norm{h}_{\D^{1,p}(\RR^n)}\right)
	\]
	as $\alpha\to\infty$ over $h\in \D^{1,p}(B_k)$. Here, we note that there is a slight abuse in notation since the right-hand side is only well-defined for $\alpha$ large. Indeed, for all $\alpha$ large, $B_k\subseteq \Om_\alpha$ so $\tau^-_{y_\alpha,\lambda_\alpha}h \in \D^{1,p}_a(\Om,0)$ for each $h\in \D^{1,p}(B_k)$. We may further bound
	\begin{align*}
		\abs{\inner{\phi^\prime_\infty(v_\alpha), h}}
		&\le \abs{\inner{\phi^\prime(u_\alpha), \tau^-_{y_\alpha,\lambda_\alpha} h}} + o\left(\norm{h}_{\D^{1,p}(\RR^n)}\right)\\
		&\le \norm{\phi^\prime(u_\alpha)}_{\D^{-1,p^\prime}_a(\Om,0)}\norm{\tau^-_{y_\alpha,\lambda_\alpha} h}_{\D_a^{1,p}(\Om,0)}+ o\left(\norm{h}_{\D^{1,p}(\RR^n)}\right)\\
		&=  o\left(\norm{h}_{\D^{1,p}(\RR^n)}\right)
	\end{align*}
	as $\alpha\to\infty$ over $h\in \D^{1,p}(B_k)$. Note that, for the last equality, we used the assumption $\phi^\prime(u_\alpha)\to 0$ strongly in the dual of $\D_a^{1,p}(\Om, 0)$ and Lemma \ref{lem:rescale_mod_inv_bound}. From our work thus far, we see that
	\begin{equation}\label{eq:h energy}
		\phi_\infty^\prime(v_\alpha) \to 0 \text{ strongly in } \D^{-1,p^\prime}(B_k)
	\end{equation}
	for each fixed $k\in \NN$.

	Next, we will prove that the conditions of Corollary \ref{cor:gradientConvergence} are satisfied for the sequence $(v_\alpha)$ with the increasing sequence of open subsets $B_k \nearrow \Om_\infty$. First, it is clear from Lemma \ref{lem:rescale_mod_bound} that $(v_\alpha)$ is bounded in $\D^{1,p}(\RR^n)$. Let $T :\RR\to \RR$ be the Lipschitz continuous function
	\[
	T(x) = \begin{cases}
		x &\text{if }\abs{x} \le 1\\
		\frac{x}{\abs{x}} &\text{if } \abs{x} > 1.
	\end{cases}
	\]
	Fix $k\in\NN$ and let $\rho\in C_c^\infty(\Om_\infty) \subseteq C_c^\infty(\RR^n)$ be a bump function such that
	\[
	\begin{cases}
		0\le\rho\le 1 &\text{in } \RR^n\\
		\rho \equiv 1 &\text{in } B_k\\
		\rho \equiv 0 &\text{outside } B_{k+1}.
	\end{cases}
	\]
	Observe that
	\[
	f_\alpha := \abs{\nabla v_\alpha}^{p-2}\nabla v_\alpha - \abs{\nabla v}^{p-2}\nabla v
	\]
	is uniformly (in $\alpha$) bounded in $L^{p^\prime}(\RR^n)$, with $p^\prime := p/(p-1)$. Moreover, by a classical result (see, for instance, Szulkin-Willem \cite[Lemma 2.1]{Szulkin-Willem} and Simon \cite[p.210-212]{simon2006regularite} for the proof), we have that $\pinner{f_\alpha, \left[\nabla v_\alpha - \nabla v\right]} \ge 0$ almost everywhere. Therefore,
	\begin{align*}
		0\le \int_{B_k} \pinner{f_\alpha, \nabla T(v_\alpha -v)}\d{x}
		&\le \int_{\RR^n} \pinner{f_\alpha, \rho\nabla T(v_\alpha -v)}\d{x}\\
		&=\int_{\RR^n} \pinner{f_\alpha, \nabla \left[\rho T(v_\alpha -v)\right]}\d{x}
		-\int_{\RR^n} \pinner{f_\alpha, T(v_\alpha -v)\nabla \rho}\d{x}\\
		&=: I_\alpha^{(1)} - I^{(2)}_\alpha.
	\end{align*}
	We claim that both $I_\alpha^{(1)}$ and $I_\alpha^{(2)}$ tend to $0$ as $\alpha\to\infty$. First, by H\"older's inequality,
	\begin{align*}
		\abs{I_\alpha^{(2)}}
		\le \norm{f_\alpha}_{L^{p^\prime}(\RR^n)}\left(\int_{\RR^n}\abs{\nabla\rho}^p\abs{T(v_\alpha-v)}^p\d{x}\right)^{1/p}
	\end{align*}
	Since the right integrand is bounded by a constant multiple of the indicator function of $B_{k+1}$, the dominated convergence theorem (using that $v_\alpha\to v$ almost everywhere) implies that $I_\alpha^{(2)}\to 0$ as $\alpha\to\infty$. 
	
	Next, in order to estimate $I_\alpha^{(1)}$, we first observe that $\rho T(v_\alpha - v) \in \D^{1,p}\left(B_{k+1}\right)$ and, in fact, $\rho T(v_\alpha-v)$ is bounded in $\D^{1,p}\left(B_{k+1}\right)$ uniformly in $\alpha$. Since $\phi_\infty^\prime(v_\alpha) \to 0$ strongly in the dual of $\D^{1,p}\left(B_{k+1}\right)$,
	\[
	\inner{\phi_\infty^\prime(v_\alpha), \rho T(v_\alpha-v)} \to 0 \quad\text{as }\alpha\to\infty.
	\]
	Therefore, writing
	\begin{align*}
		I_\alpha^{(1)} = \inner{\phi_\infty^\prime(v_\alpha),  \rho T(v_\alpha-v)}
		&+ \underbrace{\int_{\RR^n}\abs{v_\alpha}^{q-2}v_\alpha \rho T(v_\alpha-v)\d{x}}_{=: J_\alpha^{(1)}}\\
		&\quad - \underbrace{\int_{\RR^n}\abs{\nabla v}^{p-2}\pinner{\nabla v,\nabla\left[\rho T(v_\alpha-v)\right]}\d{x}}_{=: J_\alpha^{(2)}},
	\end{align*}
	we have
	\[
	I_\alpha^{(1)} = o(1) + J_\alpha^{(1)} - J_\alpha^{(2)}
	\]
	as $\alpha\to\infty$. Using H\"older's inequality, we see that
	\[
	\abs{J_\alpha^{(1)}} \le \norm{v_\alpha}_{L^q(\RR^n)}^{q-1}\left(\int_{\RR^n}\rho^q(x)\abs{T(v_\alpha-v)}^q\d{x}\right)^{1/q}
	= o(1)
	\]
	as $\alpha\to\infty$. Here, we have used the dominated convergence theorem and applied the CKN inequality \eqref{eq:CKN} to deduce that $\norm{v_\alpha}_{L^q(\RR^n)}$ is bounded. On the other hand, as $\alpha\to\infty$,
	\begin{align*}
		J_\alpha^{(2)}
		&=\int_{\RR^n}\abs{\nabla v}^{p-2}\pinner{\nabla v,\nabla\left[\rho T(v_\alpha-v)\right]}\d{x}\\
		&=\int_{\abs{v_\alpha-v} \le 1}\abs{\nabla v}^{p-2}\pinner{\nabla v,\rho \nabla(v_\alpha-v)}\d{x} + \int_{\RR^n}\abs{\nabla v}^{p-2}\pinner{\nabla v, T(v_\alpha-v)\nabla\rho}\d{x}\\
		&=\int_{\abs{v_\alpha-v} \le 1}\abs{\nabla v}^{p-2}\pinner{\nabla v,\rho \nabla(v_\alpha-v)}\d{x} + o(1)
	\end{align*}
	where, to deduce the last equality, we have used H\"older's inequality and the dominated convergence theorem. We then split the right hand-side as follows:
	\[
	J_\alpha^{(2)}
	=\int_{\RR^n}\abs{\nabla v}^{p-2}\nabla v, \rho \nabla(v_\alpha-v)\d{x} - \int_{\abs{v_\alpha-v} > 1}\abs{\nabla v}^{p-2}\pinner{\nabla v, \rho \nabla(v_\alpha-v)}\d{x} + o(1).
	\]
	Note that the first term on the right-hand side tends to $0$ as $\alpha\to\infty$. Indeed, this is because $v_\alpha \rightharpoonup v$ weakly in $\D^{1,p}(\RR^n)$ or, equivalently, $\nabla v_\alpha\rightharpoonup \nabla v$ weakly in $L^p(\RR^n)$.
	\begin{align*}
		J_\alpha^{(2)}
		&=- \int_{\abs{v_\alpha-v} > 1}\abs{\nabla v}^{p-2}\pinner{\nabla v, \rho \nabla(v_\alpha-v)}\d{x} + o(1)
	\end{align*}
	With another application of H\"older's inequality we see that,
	\begin{align*}
		\abs{J_\alpha^{(2)}} 
		&\le \int_{\set{\abs{v_\alpha-v} > 1}\cap B_{k+1}}\abs{\nabla v}^{p-1} \abs{\nabla(v_\alpha-v)}\d{x} +o(1)\\
		&\le \norm{v_\alpha - v}_{\D^{1,p}(\RR^n)}\left(\int_{\set{\abs{v_\alpha-v} > 1}\cap B_{k+1}}\abs{\nabla v}^p\d{x}\right)^{1/p^\prime} + o(1)
	\end{align*}
	as $\alpha \to\infty$. Since $v_\alpha\to v$ pointwise a.e., the dominated convergence theorem ensures that the right hand side tends to $0$ as $\alpha\to \infty$. Especially, $J_\alpha^{(2)}\to 0$ as $\alpha\to\infty$.  
	
	Combining our result for $J_\alpha^{(1)}$ and $J_\alpha^{(2)}$, we see that $I_\alpha^{(1)} \to 0$ as $\alpha\to\infty$. Putting this together with $I_\alpha^{(2)} \to 0$ as $\alpha\to\infty$, we have shown that
	\[
	\int_{B_k \cap \Om_\infty} \pinner{f_\alpha, \nabla T(v_\alpha -v)}\d{x} \to 0\quad\text{as }\alpha\to\infty
	\]
	where
	\[
	f_\alpha = \abs{\nabla v_\alpha}^{p-2}\nabla v_\alpha - \abs{\nabla v}^{p-2}\nabla v.
	\]
	Therefore, by Corollary \ref{cor:gradientConvergence} (invoked with \(a=b=0\)), up to a subsequence, we have the following.
	\begin{enumerate}[label=(\arabic*)]
		\item $\nabla v_\alpha \to \nabla v$ almost everywhere in $\Om_\infty$. Using that $v_\alpha$ is supported in $\Om_\alpha$, $v_\alpha\to v$ pointwise almost everywhere and $\Om_\alpha \to \Om_\infty$, it readily follows that $\nabla v_\alpha \to \nabla v$  almost everywhere in $\RR^n$. 
		\item $\norm{v}_{\D^{1,p}(\Om_\infty)}^p = \norm{v_\alpha}_{\D^{1,p}(\Om_\infty)}^p - \norm{v_\alpha -v}_{\D^{1,p}(\Om_\infty)}^p + o(1)$.
		\item \label{conc:3} There holds
		\[
		\abs{\nabla v_\alpha}^{p-2}\nabla v_\alpha - \abs{\nabla v_\alpha - \nabla v}^{p-2}\left(\nabla v_\alpha - \nabla v\right) \to \abs{\nabla v}^{p-2}\nabla v.
		\]
		strongly in $L^{p^\prime}(\Om_\infty)$. Then, since $v\in \D^{1,p}(\Om_\infty)$ it immediately follows that the convergence remains true in $L^{p^\prime}(\RR^n)$.
		\item\label{conc:4} Finally,
		\[
		\abs{v_\alpha}^{q-2} v_\alpha - \abs{v_\alpha - v}^{q-2}\left(v_\alpha - v\right) \to \abs{v}^{q-2} v.
		\]
		strongly in $L^{q/(q-1)}(\Om_\infty)$ and hence in \(L^{q/(q-1)}(\RR^n)\).
	\end{enumerate}
	Next, fix $h\in C_c^\infty(\Om_\infty)$. Then, $h\in C_c^\infty(B_k)$ for all $k$ large; fix such an index $k$. Since, as stated in equation \eqref{eq:h energy}, $\phi^\prime(v_\alpha) \to 0$ strongly in the dual of $\D^{1,p}(B_k)$, there holds
	\begin{align*}
		0 &= \lim_{\alpha\to\infty}\inner{\phi_\infty^\prime(v_\alpha), h}\\
		&=\lim_{\alpha\to\infty}\left(\int_{\RR^n}\abs{\nabla v_\alpha}^{p-2}\pinner{\nabla v_\alpha, \nabla h}\d{x} - \int_{\RR^n}\abs{v_\alpha}^{q-2}v_\alpha h\d{x}\right)\\
		&=\int_{\RR^n}\abs{\nabla v}^{p-2}\pinner{\nabla v, \nabla h}\d{x} - \int_{\RR^n}\abs{v}^{q-2}v h\d{x}\\
		&=\inner{\phi_\infty^\prime(v), h}
	\end{align*}
	where the second equality is due to a standard convergence result, using that $v_\alpha \to v$ and $\nabla v_\alpha \to \nabla v$ pointwise almost everywhere. Since $h\in C_c^\infty(\Om_\infty)$ was arbitrary, we conclude that $\phi_\infty^\prime(v) = 0$ in $\Om_\infty$.
	
	Next, we establish the desired properties of the sequence $(w_\alpha)$. First, \ref{it:prop iteration a} follows from Lemma \ref{lem:w_alpha norms}. By this same lemma, we also see that
	\begin{align*}
		\phi_{0, \infty}(w_\alpha) 
		= \frac{1}{p}\norm{w_\alpha}_{\D_a^{1,p}(\RR^n,0)}^p - \frac{1}{q}\norm{w_\alpha}_{L_b^q(\RR^n,0)}
		&= \phi(u_\alpha) - \phi_\infty(v) + o(1)\\
		&= c-\phi_\infty(v) + o(1)
	\end{align*}
	as $\alpha \to\infty$. That is, we have \ref{it:prop iteration b}. Finally, by Lemma \ref{lem:energies every term},
	\[
	\inner{\phi_{0,\infty}^\prime(w_\alpha), g} = \inner{\phi^\prime(u_\alpha), g} - \inner{\phi_\infty^\prime(v), \tau_{y_\alpha, \lambda_\alpha}g} + o\left(\norm{g}_{\D_a^{1,p}(\Om,0)}\right)
	\]
	as $\alpha\to\infty$ over $g\in \D_a^{1,p}(\Om)$. Then, using that $\phi^\prime(u_\alpha)\to 0$ strongly in the dual of $\D_a^{1,p}(\Om)$ and applying Lemma \ref{lem:v energy transform g}, we conclude that
	\[
	\inner{\phi_{0,\infty}^\prime(w_\alpha), g}  = o\left(\norm{g}_{\D_a^{1,p}(\Om,0)}\right)
	\]
	which gives \ref{it:prop iteration c}.
	
\end{proof}

\subsection{Bubbling}
In the iteration lemmas, the terms subtracted from the Palais-Smale sequence \((u_\alpha)\) are referred to as ``bubbles".  We claim that, as \(\alpha\to\infty\), the bubble terms converge to \(0\) in suitable senses. Explicitly, we provide two results.

First, in the first iteration result, i.e. Proposition \ref{prop:iteration msc}, the subtracted terms are
\begin{equation}\label{eq:bubs}
	B_\alpha(z) := \lambda_\alpha^{-\gamma} v\left(\frac{z}{\lambda_\alpha}-x_0\right).
\end{equation}
Now, with condition \ref{it:lambda=0}, we have the following result:
\begin{lem}\label{lem:bubs}
	Let \(v\in \D_a^{1,p}(\RR^n, -x_0)\) and assume condition \ref{it:lambda=0} holds (i.e. $\lambda_\alpha \to 0$). Then, up to a subsequence,
	\begin{enumerate}[label=(\alph*)]
		\item $\nabla B_\alpha(x) \to 0$ pointwise almost everywhere on $\RR^n$.
		\item $B_\alpha \rightharpoonup 0$ weakly in $\D^{1,p}_a(\RR^n, 0)$ and pointwise almost everywhere;
	\end{enumerate}
	as \(\alpha\to\infty\). Here, \(B_\alpha\) are given by \eqref{eq:bubs}
\end{lem}
\begin{proof}
	Fix $\epsilon > 0$ and observe that, after a change of variables,
	\[
	\int_{ B(0,\epsilon)^\complement} \abs{ \nabla B_\alpha(z) }^p\abs{z}^{-ap}\d{x}
	=\int_{B\left(-x_0,\frac{\epsilon}{\lambda_\alpha}\right)^\complement} \abs{ \nabla v\left( x\right) }^p \abs{x+x_0}^{-ap}\d{x}.
	\]
	The right-hand side of this equation converges to $0$ by the monotone convergence theorem. Passing to a subsequence, we infer that $\nabla B_\alpha \to 0$ pointwise almost everywhere on $\RR^n \setminus B(0,\epsilon)$. As $\epsilon > 0$ was arbitrary, a diagonal argument gives the existence of a subsequence such that $\nabla B_\alpha \to 0$ almost everywhere on $\RR^n$. 
	
	Since $(\nabla B_\alpha)$ is bounded in $L^p_a(\RR^n, 0)$, a simple change of variables and standard convergence results show that $B_\alpha \rightharpoonup 0$ in $\D^{1,p}_a(\RR^n, 0)$. 
	
	On the other hand, since $(B_\alpha)$ is bounded in $\D^{1,p}_a(\RR^n, 0)$, up to a subsequence, it converges both weakly in $\D^{1,p}_a(\RR^n, 0)$ and pointwise almost everywhere to some function in $\D^{1,p}_a(\RR^n, 0)$. By uniqueness of the weak limit, we conclude that \(B_\alpha \to 0\) pointwise almost everywhere.
\end{proof}

Next, in Proposition \ref{prop:iteration}, the \emph{bubbles} take the form \(\tau^-_{y_\alpha, \lambda_\alpha}v\) and we obtain a similar convergence result:		
\begin{lem}\label{lem:bubbling}
	Let $v\in \D^{1,p}(\RR^n)$. If $a=b$ and condition \ref{it:lambda=0} holds (i.e. $\lambda_\alpha \to 0$) then, up to a subsequence,
	\begin{enumerate}[label=(\alph*)]
		\item $\nabla \left[\tau^-_{y_\alpha, \lambda_\alpha}v\right] \to 0$ pointwise almost everywhere
		\item $\tau^-_{y_\alpha, \lambda_\alpha}v\rightharpoonup 0$ weakly in $\D_a^{1,p}(\RR^n,0)$ and pointwise almost everywhere.
	\end{enumerate}
\end{lem}
\begin{proof}
	Up to a subseqeunce, we may suppose that $y_\alpha \to y$. Given an arbitrary $\epsilon>0$, observe that for all $\alpha$ large
	\[
	B\left(y_\alpha, \frac{\epsilon}{2}\right)\subset B(y, \epsilon).
	\]
	Therefore,
	\begin{align*}
		\int_{B(y, \epsilon)^\complement}\abs{\nabla\left[\tau^-_{y_\alpha, \lambda_\alpha}v\right]}^p\abs{z}^{-ap}\d{z} 
		&\le \int_{B(y_\alpha, \epsilon/2)^\complement}\abs{\nabla\left[\tau^-_{y_\alpha, \lambda_\alpha}v\right]}^p\abs{z}^{-ap}\d{z} \\
		&\le C\left(\norm{v}_{L^q\left(B(0, \epsilon/(2\lambda_\alpha))^\complement\right)} + \norm{\nabla v}_{L^p\left(B(0, \epsilon/(2\lambda_\alpha))^\complement\right)}\right)^p
	\end{align*}
	where the second inequality follows from Remark \ref{rem:rescale_mod_inv_bound_set} and $C>0$ is a constant independent of $\alpha$. Since $\lambda_\alpha\to 0$, the right-hand side tends to $0$ by the dominated convergence theorem as $\alpha\to\infty$. Therefore,
	\[
	\nabla\left[\tau^-_{y_\alpha, \lambda_\alpha}v\right] \to 0 \qquad\text{in } L^p\left(B(y, \epsilon)^\complement\right).
	\]
	Since $\epsilon > 0$ was arbitrary, a diagonal argument shows that, up to a subsequence, $\nabla \left[\tau^-_{y_\alpha, \lambda_\alpha}v\right] \to 0$ almost everywhere.

	By Lemma \ref{lem:rescale_mod_inv_bound}, the sequence $(\tau^-_{y_\alpha, \lambda_\alpha}v)$ is bounded in $\D_a^{1,p}(\RR^n,0)$. Then, utilizing also that $\nabla \left[\tau^-_{y_\alpha, \lambda_\alpha}v\right] \to 0$ pointwise almost everywhere, it follows from duality and a standard convergence result that \(\nabla \left[\tau^-_{y_\alpha, \lambda_\alpha}v\right] \rightharpoonup 0\) weakly in \(L_a^p(\RR^n, 0)\). Put otherwise, $\tau^-_{y_\alpha, \lambda_\alpha}v \rightharpoonup 0$ weakly in $\D_a^{1,p}(\RR^n,0)$. 
	
	On the other hand, since \(\tau^-_{y_\alpha, \lambda_\alpha}v\) is a bounded sequence in \(\D_a^{1,p}(\RR^n, 0)\), up to a subsequence we may assume that there is some \(w\in \D_a^{1,p}(\RR^n, 0)\) such that
	\begin{itemize}[label = \tiny$\bullet$]
		\item $\tau^-_{y_\alpha, \lambda_\alpha}v \rightharpoonup w$ weakly in $\D_a^{1,p}(\RR^n,0)$, and
		\item $\tau^-_{y_\alpha, \lambda_\alpha}v \to w$ pointwise almost everywhere.
	\end{itemize}
	Then, by uniqueness of weak limits, \(w\equiv 0\).
\end{proof}

\section{The Proof When $a<b$}

We now produce the proof of Theorem \ref{thm:main a<b}, i.e. we establish the main compactness result in the parameter case $a < b$. As previously stated, this setting can only produce weighted bubbles centered at the origin, which may also occur in the limit case $a=b$. For the sake of brevity and clarity, we decompose the proof into several distinct steps.\\

\noindent\textbf{Step 1}.
Being a Palais-Smale sequence, by Proposition \ref{prop:psBounded}, $(u_\alpha)$ is bounded in $\D^{1,p}_a(\Om , 0)$. Hence, passing to a subsequence, $u_\alpha$ converges weakly to some $v_0 \in \D^{1,p}_a(\Om, 0)$, with pointwise convergence almost everywhere on $\Om$, as $\alpha \to \infty$. Therefore, by Lemma \ref{lem:extraction}, the sequence $(u_\alpha^1)$ defined by $u_\alpha^1 := u_\alpha - v_0$ satisfies
\begin{enumerate}[label=(\roman*)]
	\item $\norm{u_\alpha^1}^p_{\D^{1,p}_a(\Om, 0)} = \norm{u_\alpha}^p_{\D^{1,p}_a(\Om, 0)} - \norm{v_0}^p_{\D^{1,p}_a(\Om, 0)} + o(1)$;
	\item $\phi(u_\alpha^1) \to c - \phi(v_0)$;
	\item $\phi^\prime(u_\alpha^1) \to 0$ in $\D^{-1,p^\prime}_a(\Om, 0)$.
\end{enumerate}
Furthermore, we find that $\phi^\prime(v_0) = 0$ and $\nabla u_\alpha \to \nabla v_0$ a.e. on $\RR^n$ as $\alpha\to \infty$. Especially, $v_0$ solves problem \eqref{eq:probOm}.\\

\noindent\textbf{Step 2}. If $u_\alpha^1 \to 0$ strongly in $L^q_b(\Om,0)$ then the proof is complete. Indeed, 
\begin{align*}
	\lim_{\alpha \to \infty} \int_{\Om}  \left( \abs{x}^{-ap}\abs{\nabla u_\alpha^1(x)}^{p}  - \abs{x}^{-bq} \abs{u_\alpha^1(x)}^{q} \right)\d{x}
	=\lim_{\alpha \to \infty} \abs{ \inner{\phi^\prime(u_\alpha^1), u_\alpha^1} } \to 0
\end{align*}
since $\phi^\prime(u_\alpha^1) \to 0$ in the dual of $\D_a^{1,p}(\Om)$ and the sequence $(u_\alpha^1)$ is bounded in $\D^{1,p}_a(\Om, 0)$. Hence, $\nabla u^1_\alpha \to 0$ strongly in $L^p_a(\Om, 0)$, i.e. $u_\alpha \to v_0$ strongly in $\D^{1,p}_a(\Om, 0)$. Thus, the conclusions of our theorem are satisfied for $k=0$ since
\[
\phi(v_0) = \lim_{\alpha \to \infty} \phi(u_\alpha).
\]
\noindent\textbf{Step 3}. If we do not satisfy the termination condition of Step 2, then, after passing to a subsequence,
\[
\inf_{\alpha \geqslant 1} \int_\Om \abs{u_\alpha^1}^q\abs{x}^{-bq}\mathrm{d}x > \delta
\]
for some $\delta > 0$. We may suppose without harm that
\begin{equation}\label{eq:delta}
	0 < \delta < \left( \frac{C_{a,b}}{2^p}\right)^{\frac{q}{q-p}}.
\end{equation}
where $C_{a,b} > 0$ is such that the CKN inequality \eqref{eq:CKNv2} holds.

Consider the family $\left\{ Q_\alpha \right\}_{\alpha \geqslant 1}$ of L\'evy concentration functions each defined by
\[
Q_\alpha(r) := \sup_{y \in \bar{\Om}} \int_{B(y,r)} \abs{u_\alpha^1}^q\abs{x}^{-bq}\d{x},
\]
where each $(u_\alpha)$ is extended by $0$ outside $\Omega$. Note that, by CKN inequality \eqref{eq:CKN}, one has $Q_\alpha(r) < \infty$ for every $r\ge 0$. By continuity of each $Q_\alpha$, we may apply the Intermediate Value Theorem to extract a smallest $\lambda_\alpha^1 > 0$ such that $Q_\alpha(\lambda_\alpha^1) = \delta$.

Finally, for each index $\alpha$, we may select $y_\alpha^1 \in \overline{\Om}$ such that
\begin{equation}\label{eq:delta a<b}
	Q_\alpha(\lambda_\alpha^1) = \sup_{y \in \bar{\Om}} \int_{B(y,\lambda_\alpha^1)} \abs{u_\alpha^1}^q\abs{x}^{-bq}\d{x} =  \int_{B(y_\alpha^1,\lambda_\alpha^1)} \abs{u_\alpha^1}^q\abs{x}^{-bq}\d{x} = \delta.
\end{equation}
Passing to a subsequence if necessary, we will assume that $\lambda_\alpha^1 \to \lambda^1 \ge 0$ and $y_\alpha^1 \to y^1 \in \overline{\Om}$ as $\alpha \to \infty$.\\

\noindent\textbf{Step 4}. We show that the sequence of points
\(
{y_\alpha^1}/{\lambda_\alpha^1}
\)	
is bounded in $\RR^n$. If not, passing to a subsequence we may assume that 
\[
\frac{\abs{y_\alpha^1}}{\lambda_\alpha^1} \to \infty, \quad \text{as } \alpha \to \infty,
\]
and that $\abs{y_\alpha^1} > 4\lambda_\alpha^1 > 0$ for each $\alpha \in \NN$. 
Consider the functions
\[
v_\alpha(x) := \tau_{y_\alpha^1, \lambda_\alpha^1} u_\alpha^1(x)
=\left( \frac{\lambda_\alpha^1}{\abs{y_\alpha^1}} \right)^b \left(\lambda_\alpha^1\right)^\gamma u_\alpha^1(\lambda_\alpha^1 x + y_\alpha^1)\eta\left( \frac{2\lambda_\alpha^1x}{\abs{y_\alpha^1}}\right)
\]
with $\tau$ defined as in equation \eqref{eq:rescale_mod}. Lemma  \ref{lem:rescale_mod_bound} then implies that $v_\alpha \in \D^{1,p}(\RR^n)$ with $v_\alpha \to 0$ strongly. Furthermore, by the Gagliardo-Nirenberg-Sobolev inequality, we also obtain $v_\alpha \to 0$ in $L^{p^\ast}(\RR^n)$. Using  that $p^\ast \ge q$, H\"older's inequality further asserts that $v_\alpha \to 0$ in $L^q_{\text{loc}}(\RR^n)$. However, Lemma \ref{lem:unit ball Lq norm} and  \eqref{eq:delta a<b} show that, as $\alpha \to \infty$,
\[
\int_{B(0,1)} \abs{v_\alpha(x)}^q\d{x} 
=  \int_{B\left( y_\alpha^1, \lambda_\alpha^1\right)} \abs{u_\alpha^1(z)}^q \abs{z}^{-bq}\d{z} + o(1)
= \delta + o(1),
\]
contradicting our deduction that $v_\alpha \to 0$ strongly in $L^q_{\text{loc}}(\RR^n)$.\\

Since \({y_\alpha^1}/{\lambda_\alpha^1}\) is bounded, passing to a subsequence, we may assume that
\[
\frac{y_\alpha^1}{\lambda_\alpha^1} \to x_0 \in \RR^n, \quad \text{as } \alpha \to \infty.
\]
We then define a sequence of functions via homogeneity:
\[
v_\alpha^1(x) = \left( \lambda_\alpha^1 \right)^\gamma u_\alpha^1(\lambda_\alpha^1(x+x_0)).
\]
It is clear that, for each $\alpha \in \NN$, $v_\alpha^1 \in \D_a^{1,p}(\Om_\alpha, -x_0)\subset \D^{1,p}_a(\RR^n,-x_0)$ where
\[
\Om_\alpha := \frac{\Om - \lambda_\alpha^1 x_0}{\lambda_\alpha^1} = \frac{\Om}{\lambda_\alpha^1} - x_0.
\]
By homogeneity, we mean that
\[
\norm{v_\alpha^1}_{\D_a^{1,p}(\RR^n, -x_0)} = \norm{u_\alpha}_{\D_a^{1,p}(\Om, 0)},\qquad
\norm{v_\alpha^1}_{L_b^q(\RR^n, -x_0)} = \norm{u_\alpha}_{L_b^q(\Om, 0)}
\] 
In particular, $(v_\alpha^1)$ is a bounded sequence in $\D_a^{1,p}(\RR^n, -x_0)$ whence we may assume without loss of generality that
\[
v_\alpha^1 \rightharpoonup v^1 \text{ weakly in } \D_a^{1,p}(\RR^n, -x_0),\quad
v_\alpha^1 \to v^1 \text{ a.e. in }\RR^n
\]
for some $v^1\in \D_a^{1,p}(\RR^n, -x_0)$. We note that, by a change of variables
\begin{equation}\label{eq:delta v a<b}
	\sup_{y \in \overline{\Om_\alpha}} \int_{B(y,1)} \abs{v_\alpha^1(x)}^q\abs{x+x_0}^{-bq}\d{x} = \sup_{w \in \overline{\Om}} \int_{B(w, \lambda_\alpha^1)} \abs{u_\alpha^1(z)}^q\abs{z}^{-bq}\d{z} = \delta
\end{equation}
where the last equality follows from equation \eqref{eq:delta a<b}. 

\noindent\textbf{Step 5}. For a given $h \in \D^{1,p}_a(\Om_\alpha, -x_0)$, define $h_\alpha \in \D^{1,p}_a(\Om, 0)$ via the homogeneous transform:
\[
h_\alpha(z) := \left(  \lambda_\alpha^1 \right)^{-\gamma}h\left( \frac{z}{\lambda_\alpha^1} - x_0\right).
\]
A change of variables shows that
\[
\inner{\phi^\prime_{x_0, \infty}(v_\alpha^1), h} = \inner{\phi^\prime(u_\alpha^1), h_\alpha}.
\]
Since $\phi^\prime(u_\alpha^1)$ is a continuous linear functional on $\D^{1,p}_a(\Om, 0)$, for each $\alpha \in \NN$, by a Riesz-type  representation result, see \eqref{eq:Riesz}, there exist functions $f_\alpha^{(1)},\dots,f_\alpha^{(n)} \in L^{p^\prime}(\Om, \abs{x}^{-ap})$ so that
\[
\inner{\phi^\prime(u_\alpha^1), h} =\sum_{i=1}^n  \int_{\Om} f^{(i)}_\alpha(x) \partial_i h(x)\abs{x}^{-ap}\d{x}, \quad \forall h \in \D^{1,p}_a(\Om, 0).
\]
Of course, $p^\prime$ denotes the H\"older conjugate exponent to $p$.
Moreover, since $\phi^\prime(u_\alpha^1) \to 0$ in the strong operator topology,
\begin{align}\label{eq:step5f}
	\sum_{i=1}^n \int_{\Om} \abs{f^{(i)}_\alpha(x)}^{p^\prime}\abs{x}^{-ap}\d{x} \to 0, \quad \text{as } \alpha \to \infty.
\end{align}
Therefore, we may write
\begin{align*}
	\inner{\phi^\prime_{x_0,\infty}(v_\alpha^1), h} 
	&= \inner{\phi^\prime(u_\alpha^1), h_\alpha}\\
	&= \sum_{i=1}^n \int_\Om f^{(i)}_\alpha(z)\partial_i h_\alpha(z)\abs{z}^{-ap}\d{z}\\
	&= \left(\lambda_\alpha^1\right)^{n - ap - (\gamma+1)}\sum_{i=1}^n \int_{\Om_\alpha} f^{(i)}_\alpha(\lambda_\alpha^1(x+x_0))\partial_i h\left( x\right)\abs{x+x_0}^{-ap}\d{x}.
\end{align*}
It follows that
\begin{equation}\label{eq:step5identity}
	\inner{\phi_{x_0,\infty}(v_\alpha^1),h} =\sum_{i=1}^n \int_{\Om_\alpha} g^{(i)}_\alpha(x)\partial_i h\left( x\right)\abs{x+x_0}^{-ap}\d{x},
\end{equation}
where we have defined
\[
g_\alpha^{(i)}(x) := \left( \lambda_\alpha^1 \right)^{\frac{n-ap}{p^\prime}} f_\alpha^{(i)}\left( \lambda_\alpha^1(x+x_0)\right).
\]
Furthermore, by another change of variables,
\begin{equation}\label{eq:step5g}
	\sum_{i=1}^n \int_{\Om_\alpha} \abs{g_\alpha^{(i)}(x)}^{p^\prime} \abs{x+x_0}^{-ap}\d{x}
	=\sum_{i=1}^n \int_\Om \abs{f_\alpha^{(i)}(z)}^{p^\prime}\abs{z}^{-ap}\d{z}
	=o(1)
\end{equation}
as $\alpha \to \infty$ because of \eqref{eq:step5f}.\\

\noindent\textbf{Step 6}. By way of contradiction, we claim $v^1 \neq 0$. Assume $v^1 = 0$ so that $v_\alpha^1 \rightharpoonup 0$ weakly in $\D^{1,p}_a(\RR^n,-x_0)$ and pointwise a.e. on $\RR^n$ as $\alpha \to \infty$. By the compactness of the embedding $\D^{1,p}_a(\RR^n,-x_0) \hookrightarrow L^p_{\text{loc}}(\RR^n, \abs{x+x_0}^{-ap})$, we may assume that $v_\alpha^1 \to 0$ strongly in $L^p_{\text{loc}}(\RR^n, \abs{x+x_0}^{-ap})$ whence $(v_\alpha^1)$ is bounded in $L^p_{\text{loc}}(\RR^n, \abs{x+x_0}^{-ap})$. By Proposition \ref{prop:domain convergence x_0}, \(\Om_\alpha\to\Om_\infty\) where 
\[
\Om_\infty := \begin{dcases}
	\RR^n & \text{if } \lambda^1 = 0,\\
	\frac{\Om}{\lambda^1} - x_0 & \text{if } \lambda^1 > 0,
\end{dcases}
\] 
and let $\mathcal{C}$ be a countable cover of $\RR^n$ such that if $\Lambda\in \mathcal{C}$ then either 
\begin{enumerate}
	\item $\Lambda = B(y, 1/2)$ for some $y\in\Om_\infty$ or
	\item $\Lambda$ is compactly contained in the complement of $\overline{\Om_\infty}$.
\end{enumerate} 
If $\Lambda\in \mathcal{C}$ is compactly contained in the complement of $\overline{\Om_\infty}$, then \(\Lambda\) and \(\Om_\alpha\) are disjoint for all \(\alpha\) large since \(\Om_\alpha\to\Om_\infty\). Therefore, $v_\alpha^1$ vanishes on $\Lambda$ for all $\alpha$ large and, in particular, $\nabla v_\alpha^1 \to 0$ in $L^p_a\left(\Lambda, -x_0\right)$.

On the other hand, suppose $\Lambda = B(y,1/2)$ for some $y \in \Om_\infty$. Since \(\Om_\alpha \to\Om_\infty\), $y \in \Om_\alpha$ for all $\alpha$ sufficiently large. For any $h \in C_c^\infty(\RR^n; [0, 1])$ with $\operatorname{supp}(h) \subset B(y,1)$, using H\"older's inequality with the CKN-inequality \eqref{eq:CKNv2} then yields

\begin{align*}
	\int_{\RR^n} &\abs{h}^p\abs{v_\alpha^1}^q \abs{x+x_0}^{-bq}\d{x}\\
	&= \int_{\operatorname{supp}(h)} \abs{v_\alpha^1 h}^p\abs{v_\alpha^1}^{q-p} \abs{x+x_0}^{-bq}\d{x}\\
	&\le \left(  \int_{\operatorname{supp}(h)} \abs{v_\alpha^1}^q\abs{x+x_0}^{-bq} \d{x}\right)^{\frac{q-p}{q}} \left(  \int_{\operatorname{supp}(h)} \abs{hv_\alpha^1}^q\abs{x+x_0}^{-bq}\d{x} \right)^{p/q}\\
	&\le C_{a,b}^{-1} \left(  \int_{\operatorname{supp}(h)} \abs{v_\alpha^1}^q\abs{x+x_0}^{-bq} \d{x}\right)^{\frac{q-p}{q}} \left(  \int_{\operatorname{supp}(h)} \abs{\nabla (hv_\alpha^1)}^p\abs{x+x_0}^{-ap}\d{x} \right).\\
\end{align*}
Then, bounding the first integral with equations \eqref{eq:delta v a<b} and \eqref{eq:delta},
\begin{equation}\label{eq:step6-1}
	\int_{\RR^n} \abs{h}^p\abs{v_\alpha^1}^q \abs{x+x_0}^{-bq}\d{x}
	\le 2^{-p}\int_{\operatorname{supp}(h)} \abs{\nabla (hv_\alpha^1)}^p\abs{x+x_0}^{-ap}\d{x}.
\end{equation}
On the other hand, since $v_\alpha^1 \to 0$ in $L^p_{\text{loc}}(\RR^n, \abs{x+x_0}^{-ap})$,
\begin{align*}
	\int_{\RR^n} \abs{\nabla (hv_\alpha^1)}^p\abs{x+x_0}^{-ap}\d{x}
	&\leq 2^{p-1}\int_{\RR^n}\abs{h}^p\abs{\nabla v_\alpha^1}^p \abs{x+x_0}^{-ap}\d{x} + o(1)\\
	&=2^{p-1}\int_{\RR^n} \abs{\nabla v_\alpha^1}^{p-2} \pinner{\nabla v_\alpha^1, \nabla \left[\abs{h}^pv_\alpha^1\right]}\abs{x+x_0}^{-ap}\d{x} + o(1)\\
	&= 2^{p-1}\left[ \inner{\phi^\prime_{x_0,\infty}(v_\alpha^1), \abs{h}^pv_\alpha^1}  + \int_{\RR^n} \abs{v_\alpha^1}^{q}\abs{h}^p\abs{x+x_0}^{-bq}\right] + o(1).
\end{align*}
Next, by Step 5 equations \eqref{eq:step5identity}-\eqref{eq:step5g}, we see that as $\alpha \to \infty$ 
\begin{align*}
	\inner{\phi^\prime_{x_0,\infty}(v_\alpha^1), \abs{h}^pv_\alpha^1} &= \sum_{i=1}^n \int_{\Om_\alpha} g_\alpha^{(i)} \partial_i\left( \abs{h}^pv_\alpha^1 \right)\abs{x+x_0}^{-ap}\d{x}
	= o(1)
\end{align*}
where we have used that $\left(\abs{h}^pv_\alpha^1\right)$ is bounded in $\D^{1,p}_a(\Om_\alpha, -x_0)$. Ergo, concatenating with \eqref{eq:step6-1},
\begin{align*}
	\int_{\RR^n} \abs{\nabla (hv_\alpha^1)}^p\abs{x+x_0}^{-ap}\d{x}
	&\le \frac{1}{2} \int_{\RR^n} \abs{\nabla (hv_\alpha^1)}^p\abs{x+x_0}^{-ap}\d{x} + o(1).
\end{align*}
We deduce that $\int_{\RR^n} \abs{\nabla (hv_\alpha^1)}^p\abs{x+x_0}^{-ap}\d{x} \to 0$. Taking $h \equiv 1$ on $\Lambda = B(y,1/2)$, we infer that
\(
\nabla v_\alpha^1 \to 0 \text{ in } L^p_a\left(\Lambda, -x_0\right).
\) 
Having shown that this convergence holds for all $\Lambda\in\mathcal{C}$, it follows that
\(
\nabla v_\alpha^1 \to 0$ in $L^p_{loc}(\RR^n, \abs{x+x_0}^{-ap})
\).
Multiplying with a cutoff function and applying the CKN-inequality, this implies that 
\(
v_\alpha^1 \to 0
\)
strongly in \(L^q_{\text{loc}}(\RR^n, \abs{x+x_0}^{-bq})\) as $\alpha \to \infty$. However, this contradicts by \eqref{eq:delta a<b} since
\begin{align*}
	\int_{B(0,2)} \abs{v_\alpha^1}^q\abs{x+x_0}^{-bq}\d{x}
	&\ge \int_{B\left( \frac{y_\alpha^1}{\lambda_\alpha^1} - x_0, 1 \right)} \abs{v_\alpha^1}^q\abs{x+x_0}^{-bq}\d{x}\\
	&= \int_{B\left(y_\alpha^1, \lambda_\alpha^1 \right)} \abs{u_\alpha^1(z)}^q\abs{z}^{-bq}\d{x}\\
	&= \delta > 0
\end{align*}
for all $\alpha \in \NN$ sufficiently large.\\

\noindent\textbf{Step 7}. We claim that $\lambda ^ 1= \lim_{\alpha \to \infty} \lambda_\alpha^1 = 0$. To this end, we argue by contradiction and assume instead that $\lambda^1 > 0$. Given any $\varphi \in C_c^\infty(\RR^n; \RR^n)$, a straightforward change of variables shows that
\begin{equation}\label{eq:step7}
	\begin{aligned}
		\int_{\RR^n} &\pinner{\nabla v_\alpha^1(x),  \varphi (x)}\abs{x+x_0}^{-ap}\d{x}\\
		&= (\lambda_\alpha^1)^{(\gamma+1) - n + ap} \int_{\RR^n} \pinner{\nabla u_\alpha^1(z), \varphi\left(  \frac{z}{\lambda_\alpha^1} - x_0 \right)}\abs{z}^{-ap}\d{z}.
	\end{aligned}
\end{equation}
As the sequence $(\lambda_\alpha^1)$ is bounded, there is a compact set $\Lambda$ such that, for each $\alpha\in\NN$, the function
\[
z \mapsto \varphi\left(  \frac{z}{\lambda_\alpha^1} - x_0 \right)
\]
is supported on $\Lambda$. Following this, since $\nabla u_\alpha^1 \to 0$ pointwise almost everywhere on $\RR^n$ and is bounded in $L^p_a(\RR^n, 0)$, standard convergence results imply that
\[
\int_{\varLambda} \abs{\nabla u_\alpha^1(z)}\abs{z}^{-ap}\d{z} \to 0, \quad \text{as } \alpha \to \infty.
\]
Using this and that $\lambda_\alpha^1\to\lambda_1 > 0$, it follows from \eqref{eq:step7} that
\[
\int_{\RR^n} \pinner{\nabla v_\alpha^1(x), \varphi (x)} \abs{x+x_0}^{-ap}\d{x}\to 0, \quad \text{as } \alpha \to \infty.
\]
Since smooth functions of compact support are dense in $L^{p^\prime}(\RR^n, \abs{x+x_0}^{-ap})$, it readily follows that $\nabla v_\alpha^1 \rightharpoonup 0$ weakly in $L^p_a(\RR^n, -x_0)$. Put otherwise, $v_\alpha^1 \rightharpoonup 0$ in $\D^{1,p}_a(\RR^n,-x_0)$ which directly contradicts Step 6.\\

\noindent\textbf{Step 8}.  Since $v^1 \in \D^{1,p}_a(\RR^n,-x_0)$ and $\Om_\alpha \to \RR^n$, there exists a sequence functions $\psi_\alpha \in C_c^\infty(\Om_\alpha)$ such that $\psi_\alpha \to v^1$ strongly in $\D^{1,p}_a(\RR^n,-x_0)$ as $\alpha \to \infty$. 

Consider the modified sequence given by a homogeneous transform:
\[
\tilde{\psi}_\alpha(z) := \left(\lambda_\alpha^1\right)^{-\gamma}\psi_\alpha\left(\frac{z}{\lambda_\alpha^1}-x_0\right).
\]
Since $\psi_\alpha \to v^1$ strongly in $\D^{1,p}_a(\RR^n,-x_0)$, a change of variables (i.e. homogeneity) shows that
\[
\tilde{\psi}_\alpha (\cdot)- \left( \lambda_\alpha^1 \right)^{-\gamma} v^1\left( \frac{\cdot}{\lambda_\alpha^1} - x_0 \right) \to 0 \quad \text{in } \D^{1,p}_a(\RR^n, 0).
\]
Therefore, after an application of Lemma \ref{lem:bubs} (with \(v=v^1\)), we deduce that
\begin{enumerate}[label=(\roman*)]
	\item $\nabla \tilde{\psi}_\alpha \to 0$ pointwise almost everywhere on $\RR^n$;
	\item $\tilde{\psi}_\alpha \rightharpoonup 0$ weakly in $\D^{1,p}_a(\Om, 0)$ and pointwise almost everywhere $\alpha \to \infty$.\\
\end{enumerate}

\noindent\textbf{Step 9}. We outline the iteration process. Applying Proposition \ref{prop:iteration msc}, we see that $v^1$ solves a translated version of \eqref{eq:probLim} and the sequence
\[
u_\alpha^2(x) := u_\alpha^1(x) - \left(\lambda_\alpha^{(1)}\right)^{-\gamma} v^1\left( \frac{x}{\lambda_\alpha^{(1)}} -x_0^{1}\right),
\]
for $\lambda_\alpha^{(1)} := \lambda_\alpha^1$, satisfies
\[
\norm{u_n^2}_{\D^{1,p}_a(\RR^n, 0)}^p 
= \norm{u_\alpha}^p_{\D^{1,p}_a(\Om, 0)}
- \norm{v_0}^p_{{\D^{1,p}_a(\Om, 0)}}
- \norm{v^1}_{\D^{1,p}(\RR^n,-x_0^{(1)})}^p + o(1)
\]
and
\begin{align*}
	\begin{cases}
		\phi_{0,\infty}(u_\alpha^2) \to c - \phi(v_0) - \phi_{x_0^{(1)},\infty}(v^1),\\
		\phi^\prime_{0,\infty}(u_\alpha^2) \to 0  \text{ in } \D^{-1,p^\prime}_a(\Om, 0).
	\end{cases}
\end{align*}
Next, we consider the auxiliary sequence
\[
\tilde{u}_\alpha^2(x) := u_\alpha^1(x) - \tilde{\psi}_\alpha(x)
\]
which is bounded in $\D^{1,p}_a(\Om, 0)$.  Because $\psi_\alpha \to v^1$ strongly in $\D^{1,p}(\RR^n,-x_0^{1})$, 
\[
\begin{cases}
	\phi\left(\tilde{u}^2_\alpha\right) \to  c - \phi(v_0) - \phi_{x_0^{(1)},\infty}(v^1),\\
	\phi^\prime(\tilde{u}_\alpha^2) \to 0  \text{ in } \D^{-1,p^\prime}_a(\Om, 0).
\end{cases}
\]

\noindent We may therefore apply Steps 2-9 to this new sequence $\left(\tilde{u}_\alpha^2\right)$, at each stage removing another solution $v^j$. \\

\noindent\textbf{Step 10}. For each index $j$, a suitable translation of $v^j$ solves \eqref{eq:probLim}. Now, recalling Lemma \ref{lem:energy}, note that Palais-Smale sequences for problem \eqref{eq:probOm} have non-negative limiting energy. Moreover, by Proposition \ref{prop:energy}, the energies of non-trivial critical points for the weighted limiting problem \eqref{eq:probLim} are bounded below by a uniform positive constant. Thus, this procedure can only happen finitely many times
before we reach the termination condition in Step 2.

\section{The Proof When $a=b$}

We now begin the proof of Theorem \ref{thm:main a=b}, which contains the entire proof of Theorem \ref{thm:main a<b} as a sub-argument. Naturally, the proof begins in the same way (first extracting a solution to \eqref{eq:probOm} and recovering another $(PS)$-sequence by removing said solution).\\

\noindent\textbf{Steps 1-2-3}. Steps 1,2 are identical to their respective counterparts from the proof in the $a<b$ case and we modify only the upper bound on $\delta >0$ in step 3. 

More precisely, as in step 1, we deduce that there is some $v_0 \in \D^{1,p}_a(\Om, 0)$ such that, with $u_\alpha^1 := u_\alpha - v_0$, up to a subsequence,
\begin{enumerate}[label=(\roman*)]
	\item $u_\alpha \rightharpoonup v_0$ weakly in $\D_a^{1,p}(\Om, 0)$ and pointwise almost everywhere;
	\item $\norm{u_\alpha^1}^p_{\D^{1,p}_a(\Om, 0)} = \norm{u_\alpha}^p_{\D^{1,p}_a(\Om, 0)} - \norm{v_0}^p_{\D^{1,p}_a(\Om, 0)} + o(1)$;
	\item $\phi(u_\alpha^1) \to c - \phi(v_0)$;
	\item $\phi^\prime(u_\alpha^1) \to 0$ in $\D^{-1,p^\prime}_a(\Om, 0)$;
	\item  $\nabla u_\alpha \to \nabla v_0$ and pointwise almost everywhere;
\end{enumerate}
as $\alpha \to\infty$. Moreover $\phi^\prime(v_0) = 0$ so $v_0$ solves problem \eqref{eq:probOm}. Then, as in step 2, the proof terminates if $u_\alpha^1 \to 0$ strongly in $L_b^q(\Om, 0)$. Thus, we proceed to step 3 assuming that $u_\alpha^1$ does not converge to $0$ strongly in $L_b^q(\Om, 0)$. In the latter case, up to a subsequence, we may suppose that for each index $\alpha$ there holds
\begin{equation}\label{eq:levy a=b}
	\sup_{y \in \bar{\Om}} \int_{B(y,\lambda_\alpha^1)} \abs{u_\alpha^1}^q\abs{x}^{-bq}\d{x} =  \int_{B(y_\alpha^1,\lambda_\alpha^1)} \abs{u_\alpha^1}^q\abs{x}^{-bq}\d{x} = \delta.
\end{equation}
for some $0 < \lambda_\alpha^1 \le \operatorname{diam}(\Om)$ and $y_\alpha^1 \in \overline{\Omega}$. Here, 
\begin{equation}\label{eq:delta a=b}
	0 < \delta < \frac{1}{C_\tau^q}\left( \frac{C_{a,b}}{2^p}\right)^{\frac{q}{q-p}},
\end{equation}
where $C_{a,b} > 0$ is such that the CKN inequality \eqref{eq:CKNv2} holds and $C_\tau>0$ is the constant from Lemma \ref{lem:rescale_mod_bound}. Finally, passing to another subsequence if necessary, we further assume that $\lambda_\alpha^1 \to \lambda^1 \ge 0$ and $y_\alpha^1 \to y^1 \in \overline{\Om}$ as $\alpha \to \infty$.\\

\noindent\textbf{Step 4}. The following dichotomy holds: either $y_\alpha^1/\lambda_\alpha^1$ is bounded or unbounded. In the former case, we pass to a subsequence so that
\[
\frac{{y_\alpha^1}}{\lambda_\alpha^1} \to x_0 \in \RR^n
\]
and proceed exactly as in the proof with $a<b$ (i.e. we  follow Steps 5-10 from the previous section). Otherwise, we may pass to a subsequence in such a way that
\[
\frac{\abs{y_\alpha^1}}{\lambda_\alpha^1} \to \infty \quad \text{and} \quad \abs{y_\alpha^1} > 4\lambda_\alpha^1.
\]
In this case, it is automatic that $\lambda_\alpha^1 \to 0$ as $\alpha \to \infty$. Subsequently, we may define
\[
v_\alpha^1 := \tau_{y_\alpha^1, \lambda_\alpha^1}u_\alpha^1 \in \D^{1,p}(\RR^n).
\]
Then, $v_\alpha^1\in \D^{1,p}(\Om_\alpha) $ in $ \D^{1,p}(\RR^n)$ where
\[
\Om_\alpha := \frac{\Om - y_\alpha^1}{\lambda_\alpha^1}
\]
and, by Lemma \ref{lem:rescale_mod_bound}, the sequence $(v_\alpha^1)$ is bounded in $\D^{1,p}(\RR^n)$. Therefore, there exists $v^1\in \D^{1,p}(\RR^n)$ such that
\[
v_\alpha^1\rightharpoonup v^1 \text{ weakly in }\D^{1,p}(\RR^n),\quad
v_\alpha^1 \to v^1\text{ a.e. in }\RR^n.
\]
Furthermore, by Lemma \ref{lem:unit ball Lq norm} and \eqref{eq:levy a=b},
\[
\int_{B(0,1)}\abs{v_\alpha^1}^q\d{x} = \delta + o(1)\qquad\text{as}\quad\alpha\to\infty.
\]
On the other hand, by Remark \ref{rem:rescale_mod_bound_set} and \eqref{eq:levy a=b}, we see that
\[
\sup_{y\in \overline{\Om_\alpha}}\int_{B(y,1)}\abs{v_\alpha^1}^q\d{x} 
\le C_\tau^q\sup_{y\in\overline{\Om}}\int_{B(y, \lambda_\alpha^1)}\abs{u_\alpha^1}^q\abs{z}^{-bq}\d{x}
= C_\tau^q\delta
\]
as $\alpha\to \infty$. Combining these last two equations, we obtain
\begin{equation}\label{eq:delta bound}
	\delta + o(1) 
	\le \sup_{y\in \overline{\Om_\alpha}}\int_{B(y,1)}\abs{v_\alpha^1}^q\d{x} 
	\le C_\tau^q\delta 
	\le \left(\frac{C_{a,b}}{2^p}\right)^{\frac{q}{q-p}}
\end{equation}
where, for the last step, we used \eqref{eq:delta a=b}.\\

\noindent\textbf{Step 5}. Fix a compact set $K\subseteq \Om$ and let $(h_\alpha)$ in $ \D^{1,p}(\RR^n)$ be a bounded sequence with $h_\alpha\in \D^{1,p}(\Om_\alpha)$ and $h_\alpha$ supported in $K$ for each index $\alpha \in \NN$. It follows from Lemma \ref{lem:transform energy} and the CKN-inequality \eqref{eq:CKN} that
\begin{equation}\label{eq:2.7}
	\inner{\phi^\prime_\infty(v_\alpha^1), h_\alpha} = \inner{\phi^\prime(u_\alpha^1), \tau^-_{y_\alpha^1, \lambda_\alpha^1}h_\alpha} + o\left(1\right)
\end{equation}
as $\alpha \to\infty$. By a Riesz-type  representation result, i.e. characterization \eqref{eq:Riesz}, we may find, for each $\alpha$, functions $f_\alpha^{(1)}, \dots, f_\alpha^{(n)}$ in $L_a^{p^\prime}(\Om)$ such that
\[
\inner{\phi^\prime(u_\alpha^1), h} = \sum_{i=1}^n\int_\Om f_\alpha^{(i)}(z)\partial_ih(z)\abs{z}^{-ap}\d{z},\qquad \forall h\in \D_a^{1,p}(\Om,0).
\]
Moreover, since $\phi^\prime(u_\alpha^1) \to 0$ strongly in the dual of $\D_a^{1,p}(\Om,0)$,
\begin{equation}\label{eq:f_i norms}
	\sum_{i=1}^n\int_\Om\abs{f_\alpha^{(i)}(z)}^{p^\prime}\abs{z}^{-ap}\d{z} \to 0,\quad\text{as }\alpha\to\infty.
\end{equation}
Combining \eqref{eq:2.7} with the aforementioned functional representation, we may write
\begin{align*}
	\inner{\phi^\prime_\infty(v^1_\alpha), h_\alpha} 
	&= \sum_{i=1}^n\int_\Om f_\alpha^{(i)}(z)\partial_i \left[\tau^-_{y_\alpha^1, \lambda_\alpha^1}h_\alpha\right](z)\abs{z}^{-ap}\d{z} + o(1)\\
	&= \sum_{i=1}^n\int_\Om f_\alpha^{(i)}(z)\partial_i \left[\tau^-_{y_\alpha^1, \lambda_\alpha^1}h_\alpha\right](z)\left[\abs{z}^{-ap} - \abs{y_\alpha^1}^{-ap}\right]\d{z} \\
	&\qquad + \sum_{i=1}^n\int_\Om f_\alpha^{(i)}(z)\partial_i \left[\tau^-_{y_\alpha^1, \lambda_\alpha^1}h_\alpha\right](z)\abs{y_\alpha^1}^{-ap}\d{z} + o(1).
\end{align*}
Since $h_\alpha$ is supported in $K$, each integrand is supported in $\lambda_\alpha^1 K + y_\alpha^1$; then, using Lemma \ref{lem:cutoff tools}, we see that for all $\alpha$ large and each index $i=1, \dots, n$
\begin{align*}
	&\abs{\int_\Om f_\alpha^{(i)}(z)\partial_i \left[\tau^-_{y_\alpha^1, \lambda_\alpha^1}h_\alpha\right](z)\left[\abs{z}^{-ap} - \abs{y_\alpha^1}^{-ap}\right]\d{z}}\\
	&\le C_K\frac{\lambda_\alpha^1}{\abs{y_\alpha^1}} \int_\Om \abs{f_\alpha^{(i)}(z)}\abs{\partial_i \left[\tau^-_{y_\alpha^1, \lambda_\alpha^1}h_\alpha\right](z)}\abs{z}^{-ap}\d{z}\\
	&\le C_K\frac{\lambda_\alpha^1}{\abs{y_\alpha^1}}  \norm{f_\alpha^{(i)}}_{L_a^{p^\prime}(\Om,0)}\norm{h_\alpha}_{\D^{1,p}(\RR^n)} = o(1)
\end{align*}
where we have used H\"older's inequality, Lemma \ref{lem:rescale_mod_inv_bound} and equation \eqref{eq:f_i norms}. Returning to our earlier expression, we have
\begin{align*}
	\inner{\phi^\prime_\infty(v^1_\alpha), h_\alpha} 
	&=\sum_{i=1}^n\int_\Om f_\alpha^{(i)}(z)\partial_i \left[\tau^-_{y_\alpha^1, \lambda_\alpha^1}h_\alpha\right](z)\abs{y_\alpha^1}^{-ap}\d{z} + o(1)\\
	&= \left(\frac{\lambda_\alpha^1}{\abs{y_\alpha^1}}\right)^{ap-2b}\left(\lambda_\alpha^1\right)^{n-2\gamma-1-ap}\int_{\Om_\alpha} \left[\tau_{y_\alpha^1, \lambda_\alpha^1} f_\alpha^{(i)}\right] \partial_i h_\alpha\d{x} + o(1)
\end{align*}
Here, we have used that, since each $h_\alpha$ is supported in $K$, the cut-off $\eta$ in the transform reduces to $1$ on the integrand's support for all $\alpha$ large. Hence,
\begin{equation}\label{eq:step5mid}
	\inner{\phi^\prime_\infty(v^1_\alpha), h_\alpha} 
	=\sum_{i=1}^n\int_{\RR^n} g_\alpha^{(i)}\partial_i h_\alpha\d{x} + o(1)
\end{equation}
where, extending each $f_\alpha^{(i)}$ by 0 outside of $\Om$,
\[
g_\alpha^{(i)} := \left(\frac{\lambda_\alpha^1}{\abs{y_\alpha^1}}\right)^{ap-2b}\left(\lambda_\alpha^1\right)^{n-2\gamma-1-ap}\tau_{y_\alpha^1, \lambda_\alpha^1}f_\alpha^{(i)}
\]
Now, bounding as in Lemma \ref{lem:rescale_mod_bound}, we see that
\begin{align*}
	\int_{\RR^n}\abs{g_\alpha^{(i)}(x)}^{p^\prime}\d{x}
	\le C\int_{\RR^n}\abs{f_\alpha^{(i)}(z)}^{p^\prime}\abs{z}^{-ap}\d{z} = o(1)
\end{align*}
by \eqref{eq:f_i norms}. Therefore, applying H\"older's inequality to the right-hand side of equation \eqref{eq:step5mid} we conclude
\[
\inner{\phi^\prime_\infty(v^1_\alpha), h_\alpha} = o(1).
\] 

\noindent\textbf{Step 6}. We claim that $v^1\ne 0$. By contradiction, suppose that $v^1=0$. Then, $v_\alpha^1 \rightharpoonup 0$ weakly in $\D^{1,p}(\RR^n)$ and $v_\alpha^1 \to 0$ pointwise almost everywhere. Furthermore, by the compactness of the embedding $\D^{1,p}_a(\RR^n,-x_0) \hookrightarrow L^p_{\text{loc}}(\RR^n, \abs{x+x_0}^{-ap})$, we may assume without loss of generality that $v_\alpha^1 \to 0$ in $L^p_{\text{loc}}(\RR^n)$.

Since $\lambda_\alpha^1\to 0$, passing to another subsequence if necessary, it follows from Proposition \ref{prop:domain convergence} that there is a set $\Om_\infty\subseteq \RR^n$, which is either $\RR^n$ or a half-space, such that
\begin{enumerate}
	\item Any compact subset of $\Om_\infty$ is contained in $\Om_\alpha$ for all $\alpha$ large.
	\item Any compact subset of $\left(\overline{\Om_\infty}\right)^\complement$ is contained in $\left(\overline{\Om_\alpha}\right)^\complement\subseteq\Om_\alpha^\complement$ for all $\alpha$ large.
\end{enumerate}
We may cover $\RR^n$ by a countable collection $\mathcal{C}$ of such that if $\Lambda\in \mathcal{C}$ then either 
\begin{enumerate}
	\item $\Lambda = B(y, 1/2)$ for some $y\in\Om_\infty$ or
	\item $\Lambda$ is compactly contained in the complement of $\overline{\Om_\infty}$.
\end{enumerate} 
Pick an arbitrary $\Lambda$ in the collection $\mathcal{C}$. First, if $\Lambda$ is  compactly contained in the complement of $\overline{\Om_\infty}$, then $\Lambda$ is disjoint from $\overline{\Om_\alpha}$ for all $\alpha$ large. Since $v_\alpha^1$ is supported in $\Om_\alpha$, it is clear that $\nabla v_\alpha^1 \to 0$ in $L^p(\Lambda)$.

On the other hand, if $\Lambda = B(y, 1/2)$ for some $y\in \Om_\infty$ then, for all $\alpha$ large, $y\in \Om_\alpha$. Let $h\in C_c^\infty(\RR^n)$ be such that
\[
\begin{cases}
	0\le h \le 1\\
	h = 1 &\text{in } \Lambda = B(y, 1/2)\\
	h=0 &\text{outside } B(y, 1).
\end{cases}
\]
By H\"older's inequality with exponents $q/p$ and conjugate $q/(q-p) = n/p$ followed by the CKN inequality \eqref{eq:CKN},
\begin{align*}
	\int_{\RR^n}\abs{h}^p\abs{v_\alpha^1}^{q}\d{x}
	&=\int_{\operatorname{supp}(h)}\left(\abs{h v_\alpha^1}\right)^p\abs{v_\alpha^1}^{q-p}\d{x}\\
	&\le \left(\int_{\operatorname{supp}(h)}\abs{h v_\alpha^1}^q\d{x}\right)^{p/q}\left(\int_{\operatorname{supp}(h)}\abs{v_\alpha^1}^q\d{x}\right)^{p/n}\\
	&\le C_{a,b}^{-1}\left(\int_{\operatorname{supp}(h)}\abs{\nabla \left(h v_\alpha^1\right)}^p\d{x}\right)\left(\int_{B(y, 1)}\abs{v_\alpha^1}^q\d{x}\right)^{p/n}.
\end{align*}
Then, using equation \eqref{eq:delta bound}, we deduce that for all $\alpha$ large (so that $y\in \Om_\alpha$)
\begin{equation}\label{eq:step 6 bound 1}
	\begin{aligned}
		\int_{\RR^n}\abs{h}^p\abs{v_\alpha^1}^{q}\d{x} 
		&\le \left(\left(\frac{C_{a,b}}{2^p}\right)^{\frac{q}{q-p}}\right)^{p/n} C_{a,b}^{-1}\int_{\RR^n}\abs{\nabla \left(h v_\alpha^1\right)}^p\d{x}\\
		&= 2^{-p}\int_{\RR^n}\abs{\nabla \left(h v_\alpha^1\right)}^p\d{x}.
	\end{aligned}
\end{equation}
On the other hand, using that $v_\alpha^1 \to 0$ in $L_{loc}^p(\RR^n)$, we obtain
\begin{align*}
	\int_{\RR^n}\abs{\nabla \left(h v_\alpha^1\right)}^p\d{x}
	&\le 2^{p-1}\int_{\RR^n}\abs{h}^p\abs{\nabla v_\alpha^1}^p\d{x} + o(1)\\
	&=2^{p-1}\int_{\RR^n}\abs{\nabla v_\alpha^1}^{p-2}\pinner{\nabla v_\alpha^1, \nabla \left[\abs{h}^p v_\alpha^1\right]}\d{x} + o(1)\\
	&= 2^{p-1}\left[\inner{\phi_\infty^\prime(v_\alpha^1), \abs{h}^pv_\alpha^1} + \int_{\RR^n}\abs{v_\alpha^1}^q\abs{h}^p\d{x}\right] + o(1)
\end{align*}
as $\alpha \to\infty$. Then, by step 5, as $\alpha \to\infty$,
\begin{align*}
	\int_{\RR^n}\abs{\nabla \left(h v_\alpha^1\right)}^p\d{x}
	&= 2^{p-1} \int_{\RR^n}\abs{v_\alpha^1}^q\abs{h}^p\d{x} + o(1)\\
	&\le \frac{1}{2}\int_{\RR^n}\abs{\nabla \left(h v_\alpha^1\right)}^p\d{x} + o(1)
\end{align*}
where we have used equation \eqref{eq:step 6 bound 1}. Ergo, $\int_{\RR^n}\abs{\nabla \left(h v_\alpha^1\right)}^p\d{x} \to 0$ which implies that $\nabla v_\alpha^1 \to 0$ in $L^p(\Lambda)$.

Having concluded that $\nabla v_\alpha^1 \to 0$ in $L^p(\Lambda)$ for any $\Lambda\in\mathcal{C}$, we deduce that $\nabla v_\alpha^1 \to 0$ in $L^p_{loc}(\RR^n)$. Multiplying with a cutoff function and applying the CKN inequality \eqref{eq:CKN}, this clearly implies that $v_\alpha^1 \to 0$ in $L_{loc}^q(\RR^n)$; this contradicts equation \eqref{eq:delta bound}.\\

\noindent\textbf{Step 7}. In contrast with the previous case \(a<b\), it is already known that $\lambda^1 =0$. Consequently, Step 7 is superfluous here.\\

\noindent\textbf{Step 8}. Since $\Om_\alpha \to\Om_\infty$ and $v^1\in \D^{1,p}(\Om_\infty)$, there exist a sequence of functions $(\psi_\alpha)$ in $ C_c^\infty(\Om_\infty)$ with $\psi_\alpha \in C_c^\infty(\Om_\alpha)$ for each $\alpha$ such that $\psi_\alpha \to v^1$ in $\D^{1,p}(\Om_\infty)\subseteq \D^{1,p}(\RR^n)$.

By Lemma \ref{lem:rescale_mod_inv_bound}, the sequence given by $\tilde{\psi}_\alpha :=\tau^-_{y_\alpha^1, \lambda_\alpha^1}\psi_\alpha \in C_c^\infty(\Om)$ satisfies 
\[
\norm{\tilde{\psi}_\alpha - \tau^-_{y_\alpha^1, \lambda_\alpha^1}v^1}_{\D_a^{1,p}(\RR^n,0)} 
= \norm{\tau^-_{y_\alpha^1, \lambda_\alpha^1}\left(\psi_\alpha - v^1\right)}_{\D_a^{1,p}(\RR^n,0)}
\le C\norm{\psi_\alpha -v^1}_{\D^{1,p}(\RR^n)}
\]
for some constant $C>0$. Then, since $\psi_\alpha \to v^1$ in $\D^{1,p}(\RR^n)$,
\begin{equation}\label{eq:4}
	\tilde{\psi}_\alpha - \tau^-_{y_\alpha^1, \lambda_\alpha^1}v^1 \to 0 \quad\text{in }\D_a^{1,p}(\RR^n,0)
\end{equation}
Recalling Lemma \ref{lem:bubbling}, it readily follows that, up to a subsequence,

\begin{enumerate}
	\item $\nabla\tilde{\psi}_\alpha\to 0$ pointwise almost everywhere;
	\item The sequence $\tilde{\psi}_\alpha \rightharpoonup 0$ weakly in $\D_a^{1,p}(\Om,0)$ and pointwise almost everywhere.\\
\end{enumerate}

\noindent\textbf{Step 9}. We now show how to iterate to the next stage.  Define
\[
u_\alpha^2 := u_\alpha^1 - \tau^-_{y_\alpha^1, \lambda_\alpha^1}v^1
\]
By Proposition \ref{prop:iteration}, $\phi^\prime_\infty(v^1) = 0$ in $\Om_\infty$ and, as $\alpha\to\infty$, we have
\[
\norm{u_\alpha^2}_{\D_a^{1,p}(\RR^n,0)}^p = \norm{u_\alpha}^p_{\D^{1,p}_a(\Om, 0)} - \norm{v_0}^p_{\D^{1,p}_a(\Om, 0)} - \norm{v^1}^p_{\D^{1,p}(\RR^n)} + o(1)
\]
and
\[
\begin{cases}
	\phi_{0, \infty}(u_\alpha^2) \to c - \phi(v_0) - \phi_\infty(v^1)\\
	\phi_{0, \infty}^\prime(u_\alpha^2) \to 0 &\text{strongly in the dual of } \D_a^{1,p}(\Om,0).
\end{cases}
\]
Next, consider the auxillary sequence
\[
\tilde{u}_\alpha^2 := u_\alpha^1 - \tilde{\psi}_\alpha
\]
which is bounded in $\D_a^{1,p}(\Om,0)$. Since, by Step 8,
\[
u_\alpha^2 - \tilde{u}_\alpha^2 = \tilde{\psi}_\alpha - \tau^-_{y_\alpha^1, \lambda_\alpha^1}v^1 \to 0\quad\text{strongly in }\D_a^{1,p}(\RR^n,0)
\]
we see that
\[
\begin{cases}
	\phi(\tilde{u}_\alpha^2) \to c - \phi(v_0) - \phi_\infty(v^1)\\
	\phi^\prime(\tilde{u}_\alpha^2) \to 0 &\text{strongly in the dual of } \D_a^{1,p}(\Om,0).
\end{cases}
\]

\noindent We may therefore apply steps 2-9 to this new sequence $(\tilde{u}_\alpha^2)$, at each stage removing another solution $v^j$.\\

\noindent\textbf{Step 10}. At the $j$\textsuperscript{th} iteration stage, we had one of the following cases.
\begin{enumerate}
	\item If \(y_\alpha^{j}/\lambda_\alpha^j\) is bounded then, as in the \(a<b\) case, a suitable translation of \(v^j\) solves \eqref{eq:probLim}.
	\item If the ratio \(y_\alpha^{j}/\lambda_\alpha^j\) is unbounded, then \(\phi_\infty^\prime(v^j) = 0\) in a set \(\Om_\infty\). Here, either
	\begin{itemize}[label = \tiny\(\bullet\)]
		\item \(\Om_\infty = \RR^n\) and \(v^j\) solves \eqref{eq:probLimWeightless}.
		\item $\Om_\infty$ is a half-space, so a translated rotation of $v^j$ solves the limiting problem \eqref{eq:probLimWeightlessHalf} in $\HH = \set{x\in \RR^n : x_n > 0}$.
	\end{itemize}
\end{enumerate}
The procedure terminates because, by Lemma \ref{lem:energy}, Palais-Smale sequences necessarily possess non-negative limiting energy and, by Proposition \ref{prop:energy}, at each iteration stage, we remove a solution whose energy is uniformly strictly positive.

\bibliographystyle{plain}
\bibliography{../mybib}

\end{document}